\documentclass[final,3p,times]{elsarticle}
\usepackage{amssymb, amsmath, amsthm, mathtools, booktabs, caption, multirow, braket, pgfplots,subfig, bm, textcomp}
\pgfplotsset{/pgf/number format/use comma,compat=newest}

\usepackage{etex}
\usepackage{epstopdf}
\usepackage[all]{xy}
\usepackage{tikz-cd}

\usepackage{siunitx}
\usepackage[normalem]{ulem}
\usepackage{graphicx,subfig}

\newtheorem{theorem}{Theorem}[section]
\newtheorem{lemma}[theorem]{Lemma}
\newtheorem{proposition}[theorem]{Proposition}

\newtheorem{definition}[theorem]{Definition}
\newtheorem{example}[theorem]{Example}

\newtheorem{remark}[theorem]{Remark}

\numberwithin{equation}{section}

\usepackage{tikz}
\usetikzlibrary{calc}
\newcommand{\tikzmark}[1]{\tikz[overlay,remember picture] \node (#1) {};}
\newcommand{\DrawBox}[1][]{%
    \tikz[overlay,remember picture]{
    \draw[blue,#1]
      ($(left)+(-0.6em,0.9em)$) rectangle
      ($(right)+(0.3em,-0.3em)$);}
}

\newcommand\NN{\mathbb{{N}}}
\newcommand\QQ{\mathbb{{Q}}}
\newcommand\RR{\mathbb{{R}}}
\newcommand\CC{\mathbb{{C}}}
\newcommand\ZZ{\mathbb{{Z}}}

\DeclarePairedDelimiter{\abs}{\lvert}{\rvert}
\DeclarePairedDelimiter{\norm}{\lVert}{\rVert}

\newcommand{\smoother}[3]{
V_{#1, \text{#2}}^{#3}
}

\begin{document}

\title{Anisotropic, interpolatory subdivision and multigrid}

\author[]{M. Charina, M. Donatelli, L. Romani, V. Turati}

\begin{abstract}
In this paper, we present a family of multivariate grid transfer operators appropriate for anisotropic 
multigrid methods. Our grid transfer operators are derived from a new family of anisotropic interpolatory 
subdivision schemes. We study the minimality, polynomial reproduction and convergence properties of these 
interpolatory schemes and link their properties to the convergence and optimality of the corresponding
multigrid methods.  We compare the performance of our interpolarory grid transfer operators with the 
ones derived from a family of corresponding approximating subdivision schemes.
\end{abstract}

\maketitle

\section{Introduction}

In this paper, continuing our work in \cite{Charina2016multigrid}, we present a new family of subdivision based grid transfer operators whose approximation properties ensure the convergence and optimality of the corresponding anisotropic multigrid methods. 
Anisotropic elliptic problems arise when the diffusion is not uniform in every direction or when a standard finite difference discretization is applied on a stretched grid. In such case, classical multigrid methods are no longer reliable and different strategies must be implemented. Two common strategies are semi-coarsening and fine smoothers \cite{mgm-book}. In this paper we focus on semi-coarsening.
For simplicity of presentation, we study only the bi-variate case. 
In this case, the difference between the two coordinate directions is encoded in the family of dilation matrices 
\begin{equation} \label{eq:dilation}
M = 
\begin{pmatrix}
2 & 0 \\
0 & m
\end{pmatrix}, \qquad m \in \NN,\qquad \text{ gcd } (2,m) = 1.
\end{equation}
The corresponding anisotropic subdivision schemes, Sections \ref{sec:anisotropic_interp} and \ref{sec:anisotropic_pseudo}, allow for an appropriate multilevel reduction (via the grid transfer operators) of the size of the original system of equations  
\begin{equation} \label{eq:sys}
A_n \mathbf{x} = \mathbf{b}_n, \qquad A_n \in \CC^{n \times n}, \qquad \mathbf{b}_n \in \CC^n, \qquad n \in \NN,
\end{equation}
ensuring the linear computational cost of the iterative solver called multigrid.  Multigrid method \cite{Brandt1982guide, Hackbusch1985multi, RS1987, Wesseling1987linear} are efficient solvers for large ill-conditioned systems of equations with symmetric and positive definite system matrices $A_n$. It consists of two main steps: the smoother and the coarse grid correction, iterated until the remaining small linear system of equations is solved exactly. The smoother is a simple iterative solver such as Gauss-Seidel or weighted Jacobi, which is slowly convergent due to the ill-conditioning of the system matrix.
The coarse grid correction step is a standard error reduction step performed at a coarser grid. The projection of the problem onto a coarser grid and the lifting of the error correction term to the finer grid is done, in our case, via the grid transfer operators based on anisotropic subdivision schemes. 
The analysis uses the tools introduced in \cite{Charina2016multigrid}, where subdivision schemes and multigrid methods are related by means of symbols used in multigrid methods for Toeplitz matrices \cite{V-cycle_opt, BIT2014, NLL2010, Huckle2002}. At the best of our knowledge, the only paper that investigates anisotropic multigrid methods for block Toeplitz Toeplitz Block (BTTB) system matrices is \cite{Fischer2006}. 

We propose two types of families of anisotropic subdivision schemes: approximating and interpolating. 
Our results and numerical experiments show that both families lead to efficient grid transfer operators. Nevertheless, the computational cost of the multigrid based on interpolatory grid transfer operators is minimal due to the fewer non-zero coefficients in the corresponding subdivision rules. Indeed, our interpolatory subdivision schemes are constructed to be optimal in terms of the size of the support versus their polynomial reproduction properties. Similar constructions in the case of $M = 2 I_2$ are done in \cite{Han_Jia_1998optimal}, but
are not applicable for anisotropic multigrid. 
To study the dependence of the efficiency of multigrid on the reproduction/generation properties of subdivision, we also define a family of approximating schemes. Our construction resembles the one given in \cite{Conti2017pseudo} for the family of bi-variate pseudo-splines with dilation $M=2 I_2$. Our goal, for compatibility of our numerical experiments with approximating and interpolating grid transfer operators, is to define approximating schemes that have the same support as the interpolating ones and matching polynomial generation properties. We do not claim to have constructed a new family of anisotropic pseudo-splines.

Our paper is organized as follows. In Section \ref{sec:intro}, for convenience of the readers from both multigrid and subdivision communities, we first  give a brief introduction to multigrid and, afterwards, we recall some known facts about subdivision that are crucial for our analysis. 
Section \ref{sec:anisotropic_interp} is devoted to the construction and analysis of a family of anisotropic interpolatory subdivision schemes with dilation matrices in \eqref{eq:dilation}.  For suitable choices of $m$ in \eqref{eq:dilation} see numerical examples in Section \ref{sec:examples}. 
In Section \ref{sec:anisotropic_pseudo}, we present a comparable family of approximating anisotropic subdivision schemes and study their polynomial reproduction and generation properties. Numerical examples are given in Section \ref{sec:examples}.
We summarize our results and outline possible future research directions in Section \ref{sec:conclusion}.

\section{Background on multigrid and subdivision}
\label{sec:intro}

In subsection \ref{subsec:intro_mgm} we familiarize the reader with geometric multigrid. Then, in subsection \ref{subsec:intro_subdivision} we introduce multivariate subdivision.

\subsection{A short introduction to geometric multigrid}
\label{subsec:intro_mgm}

Multigrid methods are iterative methods for solving linear systems of the form 
\begin{equation} \label{eq:sys}
A_n \mathbf{x} = \mathbf{b}_n, \qquad A_n \in \CC^{n \times n}, \qquad \mathbf{b}_n \in \CC^n, \qquad n \in \NN,
\end{equation}
where often $A_n$ is assumed to be symmetric and positive definite. A basic Two-Grid Method (TGM) combines the action of a smoother and a Coarse Grid Correction (CGC): the smoother is often a simple iterative method such as Gauss-Seidel or weighted Jacobi; the CGC amounts to solving the residual equation exactly on a coarser grid.
A V-cycle multigrid method solves the residual equation approximately within the recursive application of the two-grid method, until the coarsest level is reached and there the resulting small system of equations is solved exactly.

The system matrix $A_n$ in \eqref{eq:sys} is usually derived via discretization of a $d$-dimensional elliptic PDE problem
\begin{equation} \label{eq:PDEs}
\begin{cases}
(-1)^d { \displaystyle \sum_{i=1}^d \frac{d^{2q}}{dx^{2q}_i} } \psi (\mathbf{x}) = h (\mathbf{x}) & \mathbf{x} \in \Omega = [0,1)^d, \quad q \geq 1 \\
\text{boundary conditions on } \partial \Omega.
\end{cases}
\end{equation}
Thus, $A_n$ has a certain structure depending on the discretization method and on the boundary conditions of the problem. In the case of the finite difference discretization and Dirichlet boundary conditions, the system matrix $A_n \in \CC^{n \times n}$ is multilevel Toeplitz \cite{Serra1999positive}.
In the following, we define the main ingredients of the geometric V-cycle method for multilevel Toeplitz matrices using the notation in \cite{ETNA2015} \footnote{In \cite{ETNA2015}, the authors give a complete description of the algebraic V-cycle method for multilevel Toeplitz matrices. Nevertheless, we use their notation since, in our setting, geometric and algebraic V-cycle methods differ only in the definition of the system matrices $A_{n_j}$. }. Notice that the smoother is not included in the following presentation, since it is well-known that iterative methods such as Gauss-Seidel, weighted Jacobi and weighted Richardson with an appropriate choice of the weights satisfy the condition required for convergence and optimality of multigrid (see e.g.~\cite{V-cycle_opt,lemma_smoother}).  We refer to \cite{Brandt1982guide, Hackbusch1985multi, Wesseling1987linear} for a complete description of the geometric multigrid algorithm and for its convergence and optimality analysis.

Let 
\[
\begin{array}{cc}
\begin{array}{ccccc}
\mathbf{m} &=& (m_1, \ldots, m_d) & \in \NN^d, & m_i \geq 2, \\
\mathbf{k} &=& (k_1, \ldots, k_d) & \in \NN^d, & k_i > 1, 
\end{array}
& \qquad   i = 1, \ldots, d,
\end{array}
\]
and $\ell \in \NN, \, 1 \leq \ell \leq { \displaystyle \min_ {i=1, \ldots, d} } k_i - 1$. Let
\[
\mathbf{n}_j = (m_1^{k_1-j}-1, \ldots, m_d^{k_d-j}-1), \qquad j = 0, \ldots, \ell,
\]
be the $j$-th grid of the V-cycle with $(n_j)_i+1 = m_i^{k_i-j}$ subintervals of size $(h_j)_i = m_i^{j-k_i}$ in each coordinate direction $i=1, \ldots, d$.
We recall that for $\ell = 1$ the V-cycle method reduces to the TGM, since it consists only of a fine grid $\mathbf{n}_0$ and a coarse grid $\mathbf{n}_1$.

\noindent For $j=0, \ldots, \ell$, the system matrices
\[
A_{\mathbf{n}_{j}} \in \CC^{\tilde{n}_{j} \times \tilde{n}_{j}}, \qquad \tilde{n}_j = \prod_{i=1}^d (n_j)_i,
\]
at level $j$ are obtained via finite difference discretization of \eqref{eq:PDEs} on the $j$-th grid $\mathbf{n}_j$ imposing Dirichlet boundary conditions. By construction, $A_{\mathbf{n}_j}$ is multilevel Toeplitz \cite{Serra1999positive}.
It is well-known that the entries of the multilevel Toeplitz matrices $A_{\mathbf{n}_j} = T_{\mathbf{n}_j} (f_j) $ are defined by the Fourier coefficients $\set{\mathrm{a}_j (\bm{\alpha}) \in \RR \, : \, \bm{\alpha} \in \ZZ^d}$
of the trigonometric polynomials $f_j \colon [0, 2 \pi )^d \to \CC$
\[
 f_j(\mathbf{x}) = \sum_{\substack{ \bm{\alpha} \in \ZZ^d \\ \abs{\bm{\alpha}} \leq c} } \mathrm{a}_j (\bm{\alpha}) \, e^{ - \mathrm{i} \bm{\alpha} \cdot \mathbf{x}}, \qquad \bm{\alpha} \cdot \mathbf{x} = \sum_{i=1}^d \alpha_i x_i, \qquad \mathbf{x} \in [0, 2\pi)^d,
\]
of total degree $c < \tilde{n}_j$. In general, the trigonometric polynomials $f_j$ depend on the discretization method. For example, for the $d$-variate elliptic problem of order $2q$ \eqref{eq:PDEs} and finite difference discretization, we have 
\[
f_j^{(q)}(\mathbf{x}) = \sum_{i=1}^d \frac{1}{(h_j)_i^{2q}} (2 - 2 \cos (x_i))^q = \sum_{i=1}^d \frac{1}{(h_j)_i^{2q}} (2 - e^{\mathrm{i} x_i} - e^{-\mathrm{i} x_i})^q, \qquad \mathbf{x}=(x_1, \ldots, x_d) \in [0,2 \pi)^d.
\]
Notice that the trigonometric polynomials $f_j^{(q)}$ are symmetric and vanish only at $\mathbf{0} \in \RR^d$ with order $2q$.
The properties of the matrices $A_{\mathbf{n}_j} = T_{\mathbf{n}_j} (f_j)$ are encoded in the trigonometric polynomials $f_j$. For example \cite{Tyrthyshnikov1996}, for $j = 0, \ldots, \ell$, $T_{\mathbf{n}_j}(f_j)$ is positive definite if $f_j \geq 0$, but $f_j$ not identically zero.

\noindent Another ingredient of the V-cycle method is the so-called \emph{grid transfer operator}. For $j=0, \ldots, \ell-1$, the grid transfer operator $P_{\mathbf{n}_j}$ at level $j$ is defined by
\begin{equation} \label{eq:vcycle_notation}
 P_{\mathbf{n}_j} = T_{\mathbf{n}_j} (p) \, \mathcal{K}_{\mathbf{n}_j,\mathbf{m}}^T \in \RR^{\tilde{n}_j \times \tilde{n}_{j+1}}, 
\end{equation}
where $p$ is a certain real trigonometric polynomial and $\mathcal{K}_{\mathbf{n}_j,\mathbf{m}} \in \NN^{\tilde{n}_{j+1} \times \tilde{n}_j}$ is the multilevel \emph{downsampling matrix} with the factor $\mathbf{m}$
\begin{equation} \label{eq:down_sampl}
\begin{split}
&\mathcal{K}_{\mathbf{n}_j,\mathbf{m}} = K_{(n_j)_1, m_1} \otimes \ldots \otimes  K_{(n_j)_d, m_d}, \\
&K_{(n_j)_i,m_i} = \quad \begin{bmatrix}
0_{m_i-1} & 1 & 0_{m_i-1} \\
& & & 1 & 0_{m_i-1} \\
& & & & & \ddots & \\
& & & & & & 1 & 0_{m_i-1}
\end{bmatrix} \, \in \NN^{(n_{j+1})_i \times (n_{j})_i} \qquad i =1, \ldots, d.
\end{split}
\end{equation}
Here, $0_{m_i-1} = (0, \ldots, 0) \in \NN_0^{1 \times m_i-1}$ is the zero row vector of length $m_i-1$. We are now ready to define the V-cycle algorithm.

Let $\smoother{\mathbf{n}_j}{pre}{\nu_{\text{pre}}}, \smoother{\mathbf{n}_j}{post}{\nu_{\text{post}}}$, $j = 0, \ldots, \ell$, be appropriate pre- and post-smoothers (\cite{V-cycle_opt,lemma_smoother}), and $\nu_{\text{pre}}, \, \nu_{\text{post}} \in \NN_0$ be the numbers of pre- and post-smoothing steps.
The V-cycle method (VCM) generates a sequence of iterates $\{ \mathbf{x}_{\tilde{n}_0}^{(h)} \in \RR^{\tilde{n}_0} \, : \, h \in \NN\}$ defined by
\[
\mathbf{x}_{\tilde{n}_0}^{(h+1)} = \text{VCM} (P_{\mathbf{n}_0}, A_{\mathbf{n}_0}, \mathbf{b}_{\tilde{n}_0}, 0) (\mathbf{x}_{\tilde{n}_0}^{(h)}),
\]
where the mapping $\hbox{VCM}: \RR^{\tilde{n}_0} \rightarrow \RR^{\tilde{n}_0}$ is defined iteratively by

\begin{equation} \label{alg:mgm}
\begin{array}{l}
\text{VCM} (P_{\mathbf{n}_j }, A_{\mathbf{n}_j}, \mathbf{b}_{\tilde{n}_j}, j) (\mathbf{x}_{\tilde{n}_j}^{(h)}) \\
\hline \\
\text{If } j=\ell \text{ then solve } A_{\mathbf{n}_{\ell}} \mathbf{x}_{\tilde{n}_{\ell}}^{(h+1)} = \mathbf{b}_{\tilde{n}_{\ell}} \\
\text{Else} \\
\begin{array}{lll}
1. & \text{Pre-smoother:} & \mathbf{\tilde{x}}_{\tilde{n}_j} = {\cal V}_{{\mathbf{n}_j},\text{pre}}^{\nu_{\text{pre}}}(\mathbf{x}_{\tilde{n}_j}^{(h)}) \\
\smallskip
2. & \text{Residual on the $j$-th grid:} & \mathbf{r}_{\tilde{n}_j} =  \mathbf{b}_{\tilde{n}_j} - A_{\mathbf{n}_j} \mathbf{\tilde{x}}_{\tilde{n}_j} \in \CC^{\tilde{n}_j} \\
\smallskip
3. & \text{Projection of residual on the $(j+1)$-th grid:} & \mathbf{r}_{\tilde{n}_{j+1}} = \Bigl ( \prod_{i=1}^d m_i \Bigr )^{-1} P_{\mathbf{n}_j}^T \mathbf{r}_{\tilde{n}_j} \in \CC^{\tilde{n}_{j+1}} \\
\smallskip
4. & \text{Recursion:} &  \\
\smallskip
    & \quad \mathbf{x}_{\tilde{n}_{j+1}}^{(h+1)} = 0 & \\
\smallskip
    & \quad \mathbf{x}_{\tilde{n}_{j+1}}^{(h+1)} = \text{VCM} (P_{\mathbf{n}_{j+1}}, A_{\mathbf{n}_{j+1}}, \mathbf{r}_{\tilde{n}_{j+1}} j+1) (\mathbf{x}_{\tilde{n}_{j+1}}^{(h+1)}) & \\
\smallskip
5. & \text{Correction of the previous smoothed iteration:} & \mathbf{\hat{x}}_{\tilde{n}_j} = \mathbf{\tilde{x}}_{\tilde{n}_j} + P_{\mathbf{n}_j}  \mathbf{x}_{\tilde{n}_{j+1}}^{(h+1)} \in \CC^{\tilde{n}_{j}} \\
\smallskip
6. & \text{Post-smoother:} & \mathbf{x}_{\tilde{n}_{j}}^{(h+1)} = {\cal V}_{\mathbf{n}_j,\text{post}}^{\nu_{\text{post}}} (\mathbf{\hat{x}}_{\tilde{n}_j})
\end{array}
\end{array}
\end{equation}

\noindent Notice that, at \emph{Step 3} of the V-cycle algorithm, the restriction operator is defined by $ \Bigl ( \prod_{i=1}^d m_i \Bigr )^{-1} P_{\mathbf{n}_j}^T$ (Galerkin approach). The scaling factor $c_{\mathbf{m}} = \Bigl ( \prod_{i=1}^d m_i \Bigr )^{-1}$ is necessary for the convergence of the method \cite{Hemker1990}.
For the sake of simplicity, we depict the iterative structure of the V-cycle in the following figure.
{\scriptsize
\begin{displaymath}
\xymatrix{
\CC^{\tilde{n}_0} \ar[dr]_{c_{\mathbf{m}} P_{\mathbf{n}_0}^T} & & & & & & \CC^{\tilde{n}_0} \\
& \CC^{\tilde{n}_1} \ar[dr]_{c_{\mathbf{m}} P_{\mathbf{n}_1}^T} & & & & \CC^{\tilde{n}_1} \ar[ur]_{P_{\mathbf{n}_0}} & \\
& & \CC^{\tilde{n}_2} \ar@{.>}[dr]_{c_{\mathbf{m}} P_{\mathbf{n}_{\ell-1}}^T} & & \CC^{\tilde{n}_2} \ar[ur]_{P_{\mathbf{n}_1}} & & \\
& & & \CC^{\tilde{n}_{\ell}} \ar@{.>}[ur]_{ P_{\mathbf{n}_{\ell-1}}} & & & }
\end{displaymath}
}

For $\mathbf{m}=(2,\ldots,2) \in \ZZ^d$, it is well-known (\cite{Brandt1982guide, Hackbusch1985multi, Hemker1990, Wesseling1987linear}) that a sufficient condition for convergence and optimality of the V-cycle method \eqref{alg:mgm} for elliptic PDEs problem of the form \eqref{eq:PDEs} is that the real trigonometric polynomial $p$ associated to the grid transfer operators $P_{\mathbf{n}_j}$, $j=0, \ldots, \ell-1$, in \eqref{eq:vcycle_notation} satisfies
\[
\begin{array}{cll}
\smallskip
(i) & { \displaystyle \limsup_{\mathbf{x} \in \RR^d, \, \norm{\mathbf{x}} \to 0} \frac{p(\mathbf{x} + 2 \pi \mathbf{y})}{\norm{\mathbf{x}}^{2q}}} < \infty, & \quad \forall \, \mathbf{y} = (y_1, \ldots, y_d) \in \QQ^d \, : \,y_i \in \Set{0,\frac12}, \, i=1, \ldots, d, \, \mathbf{y} \neq \bm{0}, \\
(ii) & { \displaystyle \limsup_{\mathbf{x} \in \RR^d, \, \norm{\mathbf{x}} \to 0} \frac{p(\mathbf{x}) - 2^d}{\norm{\mathbf{x}}}} < \infty. &
\end{array}
\]
This result can be extended to our general setting with $\mathbf{m} \in \NN^d$, by requiring
\begin{equation} \label{eq:opt_Vcycle_conditions}
\begin{array}{cll}
\smallskip
(i) & { \displaystyle \limsup_{\mathbf{x} \in \RR^d, \, \norm{\mathbf{x}} \to 0} \frac{p(\mathbf{x} + 2 \pi \mathbf{y})}{\norm{\mathbf{x}}^{2q}}} < \infty, & \quad \forall \, \mathbf{y} = (y_1, \ldots, y_d) \in \QQ^d \, : \,y_i \in \Set{0,\frac{1}{m_i}, \ldots, \frac{m_i-1}{m_i}}, \, i=1, \ldots, d, \, \mathbf{y} \neq \bm{0}, \\
(ii) & { \displaystyle \limsup_{\mathbf{x} \in \RR^d, \, \norm{\mathbf{x}} \to 0} \frac{p(\mathbf{x}) - \prod_{i=1}^d m_i}{\norm{\mathbf{x}}}} < \infty. &
\end{array}
\end{equation}

\begin{remark} \label{remark:mgm_subd}
We focus our attention on the grid transfer operators $P_{\mathbf{n}_j}$, $j=0, \ldots, \ell-1$, in \eqref{eq:vcycle_notation}. From the Fourier coefficients of the real trigonometric polynomial $p$ associated to the multilevel Toeplitz matrix $T_{\mathbf{n}_j} (p)$, one can build a finite sequence of real numbers $\mathbf{p} \in \ell_0 (\ZZ^d)$. Then, the multiplication by the matrix $P_{\mathbf{n}_j}$  is done by 
\begin{enumerate}
\item upsampling with the factor $\mathbf{m}$ via multiplication by $\mathcal{K}_{\mathbf{n}_j,\mathbf{m}}^T$,
\item convolution with $\mathbf{p}$ via multiplication by $T_{\mathbf{n}_j} (p)$.
\end{enumerate}
It is well-known that upsampling and convolution amount to one step of subdivision scheme with dilation $M=\underset{i=1, \ldots, d}{{\rm diag}} m_i \in \NN^{d \times d}$  and mask $\mathbf{p}$. It is then natural to study conditions on subdivision schemes that
will guarantee convergence and optimality of the corresponding multigrid methods.
\end{remark}

\subsection{Subdivision schemes}
\label{subsec:intro_subdivision}

In the following, for convenience of readers from multigrid community, we shortly describe $d$-variate subdivision schemes and list the well-known results on their convergence, polynomial generation property and stability. The results we present here are used in Sections \ref{sec:anisotropic_interp} and \ref{sec:anisotropic_pseudo}, where we define two new families of bivariate anisotropic subdivision schemes and analyze their properties. The link between multigrid and subdivision via observation in Remark \ref{remark:mgm_subd} is presented in Section \ref{sec:examples}.

Let $\mathbf{m} = (m_1, \ldots, m_d) \in \NN^d$, $m_i \geq 2$, $i=1, \ldots, d$.
The diagonal matrix $M =  \underset{i=1, \ldots, d}{\mathrm{diag}} m_i \in \NN^{d \times d}$ is a dilation matrix, as all its eigenvalues are in the absolute value greater than 1.
Let $\mathbf{p} = \{ \mathrm{p} (\bm{\alpha}) \in \RR \, : \, \bm{\alpha} \in \ZZ^d \} \in \ell_0 (\ZZ^d)$ be a finite sequence of real numbers.
The \emph{dilation} $M$ and the \emph{mask} $\mathbf{p}$ are used to define the subdivision operator $\mathcal{S}_{\mathbf{p}} \colon \ell(\ZZ^d) \to \ell (\ZZ^d)$, which is a linear operator such that
\[
\left ( \mathcal{S}_{\mathbf{p}} \mathbf{c} \right ) (\bm{\alpha}) = \sum_{\bm{\beta} \in \ZZ^d} \mathrm{p} (\bm{\alpha} - M \bm{\beta}) \mathrm{c} (\bm{\beta}), \quad \alpha \in \ZZ^d, \quad \mathbf{c} \in \ell (\ZZ^d).
\]
A subdivision scheme $S_{\mathbf{p}}$ with dilation $M$ and mask $\mathbf{p}$ is the recursive application of the subdivision operator $\mathcal{S}_{\mathbf{p}}$ to some initial data sequence
$\mathbf{c}^{(0)} = \{ \mathrm{c}^{(0)} (\bm{\alpha}) \in \RR \, : \, \bm{\alpha} \in \ZZ^d \} \in \ell (\ZZ^d)$, namely
\begin{equation}
 \label{eq:subdivision}
 \mathbf{c}^{(k+1)} = \mathcal{S}_{\mathbf{p}} \mathbf{c}^{(k)}, \quad k \in \NN_0.
\end{equation}
Notice that $\mathbf{c}^{(k+1)} = \mathcal{S}_{\mathbf{p}} \mathbf{c}^{(k)} = \cdots = (\mathcal{S}_{\mathbf{p}})^{k+1} \mathbf{c}^{(0)}$.

Since the subdivision scheme $S_{\mathbf{p}}$ generates sequences $\mathbf{c}^{(k)} \in \ell (\ZZ^d)$, $k \geq 0$, a natural way to define a notion of its convergence is to attach the data $\mathbf{c}^{(k)} = \set{\mathrm{c}^{(k)} (\bm{\alpha}) \, : \, \bm{\alpha} \in \ZZ^d }$, $k \geq 0$, to the parameter values $ \mathbf{t}^{(k)} = \set{ M^{-k} \bm{\alpha} \, : \, \bm{\alpha} \in \ZZ^d }$, $k \geq 0$, and to require that there exists a continuous function $F_{\mathbf{c}^{(0)}} \colon \RR^d \to \RR$ depending on the starting sequence $\mathbf{c}^{(0)}$ such that the values of $F_{\mathbf{c}^{(0)}}$ at the parameter values $\mathbf{t}^{(k)}$ are ''close'' enough to the data $\mathbf{c}^{(k)}$ for $k$ sufficiently large.

\begin{definition}
A subdivision scheme $S_{\mathbf{p}}$ is \emph{convergent} if for any initial data $\mathbf{c} \in \ell^{\infty} (\ZZ^d)$ there exist a uniformly continuous function
$F_{\mathbf{c}} \in C(\RR^d)$ such that
\[
\lim_{k \to \infty} \, \sup_{\bm{\alpha} \in \ZZ^d} \quad \abs*{\, F_{\mathbf{c}} \left ( M^{-k} \bm{\alpha} \right ) - \left ( \mathcal{S}_{\mathbf{p}}^k \mathbf{c} \right ) (\bm{\alpha}) \, } = 0.
\]
\end{definition}

The particular choice of the initial data $\bm{\delta} = \Set{\delta_{\bm{\alpha},0} \, : \, \bm{\alpha} \in \ZZ^d}$ defines the so-called \emph{basic limit function} $\phi = F_{\bm{\delta}}$. Since the mask $\mathbf{p} \in \ell_0 (\ZZ^d)$ is a finite sequence, $\phi$ is compactly supported. It is well-known that the basic limit function $\phi$ satisfies the refinement equation
\begin{equation} \label{eq:ref_eq}
\phi  = \sum_{\bm{\alpha} \in \ZZ^d} \mathrm{p} (\bm{\alpha}) \phi (M \, \cdot \, - \bm{\alpha}).
\end{equation}
Thus, due to the linearity of $\mathcal{S}_{\mathbf{p}}$, for any initial data $\mathbf{c} \in \ell^{\infty}(\ZZ^d)$, $ \mathbf{c} = \displaystyle \sum_{\bm{\alpha} \in \ZZ^d} \mathrm{c} (\bm{\alpha}) \bm{\delta}(\cdot - \bm{\alpha})$, we have 
\[
F_{\mathbf{c}} = \lim_{k \to \infty} \mathcal{S}_{\mathbf{p}}^k \mathbf{c} = \sum_{\bm{\alpha} \in \ZZ^d} \mathrm{c} (\bm{\alpha}) \phi(\cdot - \bm{\alpha}).
\]
For more details on the properties of the basic limit function, see the seminal work of Cavaretta et al. \cite{Cava_Dam_Micc} and the survey by Dyn and Levin \cite{Dyn_Lev}.

Most of the properties of the subdivision scheme $S_{\mathbf{p}}$ can be investigated studying the Laurent polynomial
\begin{equation} \label{eq:symbol_trigo_poly}
p (\mathbf{z}) = \sum_{\bm{\alpha} \in \ZZ^d} \mathrm{p} (\bm{\alpha}) \mathbf{z}^{\bm{\alpha}}, \qquad \mathbf{z}^{\bm{\alpha}} = z_1^{\alpha_1} \cdot \ldots \cdot z_d^{\alpha_d}, \qquad \mathbf{z} \in  ( \CC \setminus \set{0} )^d,
\end{equation}
called the \emph{symbol} of the subdivision scheme.

In Section \ref{sec:anisotropic_interp}, we are interested in interpolatory subdivision schemes. 
We say that a subdivision scheme $S_{\mathbf{p}}$ with dilation $M$ and mask $\mathbf{p}$ is \emph{interpolatory} if its mask $\mathbf{p}$ satisfies
\begin{equation} \label{eq:mask_interp}
\mathrm{p} (\bm{0}) = 1 \qquad \text{and} \qquad \mathrm{p} (M \bm{\alpha}) = 0, \qquad \forall \, \bm{\alpha} \neq 0.
\end{equation}

\noindent In \cite{Conti_Cotronei_Sauer_interp_multivariate}, in the case of dilation matrix $2 I_d$, where $I_d$ is the identity matrix of dimension $d \times d$, the authors characterized the interpolation property of subdivision in terms of the corresponding subdivision symbol. Their result can be easily extended to the case of diagonal anisotropic dilation matrix $M$. We denote by $\Gamma$ the complete set of representatives of the distinct cosets of $\ZZ^d / M \ZZ^d$, namely
\begin{equation} \label{eq:Gamma}
 \Gamma = \Set{ \bm{\gamma}=(\gamma_1, \ldots, \gamma_d) \in \ZZ^d \, : \, \gamma_i \in \set{0, \ldots, m_i-1}, \, i = 1, \ldots, d},
\end{equation}
and we define the set
\[
 E_M = \Set{e^{- \mathrm{i} 2 \pi M^{-1} \bm{\gamma}} \, : \, \bm{\gamma} \in \Gamma},
\]
containing $\mathbf{1} = (1, \dots, 1) \in \ZZ^d$. 

\begin{theorem} \label{t:interp_symbol}
A convergent subdivision scheme $S_{\mathbf{p}}$ is interpolatory if and only if
\[
\sum_{ \bm{\xi} \in E_M} p \bigl ( \bm{\xi} \cdot \bm{z} \bigr ) =  \abs{\text{det } M }, \qquad  \bm{\xi} \cdot \bm{z} = (\xi_1 z_1, \ldots, \xi_d z_d), \qquad \bm{z} \in (\CC \setminus \set{0})^d.
\]
\end{theorem}

We now introduce the concepts of polynomial generation and reproduction.
The property of generation of polynomials of degree $n$ is the capability of a subdivision scheme to generate the full space of polynomials up to degree $n$, while the property of reproduction of polynomials of degree $n$  is the capability of a subdivision scheme to produce in the limit exactly the same polynomial from which the data is sampled. It is easy to see that reproduction of polynomials of degree $n$ implies generation of polynomials of degree $n$.
We denote by $\Pi_n$ the space of polynomials of total degree less than or equal to $n \in \NN_0$.

\begin{definition} \label{d:poly_generation}
A convergent subdivision scheme $S_{\mathbf{p}}$ \emph{generates} polynomials up to degree $n$ if
\[
\text{for any } \, \pi \in \Pi_{n}, \quad \exists \, \mathbf{c} \in \ell (\ZZ^d) \quad \text{such that} \quad \sum_{\bm{\alpha} \in \ZZ^d} \mathrm{c} (\bm{\alpha}) \phi(\cdot - \bm{\alpha}) = \tilde{\pi} \in \Pi_{n}.
\]
\end{definition}
The property of polynomial generation has been studied e.g. by  Cabrelli at al. in \cite{Cabrelli2004self}, Cavaretta et al. in \cite{Cava_Dam_Micc}, Jetter and Plonka in \cite{Jetter2001shift}, Jia in \cite{Jia1996subdivision, Jia_approx_prop}, Levin in \cite{Levin_gen_non_unif}. Definition \ref{d:poly_generation} can be interpreted as follows: a convergent subdivision scheme $S_{\mathbf{p}}$ \emph{generates} polynomials up to degree $n$ if the integer shifts of its basic limit function $\set{\phi (\cdot - \bm{\alpha}) \, : \, \bm{\alpha} \in \ZZ^d}$ span the space $\Pi_n$.

Algebraic properties of the symbol $p$ characterize the polynomial generation property of subdivision.
For $\bm{\mu} \in \NN_0^d$, we denote by $D^{\bm{\mu}}$ the $\bm{\mu}$-th directional derivative and $\abs{\bm{\mu}} := \mu_1 + \ldots + \mu_d$.

\begin{theorem}\label{t:poly_generation}
Let $n \in \NN_0$. A convergent subdivision scheme $S_{\mathbf{p}}$ generates polynomials up to degree $n$ if and only if
\begin{equation} \label{conditions:poly_generation}
 D^{\bm{\mu}} p (\bm{\varepsilon}) = 0, \qquad \forall \, \bm{\varepsilon} \in E_M \setminus \set{\mathbf{1}}, \qquad \bm{\mu} \in \NN_0^d, \qquad \abs{\bm{\mu}} \leq n.
\end{equation}
\end{theorem}

Thus, the property of polynomial generation of a convergent subdivision scheme $S_{\mathbf{p}}$ is strictly related to the behavior of the subdivision symbol $p(\mathbf{z})$ and of its derivatives at the "special" points $E_M \setminus \set{\mathbf{1}}$.
Conditions in \eqref{conditions:poly_generation} are also known as \emph{zero conditions of order} $n+1$.
A subdivision symbol $p(\mathbf{z})$ satisfies the zero conditions of order $n+1$ if and only if the associated mask $\mathbf{p}$ satisfies the \emph{sum rules of order} $n+1$, namely
\begin{equation} \label{eq:sum_rules}
\sum_{\bm{\alpha} \in \ZZ^d} \mathrm{p} ( M \bm{\alpha} ) \pi ( M \bm{\alpha} ) = \sum_{\bm{\alpha} \in \ZZ^d} \mathrm{p} ( \bm{\gamma} + M \bm{\alpha}) \pi ( \bm{\gamma} + M \bm{\alpha} ), \qquad \forall \, \bm{\gamma} \in \Gamma, \, \pi \in \Pi_n. 
\end{equation}

In the univariate setting, Theorem \ref{t:poly_generation} is equivalent to requiring that the symbol $p(z)$ of the subdivision scheme  $S_{\mathbf{p}}$ of dilation $m \in \NN$, $m \geq 2$, has the following factorization
\begin{equation} \label{eq:symbol_factorization_generation}
p(z) = \left (1+z+z^2+\dots+z^{m-1} \right )^{n + 1} \, b(z),
\end{equation}
for some Laurent polynomial $b(z)$ such that $b(1) = m^{-n}$, i.e. $p(1)=m$.

In the bivariate setting, we lose the factorization property \eqref{eq:symbol_factorization_generation}. Nevertheless, Theorem \ref{t:poly_generation} can be reformulated in terms of ideals \cite{Sauer2002ideal}, leading to an equivalent decomposition property.
Let $n \in \NN_0$ and define 
\[
\mathcal{J}_n := \, < (1-z_1^{m_1})^{\mu_1} (1-z_2^{m_2})^{\mu_2} \, : \, \bm{\mu} = (\mu_1, \mu_2) \in \NN_0^2, \, \abs{\bm{\mu}} = n+1>. 
\]
$\mathcal{J}_n$ is the ideal of all bivariate polynomials $p(z_1,z_2)$ which satisfy
\[
D^{\bm{\mu}} p (\bm{\varepsilon}) = 0, \qquad \forall \bm{\varepsilon} \in E_M, \qquad \forall \bm{\mu} \in \NN_0^2, \qquad \abs{\bm{\mu}} \leq n.
\]
Thus, the quotient ideal
\[
\mathcal{I}_n \quad = \quad \mathcal{J}_n \quad : \quad <(1-z_1)^{\mu_1} (1-z_2)^{\mu_2} \, : \, \bm{\mu} = (\mu_1, \mu_2) \in \NN_0^2, \, \abs{\bm{\mu}} = n+1 >
\] 
 is the ideal of all bivariate polynomials $p(z_1,z_2)$ which satisfy \eqref{conditions:poly_generation}
Consequently, a convergent subdivision scheme generates polynomials up to degree $n$ if and only if its symbol $p \in \mathcal{I}_n$. Finally, if for $n, \ell \in \NN_0$, $p \in \mathcal{I}_n$ and $q \in \mathcal{I}_{\ell}$, then $p \cdot q \in \mathcal{I}_{n+\ell+1}$.

The definition of the polynomial reproduction property differs from the definition of the polynomial generation property as the before mentioned property depends on the so-called sequence of \emph{parameter values}.
Let $\bm{\tau} = (\tau_1, \ldots, \tau_d) \in \RR^d$. The parameter values $\mathbf{t}^{(k)} = \set{\mathbf{t}^{(k)} (\bm{\alpha}) \in \RR^d \, : \, \bm{\alpha} \in \ZZ^d }$, $k \geq 0$, are defined recursively by
\begin{equation} \label{eq:parameter_values}
\mathbf{t}^{(k)} (\bm{\alpha}) = \mathbf{t}^{(k)} (\bm{0}) + M^{-k} \bm{\alpha}, \qquad \mathbf{t}^{(k)} (\bm{0}) = \mathbf{t}^{(k-1)} (\bm{0}) - M^{-k} \bm{\tau},  \qquad \mathbf{t}^{(0)} (\bm{0}) = \bm{0}, \qquad \bm{\alpha} \in  \ZZ^d, \qquad k \geq 0.
\end{equation}

\begin{definition} \label{d:poly_reproduction}
A convergent subdivision scheme $S_{\mathbf{p}}$ \emph{reproduces} polynomials up to degree $n$  with respect to the parameter values \eqref{eq:parameter_values} if
\[
\text{for any } \, \pi \in \Pi_{n} \, \text{ and } \mathbf{c} = \Set{ \pi \bigl ( \mathbf{t}^{(0)} (\bm{\alpha}) \bigr ) \, : \, \bm{\alpha} \in \ZZ^d} \in \ell (\ZZ^d), \quad  \sum_{\bm{\alpha} \in \ZZ^d} \mathrm{c} (\bm{\alpha}) \phi(\cdot - \bm{\alpha}) = \pi \in \Pi_{n}.
\]
\end{definition}

\noindent Definition \ref{d:poly_reproduction} is more restrictive than Definition \ref{d:poly_generation} since we require that the subdivision limit is exactly the same polynomial $\pi$ from which the initial data $\mathbf{c}$ is sampled.
Charina et al. proved in \cite{Charina2014reproduction} that the property of polynomial reproduction is characterized in terms of the subdivision symbol.

\begin{theorem}\label{t:poly_reproduction}
Let $n \in \NN_0$. A convergent subdivision scheme $S_{\mathbf{p}}$ with parameter values \eqref{eq:parameter_values} reproduces polynomials up to degree $n$ if and only if
\begin{equation} \label{conditions:poly_reproduction}
 D^{\bm{\mu}} p (\bm{1}) = \abs{\text{det } M } \, \prod_{i=1}^d \, \prod_{\ell_i = 0}^{\mu_i-1} (\tau_i - \ell_i) \quad \text{and} \quad D^{\bm{\mu}} p (\bm{\varepsilon}) = 0, \quad \forall \, \bm{\varepsilon} \in E_M \setminus \set{\mathbf{1}}, \quad \bm{\mu} \in \NN_0^d, \quad \abs{\bm{\mu}} \leq n.
\end{equation}
\end{theorem}

\noindent Theorem \ref{t:poly_reproduction} implies that, in order to have the maximum degree of polynomial reproduction, it is necessary to choose the parameter $\bm{\tau} \in \RR^d$ in \eqref{eq:parameter_values} carefully. 

In the univariate case, if the subdivision mask $\mathbf{p}$ is symmetric, i.e. $\mathrm{p}(\alpha) = \mathrm{p}(-\alpha)$, or interpolatory, then $\tau = 0$ is the optimal choice (\cite{CH2011}). Thus, \eqref{conditions:poly_generation} becomes
\[
\begin{split}
&p(1) = m, \\
&D^{\mu} p (1) = 0, \quad \mu \in \NN_0, \quad 1 \leq \mu \leq n, \\
&D^{\mu} p (\varepsilon) = 0, \quad \forall \, \varepsilon \in E_m \setminus \set{1}, \quad \mu \in \NN_0, \quad 0 \leq \mu \leq n.
\end{split}
\]
Therefore, in the univariate symmetric or interpolatory setting, Theorem \ref{t:poly_reproduction} is equivalent to requiring that the symbol $p(z)$ of the subdivision scheme $S_{\mathbf{p}}$ of dilation $m \in \NN$, $m \geq 2$, has the following decomposition \cite{ CH2011, Dyn2008polynomial}
\begin{equation} \label{eq:symbol_factorization_reproduction}
p(z) = m +  (1-z)^{n + 1} \, c(z), \quad z \in \CC \setminus \set{0},
\end{equation}
for a suitable Laurent polynomial $c(z)$.

In the bivariate case, if the subdivision mask $\mathbf{p}$ is symmetric, i.e.
\[
\mathrm{p}( \alpha_1, \alpha_2) = \mathrm{p}( \alpha_1, - \alpha_2) = \mathrm{p}( - \alpha_1, \alpha_2) = \mathrm{p}( - \alpha_1, - \alpha_2),
\]
or interpolatory, then $\bm{\tau} = \bm{0}$ is the optimal choice (\cite{Charina2014reproduction}) and \eqref{conditions:poly_generation} becomes
\[
\begin{split}
&p(\mathbf{1}) = \abs{\text{det } M }, \\
&D^{\bm{\mu}} p (\mathbf{1}) = 0, \quad \bm{\mu} \in \NN_0^d,\quad 1 \leq \abs{\mu} \leq n, \\
&D^{\bm{\mu}} p (\bm{\varepsilon}) = 0, \quad \forall \, \bm{\varepsilon} \in E_M \setminus \set{\mathbf{1}}, \quad \bm{\mu} \in \NN_0^d, \quad 0 \leq \abs{\bm{\mu}} \leq n.
\end{split}
\]
Thus, in the bivariate symmetric or interpolatory setting, Theorem \ref{t:poly_reproduction} is equivalent to requiring that 
\[
p(\mathbf{z}) - \abs{\text{det } M } \quad \in \quad <(1-z_1)^{\mu_1} (1-z_2)^{\mu_2} \, : \, \bm{\mu} = (\mu_1,\mu_2) \in \NN_0^2, \, \abs{\bm{\mu}} \geq n+1 >,
\]
or, equivalently, that the symbol $p(\mathbf{z})$ of the subdivision scheme $S_{\mathbf{p}}$ with dilation $M$ has the following decomposition
\begin{equation} \label{eq:symbol_factorization_reproduction_2D}
\begin{split}
&p (z_1, z_2) = \abs{\text{det } M } + \sum_{h=0}^H (1-z_1)^{\alpha_h} \, (1-z_2)^{\beta_h} \, c_h (z_1,z_2), \qquad (z_1,z_2) \in ( \CC \setminus \set{0} )^2, \\
&\alpha_h, \, \beta_h \in \NN_0, \qquad \alpha_h + \beta_h \geq n+1, \qquad h = 0, \ldots, H,
\end{split}
\end{equation}
for suitable Laurent polynomials $c_h$, $h = 0, \ldots, H$ (we require in \eqref{eq:symbol_factorization_reproduction_2D} that at least one pair $\alpha_h, \, \beta_h \in \NN_0$ satisfies $ \alpha_h + \beta_h = n+1$).
Identity \eqref{eq:symbol_factorization_reproduction_2D} is a natural generalization of the univariate identity \eqref{eq:symbol_factorization_reproduction}.

\begin{remark}
We are interested in symmetric subdivision schemes due to the use of vertex centered discretization for our numerical examples in Section \ref{sec:examples}. 
\end{remark}

\section{Anisotropic interpolatory subdivision}
\label{sec:anisotropic_interp}

In subsection \ref{ssec:1dcase}, we start by introducing the family of univariate interpolatory Dubuc-Deslauriers subdivision schemes. These will be a basis for our bivariate construction in subsection \ref{ssec:2dcase}.

\subsection{Univariate case} \label{ssec:1dcase}

In \cite{Deslauriers_Dubuc1989symmetric}, Deslauriers and Dubuc proposed a general method for constructing symmetric interpolatory subdivision schemes of dilations $m \in \NN$, $m \geq 2$. The smoothness analysis of their schemes was conducted by Eirola et al. in \cite{Eirola1992sobolev}. Recently, Diaz Fuentes proposed in his master thesis \cite{Fuentes2015thesis} a closed formula for computing the mask of the  interpolatory Dubuc-Deslauriers subdivision schemes for any dilation $m \in \NN$, $m \geq 2$.

\begin{definition} \label{d:DD_interp_1D}
Let $m \in \NN$, $m \geq 2$, and $n \in \NN$.
The univariate $(2n)$-point Dubuc-Deslauriers interpolatory subdivision schemes of dilation $m$ is defined by its symbol
\begin{equation} \label{eq:DD_interp_1D}
a_{m,n} (z) = 1 + \sum_{\varepsilon=1}^{m-1} \sum_{\beta=-n+1}^n \, \frac{(-1)^{\beta+n}}{(2n-1)! ( \varepsilon/m - \beta )} \binom{2n-1}{n-\beta}  \Biggl (- n+1- \frac{\varepsilon}{m} \Biggr )_{2n} \, z^{- m\beta+\varepsilon}, \qquad z \in \CC \setminus \set{0},
\end{equation}
where for any $x \in \RR$, $(x)_{\ell}$ is the Pochhammer symbol defined by
\[
(x)_0 := 1, \qquad \text{and} \qquad (x)_{\ell} := x \, (x+1) \cdots (x + \ell -1), \qquad \ell \in \NN.
\]
\end{definition}

For reader's convenience, we recall the main ideas in \cite{Deslauriers_Dubuc1989symmetric} behind the construction of symmetric interpolatory subdivision schemes and repeat a few computations from \cite{Fuentes2015thesis} conducted in order to obtain the symbols $a_{m,n}$ in \eqref{eq:DD_interp_1D}. Without loss of generality, we focus on the case $m=2 \ell + 1$, $\ell \in \NN$.

Let $\mathbf{c} = \set{ \mathrm{c} (\gamma) \in \RR \, : \, \gamma \in \ZZ} \in \ell (\ZZ)$ and fix an integer $\alpha \in \ZZ$. Let $\mathbf{c}^{\alpha} = \set{ \mathrm{c} (\gamma) \, : \, \alpha - n + 1 \leq \gamma \leq \alpha + n} \in \ell_0(\ZZ)$ be $2n$ consecutive elements of $\mathbf{c}$ centered in $\mathrm{c}(\alpha)$. There exist a unique polynomial $\pi \in \Pi_{2n-1}$ of degree $2n-1$ which interpolates $\mathbf{c}^{\alpha}$ at the integers $\set{\alpha - n + 1, \ldots, \alpha + n}$, namely
\[
\begin{split}
&\pi (t) = \sum_{\beta = \alpha - n + 1}^{\alpha + n} \mathrm{c}^{\alpha} (\beta) \mathcal{L}_{\beta}^{(\alpha - n + 1,2n-1)} (t)  
= \sum_{\beta = - n + 1}^{n} \mathrm{c}^{\alpha} (\beta + \alpha) \mathcal{L}_{\beta + \alpha}^{(\alpha - n + 1,2n-1)} (t),  \\
&\mathcal{L}_{\beta}^{(\alpha - n + 1,2n-1)} (t) = \prod_{\substack{\gamma = \alpha - n + 1 \\ \gamma \neq \beta}}^{\alpha + n} \frac{t - \gamma}{\beta - \gamma}, \qquad t \in \RR.
\end{split}
\]
For $\beta \in \set{ \alpha - n + 1, \ldots,  \alpha + n}$, $\mathcal{L}_{\beta}^{(\alpha - n + 1,2n-1)}$ is the Lagrange polynomial of degree $2n-1$, centered in $\beta$, defined on the $2n$ nodes $\set{\alpha - n + 1, \ldots, \alpha + n}$, and satisfies $\mathcal{L}_{\beta}^{(\alpha - n + 1,2n-1)} (\varepsilon) = \delta_{\beta , \varepsilon}$, $\varepsilon \in \set{\alpha - n + 1, \ldots, \alpha + n}$.
In order to define the subdivision operator $\mathcal{S}_{\mathbf{a}_{m,n}}$, we define its action on the finite sequence $\mathbf{c}^{\alpha}$ by
\[
\begin{array}{rll}
(\mathcal{S}_{\mathbf{a}_{m,n}} \mathbf{c}^{\alpha} ) (m \alpha + \varepsilon) := \pi \Biggl ( \alpha + {\displaystyle \frac{\varepsilon}{m}} \Biggr ) 
& = {\displaystyle \sum_{\beta = - n + 1}^{ n}} \mathcal{L}_{\beta + \alpha }^{(\alpha - n + 1,2n-1)} \Biggl ( \alpha + {\displaystyle \frac{\varepsilon}{m}} \Biggr ) \mathrm{c}^{\alpha} (\beta + \alpha) & \quad \text{by definition of } \pi \\
& = {\displaystyle  \sum_{\beta = \alpha - n + 1}^{\alpha + n}} \mathrm{a}_{m,n} \bigl (m (\alpha - \beta) + \varepsilon \bigr ) \mathrm{c}^{\alpha}(\beta) & \quad \text{by definition of subdivision operator} \\
& = {\displaystyle \sum_{\beta = - n + 1}^{n}} \mathrm{a}_{m,n}\bigl ( - m \beta + \varepsilon \bigr ) \mathrm{c}^{\alpha}(\beta + \alpha), & \quad \varepsilon = - {\displaystyle \frac{m-1}{2}, \ldots, \frac{m-1}{2}}.
\end{array}
\]
For $\varepsilon = 0$, $(\mathcal{S}_{\mathbf{a}_{m,n}} \mathbf{c}^{\alpha}) (m \alpha) = \pi ( \alpha ) = \mathrm{c}^{\alpha} (\alpha)$, thus $\mathrm{a}_{m,n} (m \beta) = \delta_{\beta, 0}$.

\noindent For $\varepsilon \neq 0$, using simple properties of $\mathcal{L}_{\beta}^{(\alpha - n + 1,2n-1)}$ we get
\[
\begin{split}
\mathrm{a}_{m,n}\bigl ( - m \beta + \varepsilon \bigr ) &= \mathcal{L}_{\beta + \alpha}^{(\alpha - n + 1,2n-1)} \Biggl ( \alpha + \frac{\varepsilon}{m} \Biggr ), \\
&= \mathcal{L}_{\beta}^{(- n + 1,2n-1)} \Biggl ( \frac{\varepsilon}{m} \Biggr ), \\
&= \frac{(-1)^{\beta+n}}{(2n-1)! ( \varepsilon/m - \beta)} \binom{2n-1}{n-\beta} \Biggl (-n+1- \frac{\varepsilon}{m} \Biggr )_{2n}.
\end{split}
\]
Formula \eqref{eq:DD_interp_1D} follows from property
\[
a_{m,n} (z) = \sum_{\alpha \in \ZZ} \mathrm{a}_{m,n} (\alpha) z^{\alpha} = \sum_{\varepsilon = 0}^{m-1} \sum_{\beta = - n+1}^n \mathrm{a}_{m,n} ( - m \beta + \varepsilon ) z^{- m \beta + \varepsilon }, \qquad z \in \CC \setminus \set{0}.
\] 
By construction, for any $m \in \NN$, $m \geq 2$, and $n \in \NN$, the  univariate $(2n)$-point Dubuc-Deslauriers interpolatory subdivision schemes of dilation $m$ generate and reproduce polynomials up to degree $2n-1$.
We recall that $\mathcal{S}_{\mathbf{a}_{m,n}}$ is the unique univariate subdivision scheme of dilation $m$ such that 
\begin{enumerate}
\item it is interpolatory, 
\item it generates polynomials up to degree $2n-1$,
\item its mask $\mathbf{a}_{m,n}$ is symmetric and has support $\set{1-mn, \ldots, mn-1}$.
\end{enumerate}

\subsection{Bivariate case} \label{ssec:2dcase}

From the family of univariate interpolatory Dubuc-Deslauriers subdivision schemes we build a family of bivariate interpolatory subdivision schemes with dilation matrix $M$ in \eqref{eq:dilation} using the approach from \cite{Conti2017pseudo}.

\begin{definition}[Anisotropic interpolatory subdvision schemes] \label{d:interp_anisotropic}
Let $n \in \NN$. The \emph{anisotropic interpolatory} subdvision scheme  $\mathcal{S}_{\mathbf{a}_{M,n}}$ of order $n$ and dilation matrix $M$ in \eqref{eq:dilation} is defined by its symbol
\begin{equation} \label{eq:interp_anisotropic}
a_{M,n} (z_1,z_2) := \sum_{k=0}^{n-1} a_{2,n-k} (z_1) \, a_{m,k+1} (z_2) - \sum_{k=0}^{n-2} a_{2, n-k-1} (z_1) \, a_{m,k+1} (z_2), \qquad (z_1,z_2) \in ( \CC \setminus \set{0})^2,
\end{equation}
where $a_{2,k}, \, a_{m,k}$ are the symbols of the univariate $(2n)$-point Dubuc-Deslauriers interpolatory subdivision schemes defined in \eqref{eq:DD_interp_1D} of dilation 2 and $m$, respectively.
\end{definition}

Definition \ref{d:interp_anisotropic} is justified by the following result.

\begin{proposition} \label{p:anisotropic_interp}
Let $n \in \NN$ and $M$ in \eqref{eq:dilation}. The anisotropic subdvision scheme $\mathcal{S}_{\mathbf{a}_{M,n}}$ in Definition \ref{d:interp_anisotropic} is interpolatory.
\end{proposition}

\begin{proof}
Let $n \in \NN$. By Theorem \ref{t:interp_symbol}, in order to prove Proposition \ref{p:anisotropic_interp}, we need to show that
\[
s_n(z_1,z_2) := \sum_{j=0}^{m-1} a_{M, n} ( z_1, \, \xi_j z_2 ) + \sum_{j=0}^{m-1} a_{M, n} (- z_1, \, \xi_j z_2 ) = \abs{ \det M} = 2m,
\]
for $ \xi_j = {\displaystyle e^{- \frac{2 \pi \mathrm{i}}{m} j }}$, $ j=0, \ldots, m-1$.
Since $\mathcal{S}_{\mathbf{a}_{2,n}}$ and $\mathcal{S}_{\mathbf{a}_{m,n}}$ are univariate interpolatory subdivision schemes of dilation 2 and $m$ respectively, Theorem \ref{t:interp_symbol} guarantees that their symbols satisfy
\[
a_{2,n}( z_1 ) + a_{2,n}(- z_1 ) = 2 \qquad \text{and} \qquad  \sum_{j=0}^{m-1} a_{m,n}  ( \xi_j z_2  ) = m.
\]
By \eqref{eq:interp_anisotropic}, we have
\[
\begin{split}
s_n(z_1,z_2) = &\sum_{k=0}^{n-1} a_{2,n-k} (z_1) \, \sum_{j=0}^{m-1} a_{m, k+1} ( \xi_j z_2 ) - \sum_{k=0}^{n-2} a_{2,n-k-1} (z_1) \,  \sum_{j=0}^{m-1} a_{m, k+1} ( \xi_j z_2 ) \\
+ &\sum_{k=0}^{n-1} a_{2,n-k} (-z_1) \, \sum_{j=0}^{m-1} a_{m, k+1} ( \xi_j z_2 ) - \sum_{k=0}^{n-2} a_{2,n-k-1} (-z_1) \,  \sum_{j=0}^{m-1} a_{m, k+1} ( \xi_j z_2 ) \\
= & \, m \, \sum_{k=0}^{n-1} ( a_{2,n-k} (z_1) + a_{2,n-k} (-z_1) ) - m \sum_{k=0}^{n-2}  ( a_{2,n-k-1} (z_1) + a_{2,n-k-1} (-z_1) ) \\
= & \, 2mn - 2m (n-1) = 2m.
\end{split}
\]

\end{proof}

Further properties of the interpolatory subdivision schemes $\mathcal{S}_{\mathbf{a}_{M,n}}$ in Definition \ref{d:interp_anisotropic} are analyzed in subsections \ref{ssec:reproduction} (reproduction), \ref{ssec:minimality} (minimality of the support) and \ref{ssec:convergence} (convergence).

\subsection{Reproduction property of $\mathcal{S}_{\mathbf{a}_{M,n}}$}
\label{ssec:reproduction}

In this section, we show that the anisotropic interpolatory subdivision schemes $\mathcal{S}_{\mathbf{a}_{M,n}}$ in Definition \ref{d:interp_anisotropic} reproduce polynomials up to degree $2n-1$.

\begin{proposition} \label{p:anisotropic_repr}
Let $n \in \NN$ and $M$ in \eqref{eq:dilation}. The anisotropic interpolatory subdvision scheme $\mathcal{S}_{\mathbf{a}_{M,n}}$ in Definition \ref{d:interp_anisotropic} reproduces polynomials up to degree $2n-1$.
\end{proposition}

\begin{proof}
By \eqref{eq:symbol_factorization_reproduction_2D}, in order to prove Proposition \ref{p:anisotropic_repr}, we need to show that the symbol $a_{M,n}$ can be decomposed as
\[
\begin{split}
&a_{M,n} (z_1, z_2) = 2m + \sum_{h=0}^H (1-z_1)^{\alpha_h} \, (1-z_2)^{\beta_h} \, c_{M,n,h} (z_1,z_2), \qquad (z_1,z_2) \in ( \CC \setminus \set{0} )^2, \\
&\alpha_h, \, \beta_h \in \NN_0, \qquad \alpha_h + \beta_h \geq 2n, \qquad h = 0, \ldots, H,
\end{split}
\]
for some $H \in \NN$ and some suitable Laurent polynomials $c_{M,n,h}$, $h = 0, \ldots, H$. We recall that for any $k \in \NN$, the univariate $(2k)$-point Dubuc-Deslauriers interpolatory subdivision schemes $\mathcal{S}_{\mathbf{a}_{2,k}}$, $\mathcal{S}_{\mathbf{a}_{m,k}}$ in Definition \ref{d:DD_interp_1D} of dilation 2 and $m$, respectively, reproduce polynomials up to degree $2k-1$. Thus, from \eqref{eq:symbol_factorization_reproduction}, their symbols $a_{2,k}, \, a_{m,k}$ in \eqref{eq:DD_interp_1D} can be written as
\begin{equation} \label{eq:DD_factor}
\begin{split}
a_{2,k} (z) &= 2 +  (1-z)^{2k} \, b_{2,k} (z), \\
a_{m,k} (z) &= m +  (1-z)^{2k} \, b_{m,k} (z),\qquad z \in \CC \setminus \set{0},
\end{split}
\end{equation}
for suitable Laurent polynomials $b_{2,k}, \, b_{m,k}$.
By \eqref{eq:interp_anisotropic}, using factorization \eqref{eq:DD_factor}, there exist Laurent polynomials $ b_{2,n-k}, \, b_{m,k+1}$, $k = 0, \ldots, n-1$, such that
\[
\begin{split}
a_{M,n} (z_1,z_2) = &\sum_{k=0}^{n-1} a_{2,n-k} (z_1) \, a_{m,k+1} (z_2) - \sum_{k=0}^{n-2} a_{2, n-k-1} (z_1) \, a_{m,k+1} (z_2) \\
= &\sum_{k=0}^{n-1} \Bigl ( 2 + (1-z_1)^{2(n-k)} \, b_{2,n-k} (z_1) \Bigr ) \, \Bigl ( m +  (1-z_2)^{2(k+1)} \, b_{m,k+1} (z_2) \Bigr ) \\
- &\sum_{k=0}^{n-2}  \Bigl ( 2 + (1-z_1)^{2(n-k-1)} \, b_{2,n-k-1} (z_1) \Bigr ) \, \Bigl ( m +  (1-z_2)^{2(k+1)} \, b_{m,k+1} (z_2) \Bigr ) \\
= \, &2mn + m \sum_{k=0}^{n-1} (1-z_1)^{2(n-k)} \, b_{2,n-k} (z_1) + 2 \sum_{k=0}^{n-1} (1-z_2)^{2(k+1)} \, b_{m,k+1} (z_2) \\
+ &\sum_{k=0}^{n-1} (1-z_1)^{2(n-k)} (1-z_2)^{2(k+1)} \,  b_{2,n-k} (z_1) \, b_{m,k+1} (z_2) \\
- \, &2m(n-1)  -  m \sum_{k=0}^{n-2} (1-z_1)^{2(n-k-1)} \, b_{2,n-k-1} (z_1)  - 2 \sum_{k=0}^{n-2} (1-z_2)^{2(k+1)} \, b_{m,k+1} (z_2) \\
- &\sum_{k=0}^{n-2} (1-z_1)^{2(n-k-1)} (1-z_2)^{2(k+1)} \,  b_{2,n-k-1} (z_1) \, b_{m,k+1} (z_2) \\
= \, & 2m + m (1-z_1)^{2n} \, b_{2,n} (z_1) + 2 (1-z_2)^{2n} \, b_{m,n} (z_2) + \sum_{k=0}^{n-1} (1-z_1)^{2(n-k)} (1-z_2)^{2(k+1)} \,  b_{2,n-k} (z_1) \, b_{m,k+1} (z_2) \\
- & \sum_{k=0}^{n-2} (1-z_1)^{2(n-k-1)} (1-z_2)^{2(k+1)} \,  b_{2,n-k-1} (z_1) \, b_{m,k+1} (z_2) \\
= \, &2m + \sum_{h=0}^{2n} (1-z_1)^{\alpha_h} \, (1-z_2)^{\beta_h} \, c_{M,n,h} (z_1,z_2)
\end{split}
\]
where
\[
\alpha_h = 
\begin{cases}
2n & h = 0, \\
0 & h = 1, \\
2(n-h+2) & h \in \set{2, \ldots, n+1}, \\
2(2n-h+1) &  h \in \set{n+2, \ldots, 2n},
\end{cases} \qquad 
\beta_h = 
\begin{cases}
0 & h = 0, \\
2n & h = 1, \\
2(h-1) & h \in \set{2, \ldots, n+1}, \\
2(h-n-1) &  h \in \set{n+2, \ldots, 2n},
\end{cases}
\]
and
\[
 c_{M,n,h} (z_1,z_2) = 
\begin{cases}
m \, b_{2,n}(z_1) & h = 0, \\
2 \, b_{m,n}(z_2) & h = 1, \\
b_{2,n-h+2}(z_1) \, b_{m,h-1}(z_2) & h \in \set{2, \ldots, n+1}, \\
b_{2,2n-h+1}(z_1) \, b_{m,h-n-1}(z_2) & h \in \set{n+2, \ldots, 2n}.
\end{cases}
\]
The claim follows from
\[
\alpha_h  + \beta_h = 
\begin{cases}
2n & h \in \set{0,1,n+2, \ldots, 2n}, \\
2n+2 & h \in \set{2, \ldots, n+1}.
\end{cases}
\]
\end{proof}

\subsection{Minimality property of $\mathcal{S}_{\mathbf{a}_{M,n}}$}
\label{ssec:minimality}

In \cite{Han_Jia_1998optimal}, Ron and Jia constructed a family of interpolatory subdivision schemes with dilation matrix $2 I_2$ and minimal support.
The first aim of this section (see Proposition \ref{p:jia_interp_M}) is to generalize the result of Ron and Jia to our setting with dilation matrix $M$ in \eqref{eq:dilation}. Then, in Theorem \ref{p:Jia_our}, using Proposition \ref{p:jia_interp_M}, we show that Definition \ref{d:interp_anisotropic} provides a closed formula for the symbols of the minimally supported interpolatory subdivision schemes.

\begin{proposition} \label{p:jia_interp_M}
Let $n \in \NN$. There exists a unique interpolatory subdivision scheme with dilation matrix $M$ whose mask $\mathbf{c}_{M,n}$ satisfies
\begin{enumerate}
\item[(i)] $\mathbf{c}_{M,n}$ has support 
\[
\set{ (\alpha_1,\alpha_2) \in \ZZ^2 \, : \, m \abs{\alpha_1} + 2 \abs{\alpha_2} \leq 2mn-2+m} \subset \set{1-2n, \ldots, 2n-1} \times \set{1-mn, \ldots, mn-1},
\]
\item[(ii)] $\mathbf{c}_{M,n}$ is symmetric,
\item[(iii)] $\mathcal{S}_{\mathbf{c}_{M,n}}$ reproduces polynomials up to degree $2n-1$.
\end{enumerate}
\end{proposition}

Before proving Proposition \ref{p:jia_interp_M}, we present a constructive example in order to clarify the technical steps of the proof.

\begin{example}
Let $n=3$ and $M = \mathrm{ diag }(2,3)$ (thus $m=3$). We construct a mask $\mathbf{c}_{M,3}$ such that 
\[
\mathrm{c}_{M,3}(\alpha_1,\alpha_2) = 0, \qquad \forall  (\alpha_1,\alpha_2) \notin \set{ (\alpha_1,\alpha_2) \in \ZZ^2 \, : \, 3 \abs{\alpha_1} + 2 \abs{\alpha_2} \leq 19} \subset \set{-5, \ldots, 5} \times \set{-8, \ldots, 8},
\]
and the associated subdivision scheme $\mathcal{S}_{\mathbf{c}_{M,3}}$ with dilation $M$ is interpolatory, symmetric and reproduces polynomials up to degree $2n-1 = 5$. Notice that this size of the support is dictated by the desired polynomial reproduction property of the scheme we want to construct.

\emph{Step 1.} We fix the support of the mask $\mathbf{c}_{M,3}$ (unknown entries are denoted by *)
\[
\footnotesize
\left(
\begin{array}{ccccccccccccccccc}
0 & 0 & 0 & 0 & 0 & 0 & * & * & * & *  & * & 0 & 0 & 0 & 0 & 0 & 0 \\
0 & 0 & 0 & 0 & 0 & * & * & * & * & *  & * & * & 0 & 0 & 0 & 0 & 0 \\
0 & 0 & 0 & * & * & * & * & * & * & *  & * & * & * & * & 0 & 0 & 0 \\
0 & 0 & * & * & * & * & * & * & * & *  & * & * & * & * & * & 0 & 0 \\
* & * & * & * & * & * & * & * & * & *  & * & * & * & * & * & * & * \\
* & * & * & * & * & * & * & * & * & *  & * & * & * & * & * & * & * \\
* & * & * & * & * & * & * & * & * & *  & * & * & * & * & * & * & * \\
0 & 0 & * & * & * & * & * & * & * & *  & * & * & * & * & * & 0 & 0 \\
0 & 0 & 0 & * & * & * & * & * & * & *  & * & * & * & * & 0 & 0 & 0 \\
0 & 0 & 0 & 0 & 0 & * & * & * & * & *  & * & * & 0 & 0 & 0 & 0 & 0 \\
0 & 0 & 0 & 0 & 0 & 0 & * & * & * & *  & * & 0 & 0 & 0 & 0 & 0 & 0 \\
\end{array}
\right).
\]

\emph{Step 2.} We  impose the interpolatory conditions $\mathrm{c}_{M,3}(0,0)=1$, $\mathrm{c}_{M,3}(2 \alpha_1, 3 \alpha_2) = 0$, $\forall \bm{\alpha} = (\alpha_1,\alpha_2) \in \ZZ^2 \setminus \set{\bm{0}}$
\[
\footnotesize
\left(
\begin{array}{ccccccccccccccccc}
0 & 0 & 0 & 0 & 0 & 0 & * & * & * & *  & * & 0 & 0 & 0 & 0 & 0 & 0 \\
0 & 0 & 0 & 0 & 0 & 0 & * & * & 0 & *  & * & 0 & 0 & 0 & 0 & 0 & 0 \\
0 & 0 & 0 & * & * & * & * & * & * & *  & * & * & * & * & 0 & 0 & 0 \\
0 & 0 & 0 & * & * & 0 & * & * & 0 & *  & * & 0 & * & * & 0 & 0 & 0 \\
* & * & * & * & * & * & * & * & * & *  & * & * & * & * & * & * & * \\
* & * & 0 & * & * & 0 & * & * & 1 & *  & * & 0 & * & * & 0 & * & * \\
* & * & * & * & * & * & * & * & * & *  & * & * & * & * & * & * & * \\
0 & 0 & 0 & * & * & 0 & * & * & 0 & *  & * & 0 & * & * & 0 & 0 & 0 \\
0 & 0 & 0 & * & * & * & * & * & * & *  & * & * & * & * & 0 & 0 & 0 \\
0 & 0 & 0 & 0 & 0 & 0 & * & * & 0 & *  & * & 0 & 0 & 0 & 0 & 0 & 0 \\
0 & 0 & 0 & 0 & 0 & 0 & * & * & * & *  & * & 0 & 0 & 0 & 0 & 0 & 0 \\
\end{array}
\right).
\]

\emph{Step 3.} We define the remaining coefficients of $\mathbf{c}_{M,3}$ symmetrically and such that they guarantee the property of polynomial reproduction of polynomials up to degree $2n-1=5$. The latter condition leads to invertible systems of equations (one for each submask). They yield
\[
\footnotesize
\left(
\begin{array}{ccccccccccccccccc}
\smallskip
 0 & 0 & 0 & 0 & 0 & 0 & \frac{1}{256} & \frac{1}{128} & \frac{3}{256} & \frac{1}{128} & \frac{1}{256} & 0 & 0 & 0 & 0 & 0 & 0 \\
\smallskip
 0 & 0 & 0 & 0 & 0 & 0 & 0 & 0 & 0 & 0 & 0 & 0 & 0 & 0 & 0 & 0 & 0 \\
\smallskip
 0 & 0 & 0 & \frac{1}{324} & \frac{5}{1296} & 0 & -\frac{241}{6912} & -\frac{241}{3456} & -\frac{25}{256} & -\frac{241}{3456} & -\frac{241}{6912} & 0 & \frac{5}{1296} & \frac{1}{324} & 0 & 0 & 0 \\
\smallskip
 0 & 0 & 0 & 0 & 0 & 0 & 0 & 0 & 0 & 0 & 0 & 0 & 0 & 0 & 0 & 0 & 0 \\
\smallskip
 \frac{7}{1458} & \frac{4}{729} & 0 & -\frac{121}{2916} & -\frac{605}{11664} & 0 & \frac{20809}{93312} & \frac{20809}{46656} & \frac{75}{128} & \frac{20809}{46656} & \frac{20809}{93312} & 0 & -\frac{605}{11664} & -\frac{121}{2916} & 0 & \frac{4}{729} & \frac{7}{1458} \\
\smallskip
 \frac{7}{729} & \frac{8}{729} & 0 & -\frac{56}{729} & -\frac{70}{729} & 0 & \frac{280}{729} & \frac{560}{729} & 1 & \frac{560}{729} & \frac{280}{729} & 0 & -\frac{70}{729} & -\frac{56}{729} & 0 & \frac{8}{729} & \frac{7}{729} \\
\smallskip
 \frac{7}{1458} & \frac{4}{729} & 0 & -\frac{121}{2916} & -\frac{605}{11664} & 0 & \frac{20809}{93312} & \frac{20809}{46656} & \frac{75}{128} & \frac{20809}{46656} & \frac{20809}{93312} & 0 & -\frac{605}{11664} & -\frac{121}{2916} & 0 & \frac{4}{729} & \frac{7}{1458} \\
\smallskip
 0 & 0 & 0 & 0 & 0 & 0 & 0 & 0 & 0 & 0 & 0 & 0 & 0 & 0 & 0 & 0 & 0 \\
\smallskip
 0 & 0 & 0 & \frac{1}{324} & \frac{5}{1296} & 0 & -\frac{241}{6912} & -\frac{241}{3456} & -\frac{25}{256} & -\frac{241}{3456} & -\frac{241}{6912} & 0 & \frac{5}{1296} & \frac{1}{324} & 0 & 0 & 0 \\
\smallskip
 0 & 0 & 0 & 0 & 0 & 0 & 0 & 0 & 0 & 0 & 0 & 0 & 0 & 0 & 0 & 0 & 0 \\
 0 & 0 & 0 & 0 & 0 & 0 & \frac{1}{256} & \frac{1}{128} & \frac{3}{256} & \frac{1}{128} & \frac{1}{256} & 0 & 0 & 0 & 0 & 0 & 0 \\
\end{array}
\right)
\]
We notice that the main column and row of $\mathbf{c}_{M,3}$ are the univariate binary $\mathbf{a}_{2,3}$ and ternary $\mathbf{a}_{3,3}$ $6$-point Dubuc-Deslauriers masks defined in \eqref{eq:DD_interp_1D}, respectively.
\end{example}

\begin{proof}[Proof of Proposition \ref{p:jia_interp_M}]
Recall by \eqref{eq:Gamma} that
\[
\Gamma = \set{(k,j) \in \NN_0^2 \, : \, k \in \set{0, 1}, \, j \in \set{ 0, \ldots, m-1}}
\]
is a complete set of representatives of the distinct cosets of $\ZZ^2 / M \ZZ^2$.
Every interpolatory mask $\mathbf{a}$ reproduces polynomials up to degree $2n-1$ if and only if it satisfies the sum rules of order $2n$, i.e. by \eqref{eq:sum_rules}
\begin{equation} \label{eq:sum_rules_interp}
\begin{split}
&\sum_{\bm{\alpha} \in \ZZ^2} \, \mathrm{a} \bigl ( k + 2 \alpha_1, j + m \alpha_2 \bigr ) \bigl (  k + 2 \alpha_1 \bigr )^{\mu_1} \, \bigl ( j + m \alpha_2 \bigr )^{\mu_2} = \delta_{\bm{\mu},0}, \\
& (k, j) \in \Gamma \setminus \set{(0,0)}, \qquad \bm{\mu} = (\mu_1, \mu_2) \in \NN_0^2 \, : \, \mu_1 + \mu_2 \leq 2n-1.
\end{split}
\end{equation}
Notice that ($(k,j) = (0,0)$)
\[
\sum_{\bm{\alpha} \in \ZZ^2} \, \mathrm{a} \bigl (2 \alpha_1, m \alpha_2 \bigr ) \bigl ( 2 \alpha_1 \bigr )^{\mu_1} \, \bigl ( m \alpha_2 \bigr )^{\mu_2} = \delta_{\bm{\mu},0}, \qquad \bm{\mu} = (\mu_1, \mu_2) \in \NN_0^2 \, : \, \mu_1 + \mu_2 \leq 2n-1,
\]
is due to the interpolatory property \eqref{eq:mask_interp} of $\mathbf{a}$.
The construction of the mask $\mathbf{c}_{M,n}$ is split in 3 Steps.

\emph{Step 1 (support size).} We set $\mathrm{c}_{M,n} (\alpha_1, \alpha_2) = 0$, $\forall \bm{\alpha} = (\alpha_1, \alpha_2) \in \ZZ^2$ such that 
\[
\abs{\alpha_1} > 2n-1, \qquad \abs{\alpha_2} > mn-1, \qquad m \abs{\alpha_1} + 2 \abs{\alpha_2} > 2mn-2+m.
\]
Thus, condition \emph{(i)} is satisfied. 

\emph{Step 2 (interpolation).} We impose the interpolatory conditions
\[
\mathrm{c}_{M,n}(0,0)=1, \qquad \mathrm{c}_{M,n}(2 \alpha_1, m \alpha_2) = 0, \qquad \forall \bm{\alpha} = (\alpha_1,\alpha_2) \in \ZZ^2 \setminus \set{\bm{0}}. 
\]

\emph{Step 3 (symmetry and reproduction).}  The system of equations in \eqref{eq:sum_rules_interp} naturally splits into $\# \Gamma -1 = 2m - 1$ separate linear systems of equations one for each $(k,j) \in \Gamma' = \Gamma \setminus \set{(0,0)}$.
Imposing the symmetry of $\mathrm{c}_{M,n}$, the number of unknowns in the systems of equations \eqref{eq:sum_rules_interp} for $(k,j) \in \Gamma' \setminus \set{(1,0)}$ can be halfed. The corresponding $n(n+1) \times n(n+1)$ system matrices are invertible, which we prove in \emph{Step 3.a}.
We treat the case $(k,j)=(1,0)$ separately, due to the special symmetry of the corresponding submask. This case is analyzed in \emph{Step 3.b} and the corresponding ${ \displaystyle \frac{n(n+1)}{2} \times\frac{n(n+1)}{2}}$ system matrix is also invertible.

\noindent{\emph{Step 3.a.}} Let $(k,j) \in \Gamma' \setminus \set{(1,0)}$. 

Symmetry in $\alpha_1$: we only need to determine the coefficients $\mathrm{c}_{M,n}(k+ 2\alpha_1, j + m \alpha_2)$ for $\bm{\alpha} = (\alpha_1, \alpha_2) \in \NN_0 \times \ZZ$ and such that \emph{(ii)} is satisfied, i.e.
\begin{equation} \label{eq:set_A}
0 \leq k + 2 \alpha_1 \leq 2n-1, \quad 0 \leq \abs{j + m \alpha_2} \leq mn-1, \quad 0 \leq m (k + 2 \alpha_1) + 2 \abs{j + m \alpha_2} \leq 2mn -2+m.
\end{equation}
We call $\mathcal{A}$ the set of such indices. We first want to determine the geometric structure and the cardinality of $\mathcal{A}$.
First, we focus our attention on the inequality $0 \leq k + 2 \alpha_1 \leq 2n-1$ in \eqref{eq:set_A}. Since $k \in \set{0,1}$, we have
\[ 
0 \leq k + 2 \alpha_1 \leq 2n-1 \quad \iff \quad 0 \leq \alpha_1 \leq n - \bigg \lceil \frac{1+k}{2} \bigg \rceil \quad \iff \quad  0 \leq \alpha_1 \leq n-1.
\]
Then, we focus our attention on the inequality $ 0 \leq \abs{j + m \alpha_2} \leq mn-1$ in \eqref{eq:set_A}. We observe that $j + m \alpha_2 \geq 0 \iff \alpha_2 \geq 0$, thus for every $j \in \set{1, \ldots, m-1}$ we have
\[ 
\begin{array}{cccccc}
\bullet \text{ for } \alpha_2 \geq 0 \, : & 0 \leq j + m \alpha_2 \leq mn-1 & \iff & 0 \leq \alpha_2 \leq n - \bigg \lceil \frac{1+j}{m} \bigg \rceil  & \iff &  0 \leq \alpha_2 \leq n-1, \\
\bullet \text{ for } \alpha_2 < 0 \, : & 0 < - j - m \alpha_2 \leq mn-1 & \iff & 0 < - \alpha_2 \leq n + \bigg \lfloor \frac{j-1}{m} \bigg \rfloor  & \iff &  1 \leq -\alpha_2 \leq n.
\end{array}
\]
Finally, we focus our attention on the last inequality $0 \leq m (k + 2 \alpha_1) + 2 \abs{j + m \alpha_2} \leq 2mn -2+m$ in \eqref{eq:set_A}. Let $\alpha_1 \in \set{0, \ldots, n-1}$. Then, we have
\[
\begin{array}{cccl}
\bullet \text{ for } \alpha_2 \geq 0 \, : & 0 \leq m (k + 2 \alpha_1) + 2 (j + m \alpha_2) \leq 2mn -2+m & \iff & 0 \leq \alpha_2 \leq n - \alpha_1 - \bigg \lceil \frac{1+j}{m} \bigg \rceil + \bigg \lfloor \frac{1-k}{2} \bigg \rfloor, \\
& & \iff & 0 \leq \alpha_2 \leq n - \alpha_1 - 1, \\
\bullet \text{ for } \alpha_2 < 0 \, : & 0 < m (k + 2 \alpha_1) - 2 (j + m \alpha_2) \leq 2mn -2+m & \iff & 0 < -\alpha_2 \leq n - \alpha_1 + \bigg \lfloor \frac{j-1}{m} \bigg \rfloor + \bigg \lfloor \frac{1-k}{2} \bigg \rfloor, \\
& & \iff & 1 \leq - \alpha_2 \leq n - \alpha_1. \\
\end{array}
\]
Combining the above observations, we get
\[
\mathcal{A} = \Set{\bm{\alpha} = (\alpha_1, \alpha_2) \in \NN_0 \times \ZZ \, : \, 0 \leq \alpha_1 \leq n-1, \, \alpha_1 - n \leq \alpha_2 \leq n - \alpha_1 - 1}.
\]
Thus, the cardinality of the set $\mathcal{A}$ is
\[
\# \mathcal{A} = \sum_{\alpha_1=0}^{n-1} 2 (n - \alpha_1) = n (n+1),
\]
i.e. the number of unknowns is $n (n+1)$.
Moreover, \eqref{eq:sum_rules_interp} is automatically satisfied for odd $\mu_1$. Therefore, we solve
\begin{equation} \label{eq:linear_sys}
\sum_{ \bm{\alpha} \in \mathcal{A}} \, \mathrm{a} \bigl ( k + 2 \alpha_1, j + m \alpha_2 \bigr ) \bigl (  k + 2 \alpha_1 \bigr )^{2 \mu_1} \, \bigl (  j + m \alpha_2 \bigr )^{\mu_2} = \frac12 \delta_{\bm{\mu},\bm{0}} + \frac12 \delta_{\bm{\mu},\bm{0}} \delta_{k,0} \delta_{\alpha_1,0} , \qquad \bm{\mu} \in \mathcal{M},
\end{equation}
with $\mathcal{M} = \set{(\mu_1, \mu_2) \in \NN_0^2 \, : \, 0 \leq 2\mu_1 + \mu_2 \leq 2n-1}$.
We notice that $\# \mathcal{M} = \# \mathcal{A} = n(n+1)$, i.e. the corresponding system matrix is indeed a square matrix.

Symmetry in $\alpha_2$: it allows us to reduce the total number of linear systems in \eqref{eq:linear_sys}. Note that the systems in \eqref{eq:linear_sys} for $j \in \set{1, \ldots, \frac{m-1}{2}}$ and $m-j \in \set{\frac{m+1}{2}, \ldots, m-1}$ are equivalent, indeed
\[
\begin{split}
\sum_{ \bm{\alpha} \in \mathcal{A}} \,  \mathrm{a} \bigl ( k + 2 \alpha_1, \,& m-j + m \alpha_2 \bigr ) \bigl (  k + 2 \alpha_1 \bigr )^{2 \mu_1} \, \bigl ( m- j + m \alpha_2 \bigr )^{\mu_2} \\
&= \sum_{ \bm{\alpha} \in \mathcal{A}} \, \mathrm{a} \bigl ( k + 2 \alpha_1, -j + m (\alpha_2+1) \bigr ) \bigl (  k + 2 \alpha_1 \bigr )^{2 \mu_1} \, \bigl ( - j + m (\alpha_2+1) \bigr )^{\mu_2} \\
&= \sum_{ \bm{\beta} \in \mathcal{B}} \, \mathrm{a} \bigl ( k + 2 \beta_1,- (j + m \beta_2) \bigr ) \bigl (  k + 2 \beta_1 \bigr )^{2 \mu_1} \, \bigl ( -( j + m \beta_2) \bigr )^{\mu_2}, \qquad (\beta_1,\beta_2) = (\alpha_1,-\alpha_2-1), \\
&= (-1)^{\mu_2} \, \sum_{ \bm{\beta} \in \mathcal{B}} \, \mathrm{a} \bigl ( k + 2 \beta_1,j + m \beta_2 \bigr ) \bigl (  k + 2 \beta_1 \bigr )^{2 \mu_1} \, \bigl (  j + m \beta_2 \bigr )^{\mu_2},
\end{split}
\]
where
\[
\mathcal{B} = \Set{\bm{\beta} = (\beta_1, \beta_2) \in \NN_0 \times \ZZ \, : \, 0 \leq \beta_1 \leq n-1, \, \beta_1 - n \leq \beta_2 \leq n - \beta_1 - 1}.
\]
Thus, we only need to consider the case $j \in \set{1, \ldots, \frac{m-1}{2}}$. The corresponding square matrix
\[
\Biggl ( (  k + 2 \alpha_1 \bigr )^{2 \mu_1} \, (  j + m \alpha_2 \bigr )^{\mu_2} \Biggr )_{ (\alpha_1, \alpha_2) \in \mathcal{A}, \, (\mu_1, \mu_2) \in \mathcal{M}}
\]
is non-singular \cite[Theorem 3.3]{sauer2004lagrange}. Therefore, for any $j \in \set{1, \ldots, \frac{m-1}{2}}$, the linear system of equations \eqref{eq:linear_sys} is uniquely solvable and its solution is
\[
\Bigl ( \mathrm{c}_{M,n} (k + 2 \alpha_1, j + m \alpha_2 ) \Bigr )_{(\alpha_1, \alpha_2) \in \mathcal{A}}.
\]

\noindent \emph{Step 3.b.} Let $(k,j) = (1,0)$. Due to the symmetry in $\alpha_1$ and $\alpha_2$, \eqref{eq:sum_rules_interp} reduces to the following system of equations
\begin{equation} \label{eq:lin_sys_10}
\sum_{ \bm{\alpha} \in \mathcal{A}'} \, \mathrm{a} \bigl ( 1 + 2 \alpha_1, m \alpha_2 \bigr ) \bigl (  1 + 2 \alpha_1 \bigr )^{2 \mu_1} \, \bigl ( m \alpha_2 \bigr )^{2 \mu_2} = \frac14 \delta_{\bm{\mu},\bm{0}} + \frac14 \delta_{\bm{\mu},\bm{0}} \delta_{\alpha_2,0}, \qquad \bm{\mu} \in \mathcal{M}',
\end{equation}
where
\[
\begin{split}
\mathcal{A}' &= \set{ (\alpha_1, \alpha_2) \in \NN_0^2 \, : \,  0 \leq \alpha_1 \leq n-1, \,  0 \leq \alpha_2 \leq n - \alpha_1 - 1}, \\
\mathcal{M}' &= \set{(\mu_1, \mu_2) \in \NN_0^2 \, : \, 0 \leq 2\mu_1 + 2 \mu_2 \leq 2n-1}.
\end{split}
\]
We notice that $ \# \mathcal{A}' =  \# \mathcal{M}' = n(n+1)/2$. By \cite[Lemma 4.1]{Han_Jia_1998optimal}, the square matrix
\[
\Biggl ( (  1 + 2 \alpha_1 \bigr )^{2 \mu_1} \, ( m \alpha_2 \bigr )^{2 \mu_2} \Biggr )_{(\alpha_1, \alpha_2) \in \mathcal{A}', \, (\mu_1, \mu_2) \in \mathcal{M}'}
\]
is non-singular. Therefore, the linear system of equations \eqref{eq:lin_sys_10} is uniquely solvable and its solution is
\[
\Bigl ( \mathrm{c}_{M,n} (1 + 2 \alpha_1, m \alpha_2 ) \Bigr )_{(\alpha_1, \alpha_2) \in \mathcal{A}'}.
\]
\end{proof}

We notice another special property of the masks $\mathrm{c}_{M,n}$ in Proposition \ref{p:jia_interp_M}.

\begin{remark}
Let $n \in \NN$. Since the mask $\mathbf{a}_{2,n}$ of the univariate binary $(2n)$-point Dubuc-Deslauriers interpolatory subdivision scheme $\mathcal{S}_{\mathbf{a}_{2,n}}$  in Definition \ref{d:DD_interp_1D} satisfies the sum rules of order $2n$, the solution of \eqref{eq:lin_sys_10} is given by
\[
\mathrm{c}_{M,n} (1+2 \alpha_1, m \alpha_2) = 
\begin{cases}
\mathrm{a}_{2,n} (1 + 2 \alpha_1), & \alpha_2 = 0, \\
0, & \alpha_2 \neq 0,
\end{cases} \qquad \bm{\alpha} = (\alpha_1,\alpha_2) \in \mathcal{A}'.
\]
Analogously, the solution of \eqref{eq:linear_sys} for $(0,j) \in \Gamma'$, $j \in \set{1, \ldots, m-1}$, is given by
\[
\mathrm{c}_{M,n} (2 \alpha_1, j + m \alpha_2) = 
\begin{cases}
\mathrm{a}_{m,n} ( j + m \alpha_2), & \alpha_1 = 0, \\
0, & \alpha_1 \neq 0,
\end{cases}
\qquad \bm{\alpha} = (\alpha_1,\alpha_2) \in \mathcal{A},
\]
where $\mathbf{a}_{m,n}$ is the mask of the univariate $m$-arity $(2n)$-point Dubuc-Deslauriers interpolatory subdivision scheme $\mathcal{S}_{\mathbf{a}_{m,n}}$ in Definition \ref{d:DD_interp_1D}. 
\end{remark}

We now show that the masks in Definition \ref{d:interp_anisotropic} and the ones obtained in Proposition \ref{p:jia_interp_M} actually coincide.

\begin{theorem} \label{p:Jia_our}
Let $n \in \NN$ and $M$ in \eqref{eq:dilation}. The mask $\mathbf{a}_{M,n}$ of the anisotropic interpolatory subdvision scheme $\mathcal{S}_{\mathbf{a}_{M,n}}$ in Definition \ref{d:interp_anisotropic} satisfies
\begin{enumerate}
\item[(i)] $\mathbf{a}_{M,n}$ has support 
\[
\set{ (\alpha_1,\alpha_2) \in \ZZ^2 \, : \, m \abs{\alpha_1} + 2 \abs{\alpha_2} \leq 2mn-2+m} \subset \set{1-2n, \ldots, 2n-1} \times \set{1-mn, \ldots, mn-1},
\]
\item[(ii)] $\mathbf{a}_{M,n}$ is symmetric,
\item[(iii)] $\mathcal{S}_{\mathbf{a}_{M,n}}$ reproduces polynomials up to degree $2n-1$.
\end{enumerate}
\end{theorem}

\begin{proof}
\noindent \emph{Step 1.} Condition \emph{(ii)} follows directly from Definition \ref{d:interp_anisotropic} and from the symmetry of the univariate masks $\mathbf{a}_{2,h}, \, \mathbf{a}_{m,h}$, $h=1, \ldots, n$, in Definition \ref{d:DD_interp_1D}. 

\noindent \emph{Step 2.} Condition \emph{(iii)} follows directly from Proposition \ref{p:anisotropic_repr}.

\noindent \emph{Step 3.} We focus our attention on condition \emph{(i)}. Let $h \in \set{1, \ldots, n}$. The univariate masks $\mathbf{a}_{2,h}, \, \mathbf{a}_{m,h}$, in Definition \ref{d:DD_interp_1D} have supports $\set{1-2h, \ldots, 2h-1}$ and $\set{1-mh, \ldots, mh-1}$, respectively. Thus, for $k = 0, \ldots, n-1$, the masks associated to the symbols 
\[
a_{2,n-k} (z_1) \, a_{m,k+1} (z_2), \qquad (z_1,z_2) \in (\CC \setminus \set{0})^2,
\]
of the first sum in \eqref{eq:interp_anisotropic} have support $S_k = \set{1-2(n-k), \ldots, 2(n-k)-1} \times \set{1-m(k+1), \ldots, m(k+1)-1}$. Moreover,  for $k = 0, \ldots, n-1$, we have
\[
m \abs{\alpha_1} + 2 \abs{\alpha_2} \leq  m \bigl ( 2(n-k) - 1 \bigr ) + 2 \bigl ( m(k+1) - 1 \bigr) = 2mn - 2 + m, \qquad (\alpha_1,\alpha_2) \in S_k.
\]
Thus, the mask associated to the symbol
\[
\sum_{k=0}^{n-1} a_{2,n-k} (z_1) \, a_{m,k+1} (z_2), \qquad (z_1,z_2) \in (\CC \setminus \set{0})^2,
\]
 in \eqref{eq:interp_anisotropic} has support
\[
\mathcal{A} = \set{ (\alpha_1,\alpha_2) \in \ZZ^2 \, : \, m \abs{\alpha_1} + 2 \abs{\alpha_2} \leq 2mn-2+m} \subset \set{1-2n, \ldots, 2n-1} \times \set{1-mn, \ldots, mn-1}.
\]
Using the same argument, the support of the mask associated to the symbol
\[
\sum_{k=0}^{n-2} a_{2, n-k-1} (z_1) \, a_{m,k+1} (z_2), \qquad (z_1,z_2) \in ( \CC \setminus \set{0})^2,
\]
 in \eqref{eq:interp_anisotropic} is contained in $\mathcal{A}$, so that the support of the mask $\mathbf{a}_{M,n}$ satisfies \emph{(i)}.
\end{proof}

\subsection{Convergence of certain $\mathcal{S}_{\mathbf{a}_{M,n}}$}
\label{ssec:convergence}

In this section, we only analyze convergence of the schemes used in section \ref{sec:examples}.
In \cite{Charina2017smoothness}, Charina and Protasov presented a detailed regularity analysis of $d$-variate anisotropic subdivision schemes. Especially, their results allow us to use the algorithm in \cite{Guglielmi2013exact} for the exact computation of the H{\"o}lder regularity of an anisotropic subdivision scheme.

\begin{definition}
A convergent subdivision scheme $S_{\mathbf{p}}$  with dilation $M$ and mask $\mathbf{p}$ has H{\"o}lder regularity $\alpha_{\phi}  \in (0,1]$ if its basic limit function $\phi$ has H{\"o}lder exponent $\alpha_{\phi}$, namely
\[
\alpha_{\phi} = \sup \set{\alpha \in (0,1] \, : \, \norm{ \phi( \cdot + \mathbf{h}) - \phi }_{C (\RR^d)} \leq C \norm{\mathbf{h}}^{\alpha}, \, \mathbf{h} \in \RR^d}.
\]
\end{definition}

\noindent The main ingredient of the regularity analysis in \cite{Charina2017smoothness} is the so-called \emph{joint spectral radius} \cite{Rota1960note}, which is a generalization of the classical notion of spectral radius of one square matrix to a compact set of square matrices. 

\begin{definition} \label{d:jsr}
Let $\mathcal{V} = \set{V_1, \ldots, V_d} \subset \RR^{N \times N}$, $N \in \NN$, be a finite set of square matrices. The joint spectral radius of $\mathcal{V}$ is defined by
\[
\rho (\mathcal{V}) = \lim_{k \to \infty} \, \max \Set{ \norm{ V_{j_1} \cdots V_{j_k}}^{1/k} \, : \, V_{j_i} \in \mathcal{V}, \, i = 1, \ldots, k}.
\]
\end{definition} 

\noindent The limit in Definition \ref{d:jsr} always exists and does not depend on the choice of the matrix norm (\cite{Rota1960note}). For practical interest (see Section \ref{sec:examples}), we check the continuity and compute the H{\"o}lder regularity of some elements of the family $\mathcal{S}_{\mathbf{a}_{M,n}}$ with $M = {\displaystyle \begin{pmatrix} 2 & 0 \\ 0 & m \end{pmatrix}}$, $m = 3,5$. To do so, we first define the set $\mathcal{V}$. The size of the elements of $\mathcal{V}$ depends on the support of the basic limit function and the cardinality of the set $\Omega$ in \eqref{eq:Omega}.

Let $n \in \NN$ and $\phi$ be the basic limit function of the anisotropic interpolatory subdvision schemes $\mathcal{S}_{\mathbf{a}_{M,n}}$ in Definition \ref{d:interp_anisotropic}. 
By \cite[Proposition 2.2 and (2.7)-(2.8)]{Cabrelli2004self}, we have
\begin{equation} \label{eq:set_K}
\text{ supp } \phi \subseteq  K = \Set{ \mathbf{x} \in \RR^2 \, : \, \mathbf{x} = \sum_{j=1}^{\infty} M^{-j} \bm{\alpha}_{k_j}, \, \bm{\alpha}_{k_j} \in \text{ supp } \mathbf{a}_{M,n} }\subset \RR^2. 
\end{equation}
Since $\phi$ is compactly supported, $K$ is a compact set. Thus, due to \cite[Lemma 2.3]{Cabrelli2004self}, there exists a minimal set $\Omega \subset \ZZ^2$ such that
\begin{equation} \label{eq:Omega}
K \subset \Omega + [0,1]^2 = \bigcup_{\omega \in \Omega} (\omega + [0,1]^2).
\end{equation}
The minimality of $\Omega$ reads as follows: if there exists $\tilde{\Omega} \subset \ZZ^2$ such that $K \subset \tilde{\Omega} + [0,1]^2$, then $\tilde{\Omega} \supseteq \Omega$. We refer to \cite{Cabrelli2004self} for more details.

Let $N = \# \Omega$. By \eqref{eq:Gamma}, $\Gamma = \set{(k,j) \in \NN_0^2 \, : \, k \in \set{0, 1}, \, j \in \set{0, \ldots, m-1}}$ is a complete set of representatives of the distinct cosets of $\ZZ^2 / M \ZZ^2$. Notice that $\# \Gamma = 2m$.
For every $\bm{\gamma} \in \Gamma$, we define the transition matrix
\[
T_{\bm{\gamma}} = \Bigl ( \mathrm{a}_{M,n} (M \bm{\alpha} - \bm{\beta} + \bm{\gamma} ) \Bigr )_{\bm{\alpha}, \bm{\beta} \in \Omega}.
\]
We denote $\mathcal{T} = \set{T_{\bm{\gamma}} \, : \, \bm{\gamma} \in \Gamma}$ the set of all the transition matrices. Notice that $\# \mathcal{T} = 2m$. 
For every $\bm{\gamma} \in \Gamma$, the rows and columns of $T_{\bm{\gamma}}$ are enumerated by the elements from the set $\Omega$, thus $T_{\bm{\gamma}} \in \RR^{N \times N}$. By construction, the entries of any column of $T_{\bm{\gamma}}$ sum up to 1, thus $T_{\bm{\gamma}}$ has eigenvalue 1 (i.e. there exist $\mathbf{v}_0 \in \RR^N$ such that $T_{\bm{\gamma}} \mathbf{v}_0 = \mathbf{v}_0$). This property of $\mathcal{T}$ implies (\cite{Cabrelli2004self, Jia1996subdivision}) the existence of certain invariant subspaces of $\mathcal{T}$ crucial for the definition of the set $\mathcal{V}$. To determine these invariant subspaces, we define the vector-valued function
\begin{equation} \label{eq:phi_vector}
\begin{array}{ccl}
v \colon [0,1]^2 & \to &\RR^N, \\
\mathbf{x} & \mapsto & \Bigl ( \phi ( \mathbf{x} + \omega) \Bigr )_{\omega \in \Omega}.
\end{array}
\end{equation}
Now we are able to define the following subspaces of $\RR^N$
\begin{equation} \label{eq:spaces_U_U1_U2}
\begin{split}
U     &= \text{ span } \Set{v(\mathbf{y}) - v(\mathbf{x}) \, : \, \mathbf{x},\mathbf{y} \in [0,1]^2 }, \\
U_1 &= \text{ span } \Set{v(\mathbf{y}) - v(\mathbf{x}) \, : \, \mathbf{x},\mathbf{y} \in [0,1]^2, \, \mathbf{y}-\mathbf{x} = (\alpha, 0), \, \alpha \in \RR}, \\
U_2 &= \text{ span } \Set{v(\mathbf{y}) - v(\mathbf{x}) \, : \, \mathbf{x},\mathbf{y} \in [0,1]^2, \, \mathbf{y}-\mathbf{x} = (0, \beta), \, \beta \in \RR}, \\
\end{split}
\end{equation}
invariant under $\mathcal{T}$.
Notice that $U_1, \, U_2$ contain differences in the directions of the eigenvectors of $M$.

Finally, we define
\[
\mathcal{V} = \mathcal{T} |_{U} = \Set{T_{\bm{\gamma}} |_U \, : \, \bm{\gamma} \in \Gamma}, \qquad \mathcal{V}_1 = \mathcal{V} |_{U_1}, \qquad \mathcal{V}_2 = \mathcal{V} |_{U_2}.
\]
The following statement is a direct consequence of Theorem 1 in \cite{Charina2017smoothness}.

\begin{theorem} \label{t:holder_regularity}
Let $n \in \NN$. The basic limit function $\phi$ of the anisotropic interpolatory subdvision scheme $\mathcal{S}_{\mathbf{a}_{M,n}}$ in Definition \ref{d:interp_anisotropic} belongs to $C(\RR^2)$ if and only if $\rho (\mathcal{V}) < 1$. In this case, $\phi$ has the H{\"o}lder exponent
\[
\alpha_{\phi} = \min \Set{ \log_{1/2} \, \rho (\mathcal{V}_1 ), \, \log_{1/m} \, \rho (\mathcal{V}_2 )}.
\]
\end{theorem}

\noindent In order to properly end this section, we would like to answer a few questions which naturally arise from reading of the above analysis:
\begin{itemize}
\item[\emph{Q1.}] How to determine the spaces $U, \, U_1$ and $U_2$ ($\phi$ is not known analytically)?
\item[\emph{Q2.}] How to determine the sets $\mathcal{V}, \mathcal{V}_1$ and $\mathcal{V}_2$?
\end{itemize}
The questions \emph{Q1.} and \emph{Q2.} will be answered in the following Example.

\begin{example} \label{ex:holder}
Let $n=1$ and $M = {\displaystyle \begin{pmatrix} 2 & 0 \\ 0 & 3 \end{pmatrix}}$. The anisotropic interpolatory subdvision scheme $\mathcal{S}_{\mathbf{a}_{M,1}}$ in Definition \ref{d:interp_anisotropic} has the mask
\[
\mathbf{a}_{M,1} = \frac16 \, 
\left(
\begin{array}{ccccc}
 1 & 2 & 3 & 2 & 1 \\
 2 & 4 & 6 & 4 & 2 \\
 1 & 2 & 3 & 2 & 1 \\
\end{array}
\right).
\]
Since the support of the basic limit function is a subset of $[-1,1]^2$ (see Figure \ref{fig:supp_phi} and \eqref{eq:set_K}), we determine $\Omega = \set{-1,0}^2$, $N = \# \Omega = 4$.
\begin{figure}
\centering
\includegraphics[scale =0.6,trim={1cm 0.5cm 1cm 0.5cm},clip]{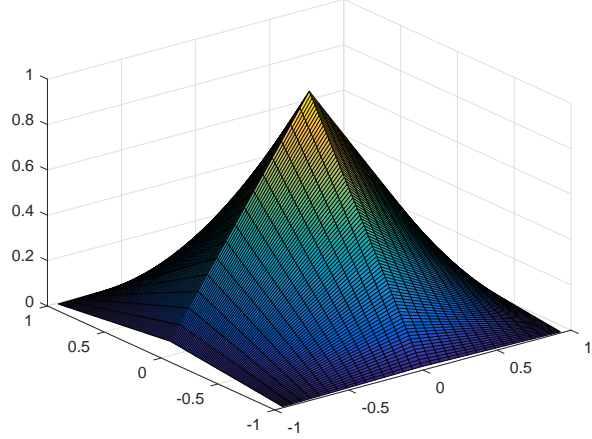}
\caption{Basic limit function of the anisotropic interpolatory subdvision schemes $\mathcal{S}_{\mathbf{a}_{M,1}}$, $M= \mathrm{ diag }(2,3)$.}
\label{fig:supp_phi}
\end{figure} 
For 
\[
\Gamma = \Set{
\begin{pmatrix} 0 \\ 0 \end{pmatrix},
\begin{pmatrix} 0 \\ 1 \end{pmatrix},
\begin{pmatrix} 0 \\ 2 \end{pmatrix},
\begin{pmatrix} 1 \\ 0 \end{pmatrix},
\begin{pmatrix} 1 \\ 1 \end{pmatrix},
\begin{pmatrix} 1 \\ 2 \end{pmatrix}}, \qquad \# \Gamma = 6,
\]
the corresponding transition matrices are
\[
\begin{split}
&T_{(0,0)} = \frac16
\left(
\begin{array}{cccc}
 1 & 0 & 0 & 0 \\
 2 & 3 & 0 & 0 \\
 1 & 0 & 2 & 0 \\
 2 & 3 & 4 & 6 \\
\end{array}
\right), \qquad
T_{(0,1)} = \frac16
\left(
\begin{array}{cccc}
 2 & 1 & 0 & 0 \\
 1 & 2 & 0 & 0 \\
 2 & 1 & 4 & 2 \\
 1 & 2 & 2 & 4 \\
\end{array}
\right), \qquad
T_{(0,2)} = \frac16
\left(
\begin{array}{cccc}
 3 & 2 & 0 & 0 \\
 0 & 1 & 0 & 0 \\
 3 & 2 & 6 & 4 \\
 0 & 1 & 0 & 2 \\
\end{array}
\right), \\
&T_{(1,0)} = \frac16
\left(
\begin{array}{cccc}
 2 & 0 & 1 & 0 \\
 4 & 6 & 2 & 3 \\
 0 & 0 & 1 & 0 \\
 0 & 0 & 2 & 3 \\
\end{array}
\right), \qquad
T_{(1,1)} = \frac16
\left(
\begin{array}{cccc}
 4 & 2 & 2 & 1 \\
 2 & 4 & 1 & 2 \\
 0 & 0 & 2 & 1 \\
 0 & 0 & 1 & 2 \\
\end{array}
\right), \qquad
T_{(1,2)} = \frac16
\left(
\begin{array}{cccc}
 6 & 4 & 3 & 2 \\
 0 & 2 & 0 & 1 \\
 0 & 0 & 3 & 2 \\
 0 & 0 & 0 & 1 \\
\end{array}
\right).
\end{split}
\]

Let us compute the spaces $U, \, U_1$ and $U_2$.

\noindent \emph{Space $U$:} the transition matrix $T_{(0,0)}$ has eigenvalue 1 with respect to the eigenvector $\mathbf{v}_0 = (0,0,0,1)^T$. To determine $U$, we proceed as follows.
\begin{itemize}
\item[\emph{Step 1.}] We define $W^{(1)} = \text{ span } \Set{T_{\bm{\gamma}} \mathbf{v}_0 -  \mathbf{v}_0 \, : \, \bm{\gamma} \in \Gamma \setminus \set{\bm{0}}}$.
\item[\emph{Step 2.}] We compute recursively
\[
W^{(k+1)} = W^{(k)} \cup \text{ span } \Set{T_{\bm{\gamma}} \mathbf{w}^{(k)} \, : \,  \mathbf{w}^{(k)} \in W^{(k)}, \, \bm{\gamma} \in \Gamma }, \quad 1 \leq k < N-1
\] 
until $\text{ dim } (W^{(k)} )< \text{ dim } ( W^{(k+1)} )$.
\item[\emph{Step 3.}] We define $U = W^{(k)}$.
\end{itemize}

\noindent Notice that the constraint $1 \leq k < N-1$ makes sense since 
\[
U \subseteq W = \set{\mathbf{w} = (w_1, \ldots, w_N) \in \RR^N \, : \, \sum_{i=1}^N w_i = 0} = \set{\mathbf{w} \in \RR^N \, : \, \mathbf{w} \perp \mathbf{1}}, \qquad \text{ dim } W=N-1.
\]

\noindent In our case, 
\[
U = \text{ span } \Set{T_{(0,1)} \mathbf{v}_0 -  \mathbf{v}_0, T_{(1,0)} \mathbf{v}_0 -  \mathbf{v}_0, T_{(1,1)} \mathbf{v}_0 -  \mathbf{v}_0} = \text{ span } \Set{
\left (
\begin{array}{r}
1 \\ 0 \\ 0 \\ -1
\end{array}
\right ),
\left (
\begin{array}{r}
0 \\ 1 \\ 0 \\ -1 
\end{array}
\right ),
\left (
\begin{array}{r}0 \\ 0 \\ 1 \\ -1
\end{array}
\right )
}, \qquad \text{ dim } U = 3.
\] 

\noindent \emph{Space $U_1$:} In order to determine $U_1$, we use the algorithm above with a different starting vector. By \eqref{eq:spaces_U_U1_U2}, we compute $\mathbf{v}_0 = v((1,0)) - v((0,0))$, $\norm{\mathbf{v}_0}_1 = 1$. By definition \eqref{eq:phi_vector} and due to $\phi(\bm{\alpha}) = \delta_{\bm{\alpha},\bm{0}}$, $\forall \, \bm{\alpha} \in \ZZ^2$, we have
\begin{gather*}
\arraycolsep=3pt\def\arraystretch{1.4}
 v((1,0)) = \Bigl ( \phi \bigl ( (1,0) + \bm{\omega} \bigr ) \Bigr )_{\bm{\omega} \in \Omega} = 
\left (
\begin{array}{c}
\phi \bigl ( (1,0) + (-1,-1) \bigr ) \\
\phi \bigl ( (1,0) + (-1,0) \bigr ) \\
\phi \bigl ( (1,0) + (0,-1) \bigr ) \\
\phi \bigl ( (1,0) + (0,0) \bigr )
\end{array}
\right ) 
=
\begin{pmatrix}
0 \\ 1 \\ 0 \\ 0
\end{pmatrix} \in \RR^4, \\
 v((0,0)) = \Bigl ( \phi (\bm{\omega} ) \Bigr )_{\bm{\omega} \in \Omega} = 
\left (
\begin{array}{c}
\phi \bigl ( (-1,-1) \bigr ) \\
\phi \bigl ( (-1,0) \bigr ) \\
\phi \bigl ( (0,-1) \bigr ) \\
\phi \bigl ( (0,0) \bigr )
\end{array}
\right ) 
=
\begin{pmatrix}
0 \\ 0 \\ 0 \\ 1
\end{pmatrix} \in \RR^4.
\end{gather*}
Similarly to the computation of $U$ with $\mathbf{v}_0 = (0, 1/2, 0, -1/2)^T$, we obtain
\[
U_1 = \text{ span } \Set{T_{(0,1)} \mathbf{v}_0 -  \mathbf{v}_0, T_{(0,2)} \mathbf{v}_0 -  \mathbf{v}_0} = \text{ span } \Set{
\left (
\begin{array}{r}
1 \\ 0 \\ -1 \\ 0
\end{array}
\right),
\left (
\begin{array}{r}
0 \\ 1 \\ 0 \\ -1
\end{array}
\right)
}, \qquad \text{ dim } U_1 = 2.
\]

\noindent \emph{Space $U_2$:} In order to determine $U_2$, by \eqref{eq:spaces_U_U1_U2}, we compute $\mathbf{v}_0 = v((0,1)) - v((0,0))$, $\norm{\mathbf{v}_0}_1 = 1$. Following the procedure described for the construction of $U_1$, we have $\mathbf{v}_0 = (0, 0, 1/2, -1/2)^T$ and
\[
U_2 = \text{ span } \Set{T_{(0,1)} \mathbf{v}_0 -  \mathbf{v}_0, T_{(1,0)} \mathbf{v}_0 -  \mathbf{v}_0} = \text{ span } \Set{
\left (
\begin{array}{r}
1 \\ -1 \\ 0 \\ 0
\end{array}
\right),
\left (
\begin{array}{r}
0 \\ 0 \\ 1 \\ -1
\end{array}
\right)
}, \qquad \text{ dim } U_2 = 2.
\]
Note that $U = U_1 \bigcup U_2$.

Let us compute the set $\mathcal{V}$.

\noindent \emph{Step 1.} We extend $U \subset \RR^4$, $ \text{ dim } U = 3$, to a subspace $\tilde{U} \subset \RR^4$, $ \text{dim } \tilde{U} = 4$, adding 1 linearly independent vector to $U$, namely
\[
\tilde{U} = \text{ span } \Set{
\left (
\begin{array}{r}
1 \\ 0 \\ 0 \\ -1 
\end{array}
\right),
\left (
\begin{array}{r}
0 \\ 1 \\ 0 \\ -1 
\end{array}
\right),
\left (
\begin{array}{r}
0 \\ 0 \\ 1 \\ -1 
\end{array}
\right),
\left (
\begin{array}{r}
1 \\ 0 \\ 0 \\ 0 
\end{array}
\right)
}.
\]

\noindent \emph{Step 2.} We define the matrix $S \in \RR^{4 \times 4}$, whose columns are $4$ linearly independent elements of $\tilde{U}$, i.e.
\[
S = \left(
\begin{array}{rrrr}
 1 & 0 & 0 & 1 \\
 0 & 1 & 0 & 0 \\
 0 & 0 & 1 & 0 \\
 -1 & -1 & -1 & 0 \\
\end{array}
\right).
\]
 By construction, the matrix $S$ is invertible and, thus, we can compute the matrices 
\[
B_{\bm{\gamma}} = S^{-1} T_{\bm{\gamma}} S \in \RR^{N \times N}, \qquad \bm{\gamma} \in \Gamma.
\]
The matrices $B_{\bm{\gamma}}$ have a block structure. More precisely, the square upper-left block of $B_{\bm{\gamma}}$ of size $\text{dim } U \times \text{dim } U$ is the restriction of $T_{\bm{\gamma}}$ to $U$, namely
\[
\begin{array}{ccc}
\medskip
B_{(0,0)} = \frac16
\left(
\begin{array}{rrrr}
\tikzmark{left} 1 & 0 & 0 & -5 \\
 2 & 3 & 0 & 2 \\
 1 & 0 & 2 \tikzmark{right} & 1 \\
 0 & 0 & 0 & 6 \\
\end{array} \DrawBox[thick]
\right), &
B_{(0,1)} = \frac16
\left(
\begin{array}{rrrr}
\tikzmark{left} 2 & 1 & 0 & -4 \\
 1 & 2 & 0 & 1 \\
 0 & -1 & 2 \tikzmark{right} & 2 \\
 0 & 0 & 0 & 6 \\
\end{array}  \DrawBox[thick]
\right), &
B_{(0,2)} = \frac16
\left(
\begin{array}{rrrr}
\tikzmark{left} 3 & 2 & 0 & -3 \\
 0 & 1 & 0 & 0 \\
 -1 & -2 & 2 \tikzmark{right} & 3 \\
 0 & 0 & 0 & 6 \\
\end{array}  \DrawBox[thick]
\right), \\
B_{(1,0)} = \frac16
\left(
\begin{array}{rrrr}
\tikzmark{left} 2 & 0 & 1 & -4 \\
 1 & 3 & -1 & 4 \\
 0 & 0 & 1 \tikzmark{right} & 0 \\
 0 & 0 & 0 & 6 \\
\end{array}  \DrawBox[thick]
\right), &
B_{(1,1)} = \frac16
\left(
\begin{array}{rrrr}
\tikzmark{left} 3 & 1 & 1 & -2 \\
 0 & 2 & -1 & 2 \\
 -1 & -1 & 1 \tikzmark{right} & 0 \\
 0 & 0 & 0 & 6 \\
\end{array}  \DrawBox[thick]
\right), &
B_{(1,2)} = \frac16
\left(
\begin{array}{rrrr}
\tikzmark{left} 4 & 2 & 1 & 0 \\
 -1 & 1 & -1 & 0 \\
 -2 & -2 & 1 \tikzmark{right} & 0 \\
 0 & 0 & 0 & 6 \\
\end{array}  \DrawBox[thick]
\right).
\end{array}
\]
 Finally, $\mathcal{V}$ is the set of the restrictions of $T_{\bm{\gamma}}$, $\bm{\gamma} \in \Gamma$, to $U$. In our case,
\[
\small
\mathcal{V} = \Set{
\frac16 
\left(
\begin{array}{rrr}
 1 & 0 & 0 \\
 2 & 3 & 0 \\
 1 & 0 & 2 \\
\end{array}
\right),
\frac16 
\left(
\begin{array}{rrr}
 2 & 1 & 0 \\
 1 & 2 & 0 \\
 0 & -1 & 2 \\
\end{array}
\right),
\frac16
\left(
\begin{array}{rrr}
 3 & 2 & 0 \\
 0 & 1 & 0 \\
 -1 & -2 & 2 \\
\end{array}
\right),
\frac16
\left(
\begin{array}{rrr}
 2 & 0 & 1 \\
 1 & 3 & -1 \\
 0 & 0 & 1 \\
\end{array}
\right),
\frac16 
\left(
\begin{array}{rrr}
 3 & 1 & 1 \\
 0 & 2 & -1 \\
 -1 & -1 & 1 \\
\end{array}
\right),
\frac16
\left(
\begin{array}{rrr}
 4 & 2 & 1 \\
 -1 & 1 & -1 \\
 -2 & -2 & 1 \\
\end{array}
\right)
}.
\]

\noindent Now, let us focus on the construction of $\mathcal{V}_1, \, \mathcal{V}_2$. Since $U_1, U_2$ are invariant subspaces under $\mathcal{V}$, we can directly compute the restrictions $\mathcal{T} |_{U_1}$ and $\mathcal{T} |_{U_2}$, respectively.
Thus, we can apply the same algorithm used to determine $\mathcal{V}$ and we get
\[
\mathcal{V}_1 = \Set{
\frac16 \left(
\begin{array}{rr}
 1 & 0 \\
 2 & 3 \\
\end{array}
\right),
\frac16
\left(
\begin{array}{rr}
 2 & 1 \\
 1 & 2 \\
\end{array}
\right),
\frac16
\left(
\begin{array}{rr}
 3 & 2 \\
 0 & 1 \\
\end{array}
\right)
}
\quad \text{and} \quad
\mathcal{V}_2 = \Set{
\frac16
\left(
\begin{array}{rr}
 1 & 0 \\
 1 & 2 \\
\end{array}
\right),
\frac16
\left(
\begin{array}{rr}
 2 & 1 \\
 0 & 1 \\
\end{array}
\right)
}.
\]
Notice that $\# \mathcal{V}_1, \, \# \mathcal{V}_2 < 6$ since $U_1 \cap U_2 \neq \emptyset$.

\end{example}

In Table \ref{table:continuity_holder}, we check the continuity and compute the H{\"o}lder regularity of $\mathcal{S}_{\mathbf{a}_{M,n}}$ with $M \in \set{ \mathrm{ diag } (2,3), \, \mathrm{ diag } (2,5)}$ and $n \in \set{1,2}$ following the procedure presented in Example \ref{ex:holder}.

\begin{table} 
\centering
\small
\begin{tabular}{cccccc}
\toprule
{Dilation matrix} & {Subdivision scheme} & $\rho (\mathcal{V})$ & $\rho (\mathcal{V}_1)$ & $\rho (\mathcal{V}_2)$ & $\alpha_{\phi}$ \\
\toprule
\multirow{2}*{$M={\displaystyle \begin{pmatrix} 2 & 0 \\ 0 & 3 \end{pmatrix}}$} & $\mathcal{S}_{\mathbf{a}_{M,1}}$ & 0.500000 & 0.500000 & 0.333333 & 1 \\
 & $\mathcal{S}_{\mathbf{a}_{M,2}}$ & 0.500003 & 0.500002 & 0.333335 & 1 \\
\midrule
\multirow{2}*{$M={\displaystyle \begin{pmatrix} 2 & 0 \\ 0 & 5 \end{pmatrix}}$} & $\mathcal{S}_{\mathbf{a}_{M,1}}$ & 0.500000 & 0.500000 & 0.200000 & 1\\
& $\mathcal{S}_{\mathbf{a}_{M,1}}$ & 0.500004 & 0.500003 & 0.200002 & 1 \\
\bottomrule
\end{tabular}
\caption{Continuity and H{\"o}lder regularity of $\mathcal{S}_{\mathbf{a}_{M,n}}$ (Theorem \ref{t:holder_regularity}).}
\label{table:continuity_holder}
\end{table}

\section{Anisotropic approximating subdivision schemes}
\label{sec:anisotropic_pseudo}

In this section, we consider the dilation matrix $M = \mathrm{ diag (2,3)}$. 
We first introduce a family of symmetric four directional box-splines, see Definition \ref{def:box_anisotropic}, then we define a new family of symmetric four directional approximating subdivision schemes, see Definition \ref{d:pseudo_anisotropic}. The aim of this section is to provide a family of approximating subdivision schemes as reference schemes for our multigrid examples. Especially, in Section \ref{sec:examples}, we show that for $M =  \mathrm{ diag (2,3)}$ the interpolatory subdivision schemes in Definition \ref{d:interp_anisotropic} are computationally superior to the approximating subdivision schemes that we define in this section.

\begin{definition}[Anisotropic Symmetric Four Directional Box-Splines] \label{def:box_anisotropic}
Let $n \in \NN$. The \emph{anisotropic symmetric four directional box-spline} $\mathcal{S}_{B_n}$ of order $n$ and dilation matrix $M= \mathrm{ diag (2,3)}$ is defined by its symbol
\[
B_n (z_1,z_2) = 6 \, \left ( \frac{(1+z_1)^2}{4 z_1} \frac{(1+z_2+z_2^2)^2}{9 z_2^2} \right )^{\lceil \frac{n}{2} \rceil} \left (  \frac{(2+z_2+z_1z_2+2z_1z_2^2)(2z_1+z_2+z_1z_2+2z_2^2)}{36z_1z_2^2} \right )^{\lfloor \frac{n}{2} \rfloor}, 
\quad (z_1,z_2) \in ( \CC \setminus \set{0} )^2.
\]
\end{definition}

\noindent In order to understand the definition of the symbols $B_n$, $n \in \NN$, in Definition \ref{def:box_anisotropic}, we need to look closely at the "basic" Laurent polynomials $B_1$ and $B_2$.
For $n=1$, the symbol $B_1$  in Definition \ref{def:box_anisotropic} becomes
\[
B_1 (z_1,z_2) = \frac{(1+z_1)^2}{4 z_1} \frac{(1+z_2+z_2^2)^2}{9 z_2^2}, \qquad (z_1,z_2) \in ( \CC \setminus \set{0} )^2.
\]
The factors $(1+z_1)^2 / (4 z_1)$ and $(1+z_2+z_2^2)^2/(9 z_2^2)$ are called first and second direction and they are the symbols of the univariate binary and ternary linear B-splines, respectively. Thus, the subdivision scheme $\mathcal{S}_{B_1}$ generates polynomials up to degree 1 (by a tensor product argument) and its mask $\mathbf{B}_1$ is symmetric and minimally supported,
\[
\arraycolsep=3pt\def\arraystretch{1.4}
\mathbf{B}_{1} = \left(
\begin{array}{ccccc}
 \frac{1}{6} & \frac{1}{3} & \frac{1}{2} & \frac{1}{3} & \frac{1}{6} \\
 \frac{1}{3} & \frac{2}{3} & 1 & \frac{2}{3} & \frac{1}{3} \\
 \frac{1}{6} & \frac{1}{3} & \frac{1}{2} & \frac{1}{3} & \frac{1}{6} \\
\end{array}
\right).
\]
For $n=2$, the symbol $B_2$ in Definition \ref{def:box_anisotropic} becomes
\[
B_2 (z_1,z_2) = \frac{(1+z_1)^2}{4 z_1} \frac{(1+z_2+z_2^2)^2}{9 z_2^2} \frac{(2+z_2+z_1z_2+2z_1z_2^2)(2z_1+z_2+z_1z_2+2z_2^2)}{36z_1z_2^2} , \qquad (z_1,z_2) \in ( \CC \setminus \set{0} )^2.
\]
The factor $(2+z_2+z_1z_2+2z_1z_2^2)(2z_1+z_2+z_1z_2+2z_2^2)/(36 z_1 z_2^2)$ represents the product of the so-called third and fourth directions. We computed such a symbol $B_2$ in order to guarantee that the subdivision scheme $\mathcal{S}_{B_2}$ generates polynomials up to degree 3 (Proposition \ref{p:box_generation}) and its mask $\mathbf{B}_2$ is symmetric and minimally supported,
\[
\arraycolsep=3pt\def\arraystretch{1.4}
\mathbf{B}_{2} = \left(
\begin{array}{ccccccccc}
 0 & \frac{1}{108} & \frac{1}{24} & \frac{1}{12} & \frac{23}{216} & \frac{1}{12} & \frac{1}{24} & \frac{1}{108} & 0 \\
 \frac{1}{54} & \frac{2}{27} & \frac{5}{27} & \frac{8}{27} & \frac{19}{54} & \frac{8}{27} & \frac{5}{27} & \frac{2}{27} & \frac{1}{54} \\
 \frac{1}{27} & \frac{7}{54} & \frac{31}{108} & \frac{23}{54} & \frac{53}{108} & \frac{23}{54} & \frac{31}{108} & \frac{7}{54} & \frac{1}{27} \\
 \frac{1}{54} & \frac{2}{27} & \frac{5}{27} & \frac{8}{27} & \frac{19}{54} & \frac{8}{27} & \frac{5}{27} & \frac{2}{27} & \frac{1}{54} \\
 0 & \frac{1}{108} & \frac{1}{24} & \frac{1}{12} & \frac{23}{216} & \frac{1}{12} & \frac{1}{24} & \frac{1}{108} & 0 \\
\end{array}
\right).
\]
Finally, the definition of such a symbol $B_n$, $n \in \NN$, in Definition \ref{def:box_anisotropic}, guarantees that the subdivision scheme $\mathcal{S}_{B_n}$ generates polynomials up to degree $2n-1$ (Proposition \ref{p:box_generation}) and its mask $\mathbf{B}_n$ is symmetric.

In order to prove Propositions \ref{p:box_generation} and \ref{p:box_reproduction}, we need an auxiliary Lemma.

\begin{lemma}
Let $n \in \NN$. The Laurent polynomial $B_n$ in Definition \ref{def:box_anisotropic} satisfies 
\begin{equation} \label{eq:box_anisotropi_new}
\begin{split}
&B_n (z_1,z_2) = 6 \, \sum_{j=0}^{ \lfloor \frac{n}{2} \rfloor } \binom{ \lfloor \frac{n}{2} \rfloor }{j} (-1)^j \, b^{(1)}_{n,j} (z_1) \,  b^{(2)}_{n,j} (z_2), \\
&b^{(1)}_{n,j} (z_1) = \left ( \frac{(1+z_1)^2}{4 z_1} \right )^{n-j} \left ( \frac{(1-z_1)^2}{4 z_1} \right )^j , \\
&b^{(2)}_{n,j} (z_2) = \left ( \frac{(1+z_2+z_2^2)^2}{9 z_2^2} \right )^{n-j}  \left ( \frac{ (1-z_2^2)^2}{9 z_2^2} \right )^{j}, \qquad (z_1,z_2) \in (\CC \setminus \set{0})^2.
\end{split}
\end{equation}
\end{lemma}

\begin{proof}
We observe that the product of the third and the fourth directions can be written as
\[
 \frac{(2+z_2+z_1z_2+2z_1z_2^2)(2z_1+z_2+z_1z_2+2z_2^2)}{36z_1z_2^2} = \frac{(1+z_1)^2}{4 z_1} \frac{(1+z_2+z_2^2)^2}{9 z_2^2} - \frac{(1-z_1)^2 (1-z_2^2)^2}{36 z_1 z_2^2}.
\]
Using this identity, for $(z_1,z_2) \in (\CC \setminus \set{0})^2$, we can rewrite the Laurent polynomial $B_n$ in Definition \ref{def:box_anisotropic} as
\[
\begin{split}
B_n (z_1,z_2) &= 6 \, \left ( \frac{(1+z_1)^2}{4 z_1} \frac{(1+z_2+z_2^2)^2}{9 z_2^2} \right )^{\lceil \frac{n}{2} \rceil} 
\left ( \frac{(1+z_1)^2}{4 z_1} \frac{(1+z_2+z_2^2)^2}{9 z_2^2} - \frac{(1-z_1)^2 (1-z_2^2)^2}{36 z_1 z_2^2}  \right )^{\lfloor \frac{n}{2} \rfloor}, \\
&= 6 \, \left ( \frac{(1+z_1)^2}{4 z_1} \frac{(1+z_2+z_2^2)^2}{9 z_2^2} \right )^{\lceil \frac{n}{2} \rceil} 
\sum_{j=0}^{\lfloor \frac{n}{2} \rfloor} \binom{\lfloor \frac{n}{2} \rfloor}{j} (-1)^j 
\left ( \frac{(1+z_1)^2}{4 z_1} \frac{(1+z_2+z_2^2)^2}{9 z_2^2} \right )^{\lfloor \frac{n}{2} \rfloor-j}
\left ( \frac{(1-z_1)^2 (1-z_2^2)^2}{36 z_1 z_2^2} \right )^j, \\
&=6 \,  \sum_{j=0}^{\lfloor \frac{n}{2} \rfloor} \binom{\lfloor \frac{n}{2} \rfloor}{j} (-1)^j 
\left ( \frac{(1+z_1)^2}{4 z_1} \frac{(1+z_2+z_2^2)^2}{9 z_2^2} \right )^{n-j}
\left ( \frac{(1-z_1)^2 (1-z_2^2)^2}{36 z_1 z_2^2} \right )^j, \\
&= 6 \,  \sum_{j=0}^{\lfloor \frac{n}{2} \rfloor} \binom{\lfloor \frac{n}{2} \rfloor}{j} (-1)^j 
\left ( \frac{(1+z_1)^2}{4 z_1} \right )^{n-j} \left ( \frac{(1-z_1)^2}{4 z_1} \right )^j
\left ( \frac{(1+z_2+z_2^2)^2}{9 z_2^2} \right )^{n-j}  \left ( \frac{ (1-z_2^2)^2}{9 z_2^2} \right )^{j}, \\
&= 6 \, \sum_{j=0}^{ \lfloor \frac{n}{2} \rfloor } \binom{ \lfloor \frac{n}{2} \rfloor }{j} (-1)^j \, b^{(1)}_{n,j} (z_1) \,  b^{(2)}_{n,j} (z_2).
\end{split}
\]
\end{proof}

\begin{proposition} \label{p:box_generation}
Let $n \in \NN$. The anisotropic  symmetric four directional box-spline $\mathcal{S}_{B_n}$ of order $n$ associated with the symbol $B_n$ in Definition \ref{def:box_anisotropic} generates polynomials up to degree $2n-1$.
\end{proposition}

\begin{proof}
We proceed by induction. 

\noindent \emph{Step 1.} For $n=1$, the base case is trivial due to a tensor product argument.

\noindent \emph{Step 2.} Let us suppose that for any $n \geq 1$, $\mathcal{S}_{B_n}$ generates polynomials up to degree $2n-1$. We want to show that $\mathcal{S}_{B_{n+1}}$ generates polynomials up to degree $2(n+1)-1=2n+1$.
The symbol of the anisotropic symmetric four directional box-spline $\mathcal{S}_{B_{n+1}}$ satisfies the recursive formula
\[
B_{n+1}(z_1,z_2) = 
\begin{cases}
\bigskip
 B_n (z_1,z_2) \cdot \frac16 B_1 (z_1,z_2), & n \text{ even}, \\
 B_n (z_1,z_2) \cdot Q(z_1,z_2), & n \text{ odd},
\end{cases}
\]
where
\[
Q(z_1,z_2) = \frac{ \bigl ( 2+z_2+z_1z_2+2z_1z_2^2 \bigr ) \bigl ( 2z_1+z_2+z_1z_2+2z_2^2 \bigr ) }{36z_1z_2^2}.
\]

\noindent For $n$ even, $\mathcal{S}_{B_n} \in \mathcal{I}_{2n-1}$ by induction and $\mathcal{S}_{B_1} \in \mathcal{I}_{1}$ by \emph{Step 1}, thus, $\mathcal{S}_{B_{n+1}} \in \mathcal{I}_{2n+1}$.

\noindent For $n$ odd, we cannot apply the same argument as before since $Q(z_1,z_2)$ does not vanish on $E_M \setminus \set{(1,1)}$. 
Let $\bm{\alpha} \in \NN_0^2$, $\abs{\bm{\alpha}} \leq 2n+1$. Applying the Leibniz formula to $B_{n+1}$ we get
\begin{equation} \label{eq:aux}
D^{(\alpha_1,\alpha_2)} B_{n+1}(z_1,z_2) = \sum_{\beta_1=0}^{\alpha_1} \, \sum_{\beta_2=0}^{\alpha_2} \binom{\alpha_1}{\beta_1} \, \binom{\alpha_2}{\beta_2} \,
D^{(\beta_1,\beta_2)} B_{n}(z_1,z_2) \cdot D^{(\alpha_1-\beta_1, \alpha_2-\beta_2)} Q(z_1,z_2).
\end{equation}
The following analysis is split in 2 steps: $\abs{\bm{\alpha}} = 2n$ and $\abs{\bm{\alpha}} = 2n+1$.

\emph{(i)} Let $\abs{\bm{\alpha}} \in \NN_0^2$, $\abs{\bm{\alpha}} = 2n$. From \eqref{eq:aux} and \emph{(i)}, we get
\[
D^{(\alpha_1,\alpha_2)} B_{n+1}(z_1,z_2) = D^{(\alpha_1,\alpha_2)} B_{n}(z_1,z_2) \cdot Q(z_1,z_2).
\]
By straightforward computation, we have
\[
Q(\bm{\varepsilon}) =0, \qquad \bm{\varepsilon} \in \Set{ \bigl(1, e^{2/3 \pi \mathrm{i}} \bigr ), \,  \bigl(1, e^{4/3 \pi \mathrm{i}} \bigr ), \, (-1,1)},
\]
thus 
\[
D^{(\alpha_1,\alpha_2)} B_{n+1}(\bm{\varepsilon}) =0, \qquad \bm{\varepsilon} \in \Set{ \bigl(1, e^{2/3 \pi \mathrm{i}} \bigr ), \,  \bigl(1, e^{4/3 \pi \mathrm{i}} \bigr ), \, (-1,1)}.
\]
Now we need to study the behavior of $D^{(\alpha_1,\alpha_2)} B_{n+1} (\bm{\varepsilon})$, that is the behavior of $D^{(\alpha_1,\alpha_2)} B_{n} (\bm{\varepsilon})$, for
\[
 \bm{\varepsilon} \in E_M \setminus \Set{ (1,1), \, \bigl(1, e^{2/3 \pi \mathrm{i}} \bigr ), \,  \bigl(1, e^{4/3 \pi \mathrm{i}} \bigr ), \, (-1,1)} 
= \Set{\bigl(-1, e^{2/3 \pi \mathrm{i}} \bigr ), \,  \bigl(-1, e^{4/3 \pi \mathrm{i}} \bigr )}.
\]
W.l.o.g., we focus our attention on $\bm{\varepsilon} = \bigl(-1, e^{2/3 \pi \mathrm{i}} \bigr )$.
From \eqref{eq:box_anisotropi_new}, we get
\begin{equation} \label{eq:Bn_derivative_z1z2separate}
D^{(\alpha_1,\alpha_2)} B_n (z_1,z_2) = 6 \, \sum_{j=0}^{ \frac{n-1}{2} } \binom{\frac{n-1}{2}}{j} (-1)^j 
\frac{d^{\alpha_1}}{d z_1^{\alpha_1}} b^{(1)}_{n,j} (z_1) \, \frac{d^{\alpha_2}}{d z_2^{\alpha_2}} b^{(2)}_{n,j} (z_2),
\end{equation}
thus, in order to compute $D^{(\alpha_1,\alpha_2)} B_n (\bm{\varepsilon})$, we need to study separately the behavior of
\[
\begin{split}
&\frac{d^{\alpha_1}}{d z_1^{\alpha_1}} b^{(1)}_{n,j} (z_1) \Bigg |_{z_1=-1} = \sum_{\beta_1=0}^{\alpha_1} \binom{\alpha_1}{\beta_1} 
\frac{d^{\beta_1}}{d z_1^{\beta_1}} (1+z_1)^{2(n-j)} \Bigg |_{z_1=-1} \cdot \frac{d^{\alpha_1-\beta_1}}{d z_1^{\alpha_1-\beta_1}} (4^{-n} z_1^{-n} (1-z_1)^{2j} )  \Bigg |_{z_1=-1}, \\
&\frac{d^{\alpha_2}}{d z_2^{\alpha_2}} b^{(2)}_{n,j} (z_2) \Bigg |_{z_2=e^{2/3 \pi \mathrm{i}}} = \sum_{\beta_2=0}^{\alpha_2} \binom{\alpha_2}{\beta_2} 
\frac{d^{\beta_2}}{d z_2^{\beta_2}} (1+z_2+z_2^2)^{2(n-j)} \Bigg |_{z_2=e^{2/3 \pi \mathrm{i}}} \cdot \frac{d^{\alpha_2-\beta_2}}{d z_2^{\alpha_2-\beta_2}} (9^{-n} z_2^{-2n} (1-z_2^2)^{2j}) \Bigg |_{z_2=e^{2/3 \pi \mathrm{i}}},
\end{split}
\] 
for $j = 0, \ldots, \frac{n-1}{2}$.
Notice that for any $j \in \set{0, \ldots, \frac{n-1}{2}}$, we have $2(n-j) \in \set{n+1, \ldots ,2n}$. 

\noindent Let $\alpha_1 \in \set{0, \ldots, n}$. For any $\beta_1 \in \set{0, \ldots, \alpha_1}$ and for any $j \in \set{0, \ldots, \frac{n-1}{2}}$, we have $\beta_1 \leq \alpha_1 \leq n < 2(n-j)$. Thus,
\[
\frac{d^{\beta_1}}{d z_1^{\beta_1}} (1+z_1)^{2(n-j)} \Bigg |_{z_1=-1} = 0,
\]
and we get $D^{(\alpha_1,\alpha_2)} B_n (\bm{\varepsilon})=0$.

\noindent Let $\alpha_1 \in \set{n+1, \ldots, 2n}$. Then $\alpha_2 = 2n-\alpha_1 \in \set{0, \ldots,n-1}$. Using the same argument as before, for any $\beta_2 \in \set{0, \ldots, \alpha_2}$ and for any $j \in \set{0, \ldots, \frac{n-1}{2}}$, we get
\[
\frac{d^{\beta_2}}{d z_2^{\beta_2}} (1+z_2+z_2^2)^{2(n-j)} \Bigg |_{z_2=e^{2/3 \pi \mathrm{i}}} = 0.
\]
Thus, from \eqref{eq:Bn_derivative_z1z2separate}, we get $D^{(\alpha_1,\alpha_2)} B_n (\bm{\varepsilon})=0$.

\emph{(ii)} Let $\abs{\bm{\alpha}} \in \NN_0^2$, $\abs{\bm{\alpha}} = 2n+1$. From \eqref{eq:aux}, \emph{(i)} and \emph{(ii)}, we get
\[
D^{(\alpha_1,\alpha_2)} B_{n+1}(z_1,z_2) =  D^{(\alpha_1,\alpha_2)} B_{n}(z_1,z_2) \cdot Q(z_1,z_2).
\]
Thus, we need to study the behavior of $D^{(\alpha_1,\alpha_2)} B_{n+1} (\bm{\varepsilon})$, that is the behavior of $D^{(\alpha_1,\alpha_2)} B_{n} (\bm{\varepsilon})$, for
\[
 \bm{\varepsilon} \in E_M \setminus \Set{ (1,1), \, \bigl(1, e^{2/3 \pi \mathrm{i}} \bigr ), \,  \bigl(1, e^{4/3 \pi \mathrm{i}} \bigr ), \, (-1,1)} 
= \Set{\bigl(-1, e^{2/3 \pi \mathrm{i}} \bigr ), \,  \bigl(-1, e^{4/3 \pi \mathrm{i}} \bigr )}.
\]
We notice that 
\begin{itemize}
\item $\alpha_1 \in \set{0, \ldots, n}$: for any $\beta_1 \in \set{0, \ldots, \alpha_1}$ and for any $j \in \set{0, \ldots, \frac{n-1}{2}}$, we have $\beta_1 \leq \alpha_1 \leq n < 2(n-j)$,
\item $\alpha_1 \in \set{n+1, \ldots, 2n+1}$: $\alpha_2 = 2n+1-\alpha_1 \in \set{0, \ldots,n}$, thus for any $\beta_2 \in \set{0, \ldots, \alpha_2}$ and for any $j \in \set{0, \ldots, \frac{n-1}{2}}$, we have $\beta_2 \leq \alpha_2 \leq n < 2(n-j)$.
\end{itemize}
Thesis follows from the same argument of \emph{(ii)}.
\end{proof}

\begin{proposition} \label{p:box_reproduction}
Let $n \in \NN$. The anisotropic symmetric four directional box-spline $\mathcal{S}_{B_n}$ of order $n$ associated with the symbol $B_n$ in Definition \ref{def:box_anisotropic} reproduces polynomials up to degree 1.
\end{proposition}

\begin{proof}
In order to prove Proposition \ref{p:box_reproduction}, by Theorem \ref{t:poly_reproduction} and Proposition \ref{p:box_generation}, we need to show that
\[
\begin{array}{cll}
(i)   & B_n(1,1) = \abs{\det{M}} = 6, & \\
(ii)  & D^{\bm{\alpha}} B_n (1,1) = 0, & \forall \bm{\alpha} \in \NN_0^2 \, : \abs{\bm{\alpha}} = 1, \\
(iii) & D^{\bm{\alpha}} B_n (1,1) \neq 0 & \text{for some } \bm{\alpha} \in \NN_0^2 \, : \abs{\bm{\alpha}} = 2.
\end{array}
\]

\noindent \emph{(i)} By Definition \ref{def:box_anisotropic}, $B_n(1,1) = 6$.

\noindent \emph{(ii)} Let $\bm{\alpha} = (1,0)$, $\abs{\bm{\alpha}}=1$. Using \eqref{eq:box_anisotropi_new} and noticing that $b^{(2)}_{n,j} (1) = 0$ for $j = 1 , \dots,  \lfloor n/2 \rfloor$, the $(1,0)$-th directional derivative of $B_n$ evaluated at $(1,1)$ becomes
\[
\begin{split}
D^{(1,0)} B_n (1,1) &=  6 \, \sum_{j=0}^{ \lfloor \frac{n}{2} \rfloor } \binom{ \lfloor \frac{n}{2} \rfloor }{j} (-1)^j \, \frac{d}{d z_1} b^{(1)}_{n,j} (z_1) \big |_{z_1 = 1} \,  b^{(2)}_{n,j} (1) \\
&=  6 \, \frac{d}{d z_1} \left ( \frac{(1+z_1)^2}{4 z_1} \right )^n \Bigg |_{z_1=1} \\
&= 6 n \, \left ( \frac{(1+z_1)^2}{4 z_1} \right )^{n-1} \Bigg |_{z_1=1} \underbrace{ \left ( \frac{1+z_1}{2 z_1} - \frac{(1+z_1)^2}{4 z_1^2} \right ) \Bigg |_{z_1=1}}_{{} = 0} = 0.
\end{split}
\]
Analogously for $\bm{\alpha} = (0,1)$, $\abs{\bm{\alpha}}=1$.

\noindent \emph{(iii)} Let $\bm{\alpha} = (2,0)$, $\abs{\bm{\alpha}}=2$. We show that $D^{(2,0)} B_n (1,1) \neq 0$. Using the previous argument, we get
\[
D^{(2,0)} B_n (1,1) =  6 \, \frac{d^2}{d z_1^2} \left ( \frac{(1+z_1)^2}{4 z_1} \right )^n \Bigg |_{z_1=1} = 3n.
\]
\end{proof}

\begin{remark}
Let $n \in \NN$. The convergence of the anisotropic symmetric four directional box-spline $\mathcal{S}_{B_n}$ in Definition \ref{def:box_anisotropic} follows by standard argument involving the smoothing factors.
\end{remark}

We are now ready to define the family of anisotropic symmetric four directional approximating schemes.

\begin{definition} \label{d:pseudo_anisotropic}
Let $n \in \NN$, $\ell \in \set{0, \ldots, n-1}$. The \emph{anisotropic symmetric four directional approximating scheme} of order $(n,\ell)$ and dilation matrix $M = \mathrm{diag } (2,3)$ is defined by its symbol
\[
B_{n,\ell} (z_1,z_2) = \sum_{i=0}^{\ell} B_{n-i} (z_1,z_2) \cdot \sum_{j=0}^{i} \mathrm{c}_n^{(i,j)} \delta_1 (z_1)^{i-j} \delta_2 (z_2)^{j}, \qquad (z_1,z_2) \in ( \CC \setminus \set{0} )^2,
\]
where
\[
\delta_1 (z_1) = - \frac{ \bigl ( 1-z_1^2 \bigr )^2}{16 z_1^2}, \qquad \delta_2 (z_2) = - \frac{ \bigl ( 1-z_2^3 \bigr )^2}{27 z_2^3},
\]
the coefficients $\mathrm{c}_n^{(i,j)}$ are computed recursively as the solution of the system of equations
\[
D^{(2(i-j),2j)} B_{n,\ell} (1,1) = 0, \qquad j=0, \ldots, i, \qquad i=0, \ldots, \ell-1.
\]
\end{definition}

\begin{proposition} \label{p:pseudo_generation}
Let $n \in \NN$, $\ell \in \set{0, \ldots, n-1}$. The anisotropic symmetric four directional approximating scheme $\mathcal{S}_{B_{n,\ell}}$ of order $(n,\ell)$ associated with the symbol $B_{n,\ell}$ in Definition \ref{d:pseudo_anisotropic} generates polynomials up to degree $2n-1$.
\end{proposition}

\begin{proof}
By Proposition \ref{p:box_generation}, for $i=0, \ldots, \ell$, the symmetric four directional box-spline $\mathcal{S}_{B_{n-i}} \in \mathcal{I}_{2n-2i-1}$.
By definition, 
\[
\delta_1 (z_1)^{i-j} \delta_2 (z_2)^{j} = \Biggl ( - \frac{ \bigl ( 1-z_1^2 \bigr )^2}{16 z_1^2} \Biggr )^{i-j} \Biggl ( - \frac{ \bigl ( 1-z_2^3 \bigr )^2}{27 z_2^3} \Biggr )^j \in \mathcal{J}_{2i-1} \subset \mathcal{I}_{2i-1}.
\]
Since the ideal $\mathcal{I}_k$, $k \in \NN_0$, is closed under addition, we have
\[
B_{n-i} \cdot \underbrace{\sum_{j=0}^{i} \mathrm{c}_n^{(i,j)} \delta_1 (z_1)^{i-j} \delta_2 (z_2)^{j}}_{{} \in \mathcal{I}_{2i-1}} \in \mathcal{I}_{2n-1} \qquad \implies \qquad
B_{n,\ell} (z_1,z_2) = \sum_{i=0}^{\ell} B_{n-i} \cdot \sum_{j=0}^{i} \mathrm{c}_n^{(i,j)} \delta_1 (z_1)^{i-j} \delta_2 (z_2)^{j} \in \mathcal{I}_{2n-1},
\]
that is $\mathcal{S}_{B_{n,\ell}}$ generates polynomials up to degree $2n-1$.
\end{proof}

\begin{remark}
Let $n \in \NN$, $\ell \in \set{0, \ldots, n-1}$. The convergence analysis of the anisotropic symmetric four directional approximating scheme $\mathcal{S}_{B_{n,\ell}}$ in Definition \ref{d:pseudo_anisotropic} can be done similarly to subsection \ref{ssec:convergence}.
\end{remark}
Finally, we believe that for any $n \in \NN$, $ \ell \in \set{0, \ldots, n-1}$, the anisotropic symmetric four directional approximating scheme $\mathcal{S}_{B_{n,\ell}}$ in Definition \ref{d:pseudo_anisotropic} reproduces polynomials up to degree $2 \ell + 1$, and we actually verified it for $n \leq 10$. Notice that, if $\ell = n-1$, then $\mathcal{S}_{B_{n,n-1}}$ generates and reproduces polynomials up to the same degree $2n-1$. Contrary to the univariate case, in the bivariate case this property does not imply that the subdivision scheme $\mathcal{S}_{B_{n,n-1}}$ is interpolatory. Indeed, its mask $B_{n,n-1}$ does not satisfy the interpolatory condition \eqref{eq:mask_interp}. See Example \ref{ex:approx23}, masks  $B_{2,1}$ and  $B_{3,2}$.

\section{Subdivision, multigrid and examples}
\label{sec:examples}

In the following, we assume that the system matrix $A_n$ in \eqref{eq:sys} is derived via finite difference discretization of the $d$-dimensional elliptic PDE problem \eqref{eq:PDEs} of order $2q$, $q \in \NN$, with Dirichlet boundary conditions.

In this section, we exhibit a new class of V-cycle grid transfer operators $P_{\mathbf{n}_j}$, $j=0, \ldots, \ell-1$, defined from symbols of subdivision schemes with certain polynomial generation properties, see Theorem \ref{t:subd_vcycle}. Then, in subsections \ref{subsec:interpolatory_gto} and \ref{subsec:approximating_gto}, we provide explicit examples of bivariate grid transfer operators defined from the symbols of anisotropic interpolatory and approximating subdivision schemes from Section \ref{sec:anisotropic_interp} and \ref{sec:anisotropic_pseudo}, respectively.

Let $p$ be the symbol in \eqref{eq:symbol_trigo_poly} associated to the subdivision scheme $S_{\mathbf{p}}$. For
\[
\mathbf{z} = e^{-\mathrm{i} \mathbf{x}} = (e^{- \mathrm{i} x_1}, \ldots, e^{- \mathrm{i} x_d}) \in (\CC \setminus \set{0})^d, \qquad \mathbf{x} \in \RR^d,
\]
the symbol $p$ is a $2 \pi$-periodic trigonometric polynomial. Thus, we write 
\begin{equation} \label{eq:trig_from_symb}
p(\mathbf{y}) := p(e^{- \mathrm{i} \mathbf{y}}), \qquad \mathbf{y} \in [0,2\pi)^d.
\end{equation}
By slight abuse of notation, we call $p$ both the symbol of the subdivision scheme and the associated trigonometric polynomial in \eqref{eq:trig_from_symb}, since it is clear from the context to which "class" of polynomials (Laurent or trigonometric) $p$ belongs.

\begin{theorem} \label{t:subd_vcycle}
Let $S_{\mathbf{p}}$ be a subdivision scheme with dilation $M=\underset{i=1, \ldots, d}{\mathrm{diag}} m_i \in \NN^{d \times d}$, $m_i \geq 2$,  and mask $\mathbf{p}$. If the symbol $p$ in \eqref{eq:symbol_trigo_poly} associated to the subdivision scheme $S_{\mathbf{p}}$ satisfies
\begin{equation} \label{aux}
\begin{array}{ll}
(i)  & D^{\bm{\mu}} p (\bm{\varepsilon}) = 0, \qquad \forall \, \bm{\varepsilon} \in E_M \setminus \set{\mathbf{1}}, \qquad \bm{\mu} \in \NN_0^d, \qquad \abs{\bm{\mu}} \leq 2q-1,  \\
(ii) & p (\mathbf{1}) = \abs{ \text{det } M} = { \displaystyle \prod_{i=1}^d m_i},
\end{array}
\end{equation}
then the associated trigonometric polynomial in \eqref{eq:trig_from_symb} satisfies conditions (i) and (ii) in \eqref{eq:opt_Vcycle_conditions}.
\end{theorem}

\begin{proof}
To prove the claim, we show that $(i)$ and $(ii)$ in \eqref{aux} imply conditions $(i)$ and $(ii)$ in \eqref{eq:opt_Vcycle_conditions}. Let $p$ be the trigonometric polynomial in \eqref{eq:trig_from_symb} associated to the symbol of the subdivision scheme $S_{\mathbf{p}}$.
 From
\[
\begin{split}
E_M \setminus \set{\mathbf{1}} &= \Set{ e^{- 2 \pi \mathrm{i} M^{-1} \bm{\gamma}} \, : \, \bm{\gamma} \in \Gamma \setminus \set{\bm{0}} }, \\
&= \Set{ e^{- 2 \pi \mathrm{i} \mathbf{y}} \, : \, \mathbf{y} = (y_1, \ldots, y_d) \in \QQ^d, \, y_i \in \Set{0, \frac{1}{m_i}, \ldots, \frac{m_i-1}{m_i}}, \, i = 1, \ldots, d, \, \mathbf{y} \neq \bm{0} },
\end{split}
\]
conditions $(i)$ and $(ii)$ in \eqref{aux} become
\[
\begin{array}{ll}
 (i) & D^{\bm{\mu}} p ( 2 \pi \mathbf{y}) = 0, \quad \forall \, \mathbf{y}= (y_1, \ldots, y_d) \in \QQ^d, \, y_i \in \Set{0, \frac{1}{m_i}, \ldots, \frac{m_i-1}{m_i}}, \, i = 1, \ldots, d, \, \mathbf{y} \neq \bm{0}, \quad  \abs{\bm{\mu}} \leq 2q-1,  \\
(ii) & p (\mathbf{0}) = {\displaystyle \prod_{i=1}^d m_i},
\end{array}
\]
which imply conditions $(i)$ and $(ii)$ in \eqref{eq:opt_Vcycle_conditions}.
\end{proof}

In subsections \ref{subsec:interpolatory_gto} and \ref{subsec:approximating_gto}, we restrict our attention to the bivariate case and, on the strength of Theorem \ref{t:subd_vcycle}, we propose appropriate grid transfer operators for the bivariate elliptic  PDE problem \eqref{eq:PDEs} of order $2q$, $q \in \NN$, with Dirichlet boundary conditions. 
We recall that, from the mask $\mathbf{p}$ of the bivariate subdivision scheme $\mathcal{S}_{\mathbf{p}}$ of dilation $M$, one can build the trigonometric polynomial $p (\mathbf{x})$, $\mathbf{x} \in [0, 2 \pi )^2$, as in \eqref{eq:trig_from_symb}.
Then, for $j=0, \ldots, \ell-1$, the grid transfer operator $P_{\mathbf{n}_j}$ at level $j$ associated to the subdivision scheme $\mathcal{S}_{\mathbf{p}}$ is defined by \eqref{eq:vcycle_notation}.

\subsection{Interpolatory grid transfer operators}
\label{subsec:interpolatory_gto}

The following result is a direct consequence of Theorem \ref{t:subd_vcycle}.

\begin{proposition} \label{p:subd_anisotropic_vcycle}
Let $n \in \NN$, $n \geq q$. Then, the trigonometric polynomial $a_{M,n}$ associated to the anisotropic interpolatory subdvision scheme $\mathcal{S}_{\mathbf{a}_{M,n}}$ in Definition \ref{d:interp_anisotropic} satisfies conditions (i) and (ii) in \eqref{eq:opt_Vcycle_conditions}.
\end{proposition}

\begin{proof}
By Proposition \ref{p:anisotropic_repr}, the anisotropic interpolatory subdvision scheme $\mathcal{S}_{\mathbf{a}_{M,n}}$ in Definition \ref{d:interp_anisotropic} generates polynomials up to degree $2n-1 \geq 2q-1$. Thus, by Theorem \ref{t:poly_generation}, $(i)$ of Theorem \ref{t:subd_vcycle} is satisfied.

\noindent Moreover, the symbol $a_{M,n}$ in \eqref{eq:interp_anisotropic} satisfies the necessary condition for convergence, namely $a_{M,n} (\mathbf{1}) = \abs{ \text{det } M} = 2m$.
Thus, $(ii)$ of Theorem \ref{t:subd_vcycle} is also satisfied.
\end{proof}

In Examples \ref{ex:interp23} and \ref{ex:interp25}, we give several examples of masks of the anisotropic interpolatory subdvision schemes $\mathcal{S}_{\mathbf{a}_{M,n}}$ with $M = { \displaystyle \begin{pmatrix} 2 & 0 \\ 0 & 3 \end{pmatrix}}$ and $M = { \displaystyle \begin{pmatrix} 2 & 0 \\ 0 & 5 \end{pmatrix}}$, respectively. The corresponding grid transfer operators are used in our numerical experiments.

\begin{example} \label{ex:interp23}
We focus our attention on the case $M = { \displaystyle \begin{pmatrix} 2 & 0 \\ 0 & 3 \end{pmatrix}}$. For $n = 1,2,3$, the masks of the anisotropic interpolatory subdvision scheme $\mathcal{S}_{\mathbf{a}_{M,n}}$ in Definition \ref{d:interp_anisotropic} are
\[
\arraycolsep=3pt\def\arraystretch{1.4}
\mathbf{a}_{M,1} = \left(
\begin{array}{ccccc}
 \frac{1}{6} & \frac{1}{3} & \frac{1}{2} & \frac{1}{3} & \frac{1}{6} \\
 \frac{1}{3} & \frac{2}{3} & 1 & \frac{2}{3} & \frac{1}{3} \\
 \frac{1}{6} & \frac{1}{3} & \frac{1}{2} & \frac{1}{3} & \frac{1}{6} \\
\end{array}
\right), \]
\[
\arraycolsep=3pt\def\arraystretch{1.4}
\mathbf{a}_{M,2} = {\small
\left(
\begin{array}{ccccccccccc}
 0 & 0 & 0 & -\frac{1}{48} & -\frac{1}{24} & -\frac{1}{16} & -\frac{1}{24} & -\frac{1}{48} & 0 & 0 & 0 \\
 0 & 0 & 0 & 0 & 0 & 0 & 0 & 0 & 0 & 0 & 0 \\
 -\frac{2}{81} & -\frac{5}{162} & 0 & \frac{89}{432} & \frac{89}{216} & \frac{9}{16} & \frac{89}{216} & \frac{89}{432} & 0 & -\frac{5}{162} & -\frac{2}{81} \\
 -\frac{4}{81} & -\frac{5}{81} & 0 & \frac{10}{27} & \frac{20}{27} & 1 & \frac{20}{27} & \frac{10}{27} & 0 & -\frac{5}{81} & -\frac{4}{81} \\
 -\frac{2}{81} & -\frac{5}{162} & 0 & \frac{89}{432} & \frac{89}{216} & \frac{9}{16} & \frac{89}{216} & \frac{89}{432} & 0 & -\frac{5}{162} & -\frac{2}{81} \\
 0 & 0 & 0 & 0 & 0 & 0 & 0 & 0 & 0 & 0 & 0 \\
 0 & 0 & 0 & -\frac{1}{48} & -\frac{1}{24} & -\frac{1}{16} & -\frac{1}{24} & -\frac{1}{48} & 0 & 0 & 0 \\
\end{array}
\right)}, \]
\[
\arraycolsep=3pt\def\arraystretch{1.4}
\mathbf{a}_{M,3} = {\small
\left(
\begin{array}{ccccccccccccccccc}
 0 & 0 & 0 & 0 & 0 & 0 & \frac{1}{256} & \frac{1}{128} & \frac{3}{256} & \frac{1}{128} & \frac{1}{256} & 0 & 0 & 0 & 0 & 0 & 0 \\
 0 & 0 & 0 & 0 & 0 & 0 & 0 & 0 & 0 & 0 & 0 & 0 & 0 & 0 & 0 & 0 & 0 \\
 0 & 0 & 0 & \frac{1}{324} & \frac{5}{1296} & 0 & -\frac{241}{6912} & -\frac{241}{3456} & -\frac{25}{256} & -\frac{241}{3456} & -\frac{241}{6912} & 0 & \frac{5}{1296} & \frac{1}{324} & 0 & 0 & 0 \\
 0 & 0 & 0 & 0 & 0 & 0 & 0 & 0 & 0 & 0 & 0 & 0 & 0 & 0 & 0 & 0 & 0 \\
 \frac{7}{1458} & \frac{4}{729} & 0 & -\frac{121}{2916} & -\frac{605}{11664} & 0 & \frac{20809}{93312} & \frac{20809}{46656} & \frac{75}{128} & \frac{20809}{46656} & \frac{20809}{93312} & 0 & -\frac{605}{11664} & -\frac{121}{2916} & 0 & \frac{4}{729} & \frac{7}{1458} \\
 \frac{7}{729} & \frac{8}{729} & 0 & -\frac{56}{729} & -\frac{70}{729} & 0 & \frac{280}{729} & \frac{560}{729} & 1 & \frac{560}{729} & \frac{280}{729} & 0 & -\frac{70}{729} & -\frac{56}{729} & 0 & \frac{8}{729} & \frac{7}{729} \\
 \frac{7}{1458} & \frac{4}{729} & 0 & -\frac{121}{2916} & -\frac{605}{11664} & 0 & \frac{20809}{93312} & \frac{20809}{46656} & \frac{75}{128} & \frac{20809}{46656} & \frac{20809}{93312} & 0 & -\frac{605}{11664} & -\frac{121}{2916} & 0 & \frac{4}{729} & \frac{7}{1458} \\
 0 & 0 & 0 & 0 & 0 & 0 & 0 & 0 & 0 & 0 & 0 & 0 & 0 & 0 & 0 & 0 & 0 \\
 0 & 0 & 0 & \frac{1}{324} & \frac{5}{1296} & 0 & -\frac{241}{6912} & -\frac{241}{3456} & -\frac{25}{256} & -\frac{241}{3456} & -\frac{241}{6912} & 0 & \frac{5}{1296} & \frac{1}{324} & 0 & 0 & 0 \\
 0 & 0 & 0 & 0 & 0 & 0 & 0 & 0 & 0 & 0 & 0 & 0 & 0 & 0 & 0 & 0 & 0 \\
 0 & 0 & 0 & 0 & 0 & 0 & \frac{1}{256} & \frac{1}{128} & \frac{3}{256} & \frac{1}{128} & \frac{1}{256} & 0 & 0 & 0 & 0 & 0 & 0 \\
\end{array}
\right)}.
\]
\end{example}

\begin{example} \label{ex:interp25}
We focus our attention on the case $M = { \displaystyle \begin{pmatrix} 2 & 0 \\ 0 & 5 \end{pmatrix}}$. For $n = 1,2$, the masks of the anisotropic interpolatory subdvision scheme $\mathcal{S}_{\mathbf{a}_{M,n}}$ in Definition \ref{d:interp_anisotropic} are
\[
\arraycolsep=3pt\def\arraystretch{1.4}
\mathbf{a}_{M,1} = \left(
\begin{array}{ccccccccc}
 \frac{1}{10} & \frac{1}{5} & \frac{3}{10} & \frac{2}{5} & \frac{1}{2} & \frac{2}{5} & \frac{3}{10} & \frac{1}{5} & \frac{1}{10} \\
 \frac{1}{5} & \frac{2}{5} & \frac{3}{5} & \frac{4}{5} & 1 & \frac{4}{5} & \frac{3}{5} & \frac{2}{5} & \frac{1}{5} \\
 \frac{1}{10} & \frac{1}{5} & \frac{3}{10} & \frac{2}{5} & \frac{1}{2} & \frac{2}{5} & \frac{3}{10} & \frac{1}{5} & \frac{1}{10} \\
\end{array}
\right), \]
\[
\arraycolsep=3pt\def\arraystretch{1.4}
\mathbf{a}_{M,2} = {\small
\left(
\begin{array}{ccccccccccccccccccc}
 0 & 0 & 0 & 0 & 0 & -\frac{1}{80} & -\frac{1}{40} & -\frac{3}{80} & -\frac{1}{20} & -\frac{1}{16} & -\frac{1}{20} & -\frac{3}{80} & -\frac{1}{40} & -\frac{1}{80} & 0 & 0 & 0 & 0 & 0 \\
 0 & 0 & 0 & 0 & 0 & 0 & 0 & 0 & 0 & 0 & 0 & 0 & 0 & 0 & 0 & 0 & 0 & 0 & 0 \\
 -\frac{2}{125} & -\frac{7}{250} & -\frac{4}{125} & -\frac{3}{125} & 0 & \frac{241}{2000} & \frac{249}{1000} & \frac{747}{2000} & \frac{241}{500} & \frac{9}{16} & \frac{241}{500} & \frac{747}{2000} & \frac{249}{1000} & \frac{241}{2000} & 0 & -\frac{3}{125} & -\frac{4}{125} & -\frac{7}{250} & -\frac{2}{125} \\
 -\frac{4}{125} & -\frac{7}{125} & -\frac{8}{125} & -\frac{6}{125} & 0 & \frac{27}{125} & \frac{56}{125} & \frac{84}{125} & \frac{108}{125} & 1 & \frac{108}{125} & \frac{84}{125} & \frac{56}{125} & \frac{27}{125} & 0 & -\frac{6}{125} & -\frac{8}{125} & -\frac{7}{125} & -\frac{4}{125} \\
 -\frac{2}{125} & -\frac{7}{250} & -\frac{4}{125} & -\frac{3}{125} & 0 & \frac{241}{2000} & \frac{249}{1000} & \frac{747}{2000} & \frac{241}{500} & \frac{9}{16} & \frac{241}{500} & \frac{747}{2000} & \frac{249}{1000} & \frac{241}{2000} & 0 & -\frac{3}{125} & -\frac{4}{125} & -\frac{7}{250} & -\frac{2}{125} \\
 0 & 0 & 0 & 0 & 0 & 0 & 0 & 0 & 0 & 0 & 0 & 0 & 0 & 0 & 0 & 0 & 0 & 0 & 0 \\
 0 & 0 & 0 & 0 & 0 & -\frac{1}{80} & -\frac{1}{40} & -\frac{3}{80} & -\frac{1}{20} & -\frac{1}{16} & -\frac{1}{20} & -\frac{3}{80} & -\frac{1}{40} & -\frac{1}{80} & 0 & 0 & 0 & 0 & 0 \\
\end{array}
\right)}.
\]
\end{example}

\subsection{Approximating grid transfer operators}
\label{subsec:approximating_gto}

In this section, we focus our attention on the case $M = {\displaystyle \begin{pmatrix} 2 & 0 \\ 0 & 3 \end{pmatrix}}$.
The following result is a direct consequence of Theorem \ref{t:subd_vcycle}. We omit the proof of Proposition \ref{p:subd_anisotropic_approx_vcycle} since it follows by the same argument as in the proof of Proposition \ref{p:subd_anisotropic_vcycle}.

\begin{proposition} \label{p:subd_anisotropic_approx_vcycle}
Let $n \in \NN$, $n \geq q$, and $\ell \in \set{0, \ldots, n-1}$. Then, the trigonometric polynomial $B_{n, \ell}$ associated to the anisotropic symmetric four directional approximating scheme $\mathcal{S}_{B_{n,\ell}}$ in Definition \ref{d:pseudo_anisotropic} satisfies conditions (i) and (ii) in \eqref{eq:opt_Vcycle_conditions}.
\end{proposition}

In Example \ref{ex:approx23}, we give several examples of masks of the anisotropic symmetric four directional approximating scheme $\mathcal{S}_{B_{n,\ell}}$ from Definition \ref{d:pseudo_anisotropic}. 
The corresponding grid transfer operators will be compared with the ones corresponding to the masks in Example \ref{ex:interp23} for our multigrid experiments in subsection \ref{subsec:numerical_example}.

\begin{example} \label{ex:approx23}
Let $n=1$, $\ell=0$. The mask of the anisotropic symmetric four directional approximating scheme $\mathcal{S}_{B_{1,0}}$ in Definition \ref{d:pseudo_anisotropic} is equal to the mask of the anisotropic interpolatory subdvision scheme $\mathcal{S}_{\mathbf{a}_{M,0}}$ in Definition \ref{d:interp_anisotropic}, namely 
\[
\arraycolsep=3pt\def\arraystretch{1.4}
\mathbf{B}_{1,0} = \left(
\begin{array}{ccccc}
 \frac{1}{6} & \frac{1}{3} & \frac{1}{2} & \frac{1}{3} & \frac{1}{6} \\
 \frac{1}{3} & \frac{2}{3} & 1 & \frac{2}{3} & \frac{1}{3} \\
 \frac{1}{6} & \frac{1}{3} & \frac{1}{2} & \frac{1}{3} & \frac{1}{6} \\
\end{array}
\right).
\]
Let $n=2$. For $\ell=0,1$, the masks of the anisotropic symmetric four directional approximating schemes $\mathcal{S}_{B_{2,\ell}}$ in Definition \ref{d:pseudo_anisotropic} are
\[
\arraycolsep=3pt\def\arraystretch{1.4}
\mathbf{B}_{2,0} = \left(
\begin{array}{ccccccccc}
 0 & \frac{1}{108} & \frac{1}{24} & \frac{1}{12} & \frac{23}{216} & \frac{1}{12} & \frac{1}{24} & \frac{1}{108} & 0 \\
 \frac{1}{54} & \frac{2}{27} & \frac{5}{27} & \frac{8}{27} & \frac{19}{54} & \frac{8}{27} & \frac{5}{27} & \frac{2}{27} & \frac{1}{54} \\
 \frac{1}{27} & \frac{7}{54} & \frac{31}{108} & \frac{23}{54} & \frac{53}{108} & \frac{23}{54} & \frac{31}{108} & \frac{7}{54} & \frac{1}{27} \\
 \frac{1}{54} & \frac{2}{27} & \frac{5}{27} & \frac{8}{27} & \frac{19}{54} & \frac{8}{27} & \frac{5}{27} & \frac{2}{27} & \frac{1}{54} \\
 0 & \frac{1}{108} & \frac{1}{24} & \frac{1}{12} & \frac{23}{216} & \frac{1}{12} & \frac{1}{24} & \frac{1}{108} & 0 \\
\end{array}
\right), \]
\[
\arraycolsep=3pt\def\arraystretch{1.4}
\mathbf{B}_{2,1} = \left(
\begin{array}{ccccccccccc}
 0 & 0 & 0 & -\frac{1}{48} & -\frac{1}{24} & -\frac{1}{16} & -\frac{1}{24} & -\frac{1}{48} & 0 & 0 & 0 \\
 0 & 0 & \frac{1}{108} & 0 & 0 & -\frac{1}{54} & 0 & 0 & \frac{1}{108} & 0 & 0 \\
 -\frac{2}{81} & -\frac{5}{162} & 0 & \frac{89}{432} & \frac{89}{216} & \frac{9}{16} & \frac{89}{216} & \frac{89}{432} & 0 & -\frac{5}{162} & -\frac{2}{81} \\
 -\frac{4}{81} & -\frac{5}{81} & -\frac{1}{54} & \frac{10}{27} & \frac{20}{27} & \frac{28}{27} & \frac{20}{27} & \frac{10}{27} & -\frac{1}{54} & -\frac{5}{81} & -\frac{4}{81} \\
 -\frac{2}{81} & -\frac{5}{162} & 0 & \frac{89}{432} & \frac{89}{216} & \frac{9}{16} & \frac{89}{216} & \frac{89}{432} & 0 & -\frac{5}{162} & -\frac{2}{81} \\
 0 & 0 & \frac{1}{108} & 0 & 0 & -\frac{1}{54} & 0 & 0 & \frac{1}{108} & 0 & 0 \\
 0 & 0 & 0 & -\frac{1}{48} & -\frac{1}{24} & -\frac{1}{16} & -\frac{1}{24} & -\frac{1}{48} & 0 & 0 & 0 \\
\end{array}
\right).
\]
Finally, let $n=3$. For $\ell=0,1,2$, the masks of the anisotropic symmetric four directional approximating schemes $\mathcal{S}_{B_{3,\ell}}$ in Definition \ref{d:pseudo_anisotropic} are
\[
\arraycolsep=3pt\def\arraystretch{1.4}
\mathbf{B}_{3,0} = {\small
\left(
\begin{array}{ccccccccccccc}
 0 & \frac{1}{3888} & \frac{13}{7776} & \frac{7}{1296} & \frac{5}{432} & \frac{23}{1296} & \frac{53}{2592} & \frac{23}{1296} & \frac{5}{432} & \frac{7}{1296} & \frac{13}{7776} & \frac{1}{3888} & 0 \\
 \frac{1}{1944} & \frac{7}{1944} & \frac{55}{3888} & \frac{71}{1944} & \frac{5}{72} & \frac{65}{648} & \frac{49}{432} & \frac{65}{648} & \frac{5}{72} & \frac{71}{1944} & \frac{55}{3888} & \frac{7}{1944} & \frac{1}{1944} \\
 \frac{1}{486} & \frac{47}{3888} & \frac{323}{7776} & \frac{379}{3888} & \frac{25}{144} & \frac{313}{1296} & \frac{233}{864} & \frac{313}{1296} & \frac{25}{144} & \frac{379}{3888} & \frac{323}{7776} & \frac{47}{3888} & \frac{1}{486} \\
 \frac{1}{324} & \frac{17}{972} & \frac{113}{1944} & \frac{43}{324} & \frac{25}{108} & \frac{103}{324} & \frac{229}{648} & \frac{103}{324} & \frac{25}{108} & \frac{43}{324} & \frac{113}{1944} & \frac{17}{972} & \frac{1}{324} \\
 \frac{1}{486} & \frac{47}{3888} & \frac{323}{7776} & \frac{379}{3888} & \frac{25}{144} & \frac{313}{1296} & \frac{233}{864} & \frac{313}{1296} & \frac{25}{144} & \frac{379}{3888} & \frac{323}{7776} & \frac{47}{3888} & \frac{1}{486} \\
 \frac{1}{1944} & \frac{7}{1944} & \frac{55}{3888} & \frac{71}{1944} & \frac{5}{72} & \frac{65}{648} & \frac{49}{432} & \frac{65}{648} & \frac{5}{72} & \frac{71}{1944} & \frac{55}{3888} & \frac{7}{1944} & \frac{1}{1944} \\
 0 & \frac{1}{3888} & \frac{13}{7776} & \frac{7}{1296} & \frac{5}{432} & \frac{23}{1296} & \frac{53}{2592} & \frac{23}{1296} & \frac{5}{432} & \frac{7}{1296} & \frac{13}{7776} & \frac{1}{3888} & 0 \\
\end{array}
\right)}, \]
\[
\arraycolsep=3pt\def\arraystretch{1.4}
\mathbf{B}_{3,1} = {\small
\left(
\begin{array}{ccccccccccccccc}
 0 & 0 & 0 & 0 & -\frac{1}{576} & -\frac{1}{128} & -\frac{1}{64} & -\frac{23}{1152} & -\frac{1}{64} & -\frac{1}{128} & -\frac{1}{576} & 0 & 0 & 0 & 0 \\
 0 & 0 & \frac{1}{3888} & -\frac{7}{3888} & -\frac{11}{1296} & -\frac{5}{216} & -\frac{49}{1296} & -\frac{59}{1296} & -\frac{49}{1296} & -\frac{5}{216} & -\frac{11}{1296} & -\frac{7}{3888} & \frac{1}{3888} & 0 & 0 \\
 0 & -\frac{1}{648} & -\frac{11}{1944} & -\frac{11}{972} & -\frac{5}{1296} & \frac{1}{32} & \frac{103}{1296} & \frac{271}{2592} & \frac{103}{1296} & \frac{1}{32} & -\frac{5}{1296} & -\frac{11}{972} & -\frac{11}{1944} & -\frac{1}{648} & 0 \\
 -\frac{1}{243} & -\frac{7}{486} & -\frac{113}{3888} & -\frac{49}{3888} & \frac{257}{3888} & \frac{437}{1944} & \frac{497}{1296} & \frac{595}{1296} & \frac{497}{1296} & \frac{437}{1944} & \frac{257}{3888} & -\frac{49}{3888} & -\frac{113}{3888} & -\frac{7}{486} & -\frac{1}{243} \\
 -\frac{2}{243} & -\frac{25}{972} & -\frac{5}{108} & -\frac{1}{162} & \frac{983}{7776} & \frac{5543}{15552} & \frac{487}{864} & \frac{379}{576} & \frac{487}{864} & \frac{5543}{15552} & \frac{983}{7776} & -\frac{1}{162} & -\frac{5}{108} & -\frac{25}{972} & -\frac{2}{243} \\
 -\frac{1}{243} & -\frac{7}{486} & -\frac{113}{3888} & -\frac{49}{3888} & \frac{257}{3888} & \frac{437}{1944} & \frac{497}{1296} & \frac{595}{1296} & \frac{497}{1296} & \frac{437}{1944} & \frac{257}{3888} & -\frac{49}{3888} & -\frac{113}{3888} & -\frac{7}{486} & -\frac{1}{243} \\
 0 & -\frac{1}{648} & -\frac{11}{1944} & -\frac{11}{972} & -\frac{5}{1296} & \frac{1}{32} & \frac{103}{1296} & \frac{271}{2592} & \frac{103}{1296} & \frac{1}{32} & -\frac{5}{1296} & -\frac{11}{972} & -\frac{11}{1944} & -\frac{1}{648} & 0 \\
 0 & 0 & \frac{1}{3888} & -\frac{7}{3888} & -\frac{11}{1296} & -\frac{5}{216} & -\frac{49}{1296} & -\frac{59}{1296} & -\frac{49}{1296} & -\frac{5}{216} & -\frac{11}{1296} & -\frac{7}{3888} & \frac{1}{3888} & 0 & 0 \\
 0 & 0 & 0 & 0 & -\frac{1}{576} & -\frac{1}{128} & -\frac{1}{64} & -\frac{23}{1152} & -\frac{1}{64} & -\frac{1}{128} & -\frac{1}{576} & 0 & 0 & 0 & 0 \\
\end{array}
\right)}, \]
\[
\arraycolsep=3pt\def\arraystretch{1.4}
\mathbf{B}_{3,2} = {\small
 \left(
\begin{array}{ccccccccccccccccc}
 0 & 0 & 0 & 0 & 0 & 0 & \frac{1}{256} & \frac{1}{128} & \frac{3}{256} & \frac{1}{128} & \frac{1}{256} & 0 & 0 & 0 & 0 & 0 & 0 \\
 0 & 0 & 0 & 0 & 0 & -\frac{1}{576} & 0 & 0 & \frac{1}{288} & 0 & 0 & -\frac{1}{576} & 0 & 0 & 0 & 0 & 0 \\
 0 & 0 & 0 & \frac{1}{324} & \frac{5}{1296} & 0 & -\frac{241}{6912} & -\frac{241}{3456} & -\frac{25}{256} & -\frac{241}{3456} & -\frac{241}{6912} & 0 & \frac{5}{1296} & \frac{1}{324} & 0 & 0 & 0 \\
 0 & 0 & -\frac{1}{648} & 0 & 0 & \frac{17}{1296} & 0 & 0 & -\frac{5}{216} & 0 & 0 & \frac{17}{1296} & 0 & 0 & -\frac{1}{648} & 0 & 0 \\
 \frac{7}{1458} & \frac{4}{729} & 0 & -\frac{121}{2916} & -\frac{605}{11664} & 0 & \frac{20809}{93312} & \frac{20809}{46656} & \frac{75}{128} & \frac{20809}{46656} & \frac{20809}{93312} & 0 & -\frac{605}{11664} & -\frac{121}{2916} & 0 & \frac{4}{729} & \frac{7}{1458} \\
 \frac{7}{729} & \frac{8}{729} & \frac{1}{324} & -\frac{56}{729} & -\frac{70}{729} & -\frac{59}{2592} & \frac{280}{729} & \frac{560}{729} & \frac{449}{432} & \frac{560}{729} & \frac{280}{729} & -\frac{59}{2592} & -\frac{70}{729} & -\frac{56}{729} & \frac{1}{324} & \frac{8}{729} & \frac{7}{729} \\
 \frac{7}{1458} & \frac{4}{729} & 0 & -\frac{121}{2916} & -\frac{605}{11664} & 0 & \frac{20809}{93312} & \frac{20809}{46656} & \frac{75}{128} & \frac{20809}{46656} & \frac{20809}{93312} & 0 & -\frac{605}{11664} & -\frac{121}{2916} & 0 & \frac{4}{729} & \frac{7}{1458} \\
 0 & 0 & -\frac{1}{648} & 0 & 0 & \frac{17}{1296} & 0 & 0 & -\frac{5}{216} & 0 & 0 & \frac{17}{1296} & 0 & 0 & -\frac{1}{648} & 0 & 0 \\
 0 & 0 & 0 & \frac{1}{324} & \frac{5}{1296} & 0 & -\frac{241}{6912} & -\frac{241}{3456} & -\frac{25}{256} & -\frac{241}{3456} & -\frac{241}{6912} & 0 & \frac{5}{1296} & \frac{1}{324} & 0 & 0 & 0 \\
 0 & 0 & 0 & 0 & 0 & -\frac{1}{576} & 0 & 0 & \frac{1}{288} & 0 & 0 & -\frac{1}{576} & 0 & 0 & 0 & 0 & 0 \\
 0 & 0 & 0 & 0 & 0 & 0 & \frac{1}{256} & \frac{1}{128} & \frac{3}{256} & \frac{1}{128} & \frac{1}{256} & 0 & 0 & 0 & 0 & 0 & 0 \\
\end{array}
\right)}.
\]

\end{example}

\subsection{Numerical results}
\label{subsec:numerical_example}

In this section, we illustrate the theoretical results of Propositions \ref{p:subd_anisotropic_vcycle} and \ref{p:subd_anisotropic_approx_vcycle} with two bivariate numerical examples of the geometric multigrid method with grid transfer operators defined by the Galerkin approach applied to certain multilevel Toeplitz matrices. 
In  both examples, let
\[
\mathbf{m} = (2,m) \in \NN^2, \quad m \geq 2, \qquad \mathbf{k} = (k_1,k_2) \in \NN^2, \qquad  \ell = \min \set{k_1,k_2}-1 \in \NN.
\]
The choice $\ell = \min \set{k_1,k_2}-1$ implies that the V-cycle has full length.
The $j$-th grid of the V-cycle
\[
\mathbf{n}_j = (2^{k_1-j}-1, m^{k_2-j}-1), \qquad j = 0, \ldots, \ell,
\]
has $2^{k_1-j}, \, m^{k_2-j}$ subintervals of size $(h_j)_1 =2^{j-k_1}, \, (h_j)_2 =m^{j-k_2}$ in the coordinate directions $x_1, \, x_2$, respectively.

\noindent To define $\mathbf{b}_{\mathbf{n}_0} \in \RR^{\tilde{n}_0}$, we choose the exact solution $X \in \CC^{(\mathbf{n}_0)_2 \times (\mathbf{n}_0)_1}$ on the starting grid $\mathbf{n}_0$ as
\[
X =
\begin{bmatrix}
x_{1,1} & & \cdots & & x_{1,(\mathbf{n}_0)_1} \\
 & & & & \\
\vdots & & \ddots & & \vdots \\
 & & & & \\
x_{(\mathbf{n}_0)_2,1} & & \cdots & & x_{(\mathbf{n}_0)_2,(\mathbf{n}_0)_1} \\
\end{bmatrix}, \quad
x_{s,r} = \sin \left ( 5 \frac{\pi (s-1)}{(\mathbf{n}_0)_2-1} \right ) + \sin \left ( 5 \frac{\pi (r-1)}{(\mathbf{n}_0)_1-1} \right ), \quad s = 1, \ldots , (\mathbf{n}_0)_2, \, r = 1, \ldots , (\mathbf{n}_0)_1,
\]
we compute 
\[
\mathbf{x} = 
\begin{bmatrix}
x_{1,1} & \ldots & x_{(\mathbf{n}_0)_2,1} & x_{1,2} & \ldots & x_{(\mathbf{n}_0)_2,2} & \ldots & \ldots & x_{1,(\mathbf{n}_0)_1} & \ldots & x_{(\mathbf{n}_0)_2,(\mathbf{n}_0)_1} 
\end{bmatrix}^T \in \CC^{\tilde{n}_0} 
\]
and set $\mathbf{b}_{\mathbf{n}_0} = A_{\mathbf{n}_0} \mathbf{x} \in \CC^{\tilde{n}_0}$.

\subsubsection{Bivariate Laplacian problem}
\label{ssubsec:laplacian}

The first example we present arises from the discretization of the bivariate Laplacian problem ($q=1$) with Dirichlet boundary conditions, namely
\begin{equation} \label{eq:laplacian}
\begin{cases}
{ \displaystyle - \frac{\partial^2}{\partial x_1^2} u (x_1,x_2) - \frac{\partial^2}{\partial x_2^2}} u (x_1,x_2) = h(x_1,x_2), & \quad (x_1,x_2) \in \Omega = [0,1)^2, \\
u |_{\partial \Omega}= 0. &
\end{cases}
\end{equation}
Using finite difference discretization of order 2, for $j=0, \ldots, \ell$, the system matrices $A_{\mathbf{n}_j} = T_{\mathbf{n}_j} (f_j) \in \RR^{\tilde{n}_j \times \tilde{n}_j}$ are multilevel Toeplitz matrices defined by the trigonometric polynomials
\[
f_j (x_1,x_2) = \frac{1}{(h_j)_1^2} (2 - 2 \cos x_1) + \frac{1}{(h_j)_2^2} (2 - 2 \cos x_2), \qquad (x_1,x_2) \in [0, 2\pi)^2.
\]
Notice that $f_j$ vanishes at $(0,0)$ with order 2, thus by Propositions \ref{p:subd_anisotropic_vcycle} and \ref{p:subd_anisotropic_approx_vcycle} with $q=1$, the masks defined in Examples \ref{ex:interp23}, \ref{ex:interp25} and \ref{ex:approx23}  can be used to define the corresponding grid transfer operators. 
For an appropriate comparison, we use also the well-known \emph{bi-linear interpolation} and \emph{bi-cubic Bspline} grid transfer operators from \cite{NLL2010}, which are the $2$-directional box spline subdivision schemes in \cite{chui1988multivariate} with dilation $M= {\displaystyle \begin{pmatrix} 2 & 0 \\ 0 & 2 \end{pmatrix}}$ and masks
\[
\mathbf{P}_{1} = \frac14 \cdot \,
\left(
\begin{array}{ccc}
1 & 2 & 1 \\
2 & 4 & 2 \\ 
1 & 2 & 1
\end{array} \right ), \qquad 
\mathbf{P}_{2} = \frac{1}{64} \cdot \,
\left(
\begin{array}{ccccc}
1 & 4 & 6 & 4 & 1 \\
4 & 16 & 24 & 16 & 4 \\
6 & 24 & 36 & 24 & 6 \\
4 & 16 & 24 & 16 & 4 \\
1 & 4 & 6 & 4 & 1
\end{array} \right ).
\]
The corresponding subdivision schemes $\mathcal{S}_{\mathbf{P}_1}$ and $\mathcal{S}_{\mathbf{P}_2}$ generates polynomials up to degree 1 and 3, respectively.
Moreover, we consider the \emph{bi-cubic interpolation} grid transfer operator, known as \emph{Kobbelt} subdivision scheme (\cite{kobbelt1996interpolatory}), which is a tensor product scheme with dilation $M= {\displaystyle \begin{pmatrix} 2 & 0 \\ 0 & 2 \end{pmatrix}}$ based on the univariate binary $4$-point Dubuc-Deslauriers subdivision scheme with the symbol $\mathbf{a}_{2,2}$ in \eqref{eq:DD_interp_1D}.
Its mask is
\[
\mathbf{K} = \frac{1}{256}
\left(
\begin{array}{ccccccc}
1 & 0 & -9 & -16 & -9 & 0 & 1 \\
0 & 0 & 0 & 0 & 0 & 0 & 0 \\
-9 & 0 & 81 & 144 & 81 & 0 & -9 \\
-16 & 0 & 144 & 256 & 144 & 0 & -16 \\
-9 & 0 & 81 & 144 & 81 & 0 & -9 \\
0 & 0 & 0 & 0 & 0 & 0 & 0 \\
1 & 0 & -9 & -16 & -9 & 0 & 1
\end{array} \right ),
\]
and the associated subdivision scheme $\mathcal{S}_{\mathbf{K}}$ generates polynomials up to degree 3. We notice that $\mathcal{S}_{\mathbf{P}_1}, \mathcal{S}_{\mathbf{P}_2}$ and $\mathcal{S}_{\mathbf{K}}$ satisfy the hypothesis of Theorem \ref{t:subd_vcycle} with $q=1$.

\noindent We use as pre- and post-smoother one step of Gauss-Seidel method. The zero vector is used as the initial guess and the stopping criterion is $\norm{\mathbf{r}_s}_2 / \norm{\mathbf{r}_0}_2 < 10^{-7}$, where $\mathbf{r}_s$ is the residual vector after $s$ iterations and $10^{-7}$ is  the given tolerance.

\noindent We define the starting grid $\mathbf{n}_0$ in agreement with the dilation matrix $M$, namely
\[
\begin{array}{cc}
\medskip
\text{- for }  M={\displaystyle \begin{pmatrix} 2 & 0 \\ 0 & 2 \end{pmatrix}} : & \quad
\begin{cases}
\mathbf{n}_0 = (2^7-1,2^7-1), & \quad \text{Case } 1, \\
\mathbf{n}_0 = (2^8-1,2^8-1), & \quad \text{Case } 2,
\end{cases}  \\
\medskip
\text{- for }  M={\displaystyle \begin{pmatrix} 2 & 0 \\ 0 & 3 \end{pmatrix}} : & \quad
\begin{cases}
\mathbf{n}_0 = (2^7-1,3^4-1), & \quad \text{Case } 1, \\
\mathbf{n}_0 = (2^8-1,3^5-1), & \quad \text{Case } 2, 
\end{cases} \\
\text{- for }  M={\displaystyle \begin{pmatrix} 2 & 0 \\ 0 & 5 \end{pmatrix}} : & \quad
\begin{cases}
\mathbf{n}_0 = (2^7-1,5^3-1), & \quad \text{Case } 1, \\
\mathbf{n}_0 = (2^9-1,5^4-1), & \quad \text{Case } 2,
\end{cases}
\end{array}
\]

\noindent Table \ref{table:laplacian} shows how the number of iterations and convergence rates for the V-cycle change when the starting grid $\mathbf{n}_0$ becomes finer.
The results in Table \ref{table:laplacian} support our theoretical analysis in section \ref{sec:examples}, as they show that subdivision schemes with different dilation matrices and appropriate degree of polynomials generation define grid transfer operators capable of guaranteeing convergence and optimality of the corresponding V-cycle method. 
The grid transfer operators defined from the subdivision schemes with dilation $M={\displaystyle \begin{pmatrix} 2 & 0 \\ 0 & 2 \end{pmatrix}}$ perform better than the grid transfer operators defined from the anisotropic subdivision schemes. This happens since the bivariate Laplacian problem in \eqref{eq:laplacian} is symmetric with respect to the two coordinate directions. If we use grid transfer operators derived from subdivision schemes with dilation $M={\displaystyle \begin{pmatrix} 2 & 0 \\ 0 & 2 \end{pmatrix}}$ or, equivalently, grid transfer operators defined from the downsampling matrix with the factor $\mathbf{m} = (2,2)$, we preserve the symmetry of the problem at each $j$-th step of the V-cycle, $j=0, \ldots, \ell$.
Moreover, at each Coarse Grid Correction step, we downsample the data with the factor $\mathbf{m} = (2,m)$ and the larger is $m$ the more information we lose. Thus, the number of iterations required for convergence is larger for $m >2$. 
Finally, we notice that there is no crucial difference between polynomial generation and reproduction properties for convergence and optimality of the V-cycle method.

\begin{table} 
\centering
\small
\begin{tabular}{cc@{\qquad\;\;}cc@{\qquad\;\;}cc@{\qquad\;\;}c}
\toprule
{Dilation} & {Subdivision} & \multicolumn{2}{c@{\qquad\;\;}}{Case 1} & \multicolumn{2}{c@{\qquad\;\;}}{Case 2} &  {Generation}\\
{matrix} & {scheme} & {iter} & {conv. rate} & {iter} & {conv. rate} & {degree} \\
\toprule
\multirow{3}*{$M={\displaystyle \begin{pmatrix} 2 & 0 \\ 0 & 2 \end{pmatrix}}$} & $\mathcal{S}_{\mathbf{P}_1}$ & 9 & 0.1432 & 9 & 0.1374 & 1 \\
& $\mathcal{S}_{\mathbf{P}_2}$ & 13 & 0.2823 & 13 & 0.27 & 3 \\
& $\mathcal{S}_{\mathbf{K}}$ & 8 & 0.1224 & 8 & 0.1275 & 3 \\
\midrule
\multirow{8}*{$M={\displaystyle \begin{pmatrix} 2 & 0 \\ 0 & 3 \end{pmatrix}}$} & $\mathcal{S}_{\mathbf{a}_{M,1}}$ & 28 & 0.5573 & 23 & 0.4958 &  1 \\
& $\mathcal{S}_{\mathbf{a}_{M,2}}$ & 26 & 0.5297 & 22 & 0.4777 & 3 \\
& $\mathcal{S}_{\mathbf{a}_{M,3}}$ & 26 & 0.5347 & 23 & 0.4893 & 5 \\
& $\mathcal{S}_{\mathbf{B}_{2,0}}$ & 33 & 0.6082 & 26 & 0.5298 & 3 \\
& $\mathcal{S}_{\mathbf{B}_{2,1}}$ & 26 & 0.5298 & 22 & 0.4477 & 3 \\
& $\mathcal{S}_{\mathbf{B}_{3,0}}$ & 41 & 0.6718 & 35 & 0.6272 & 5 \\
& $\mathcal{S}_{\mathbf{B}_{3,1}}$ & 24 & 0.5096 & 22 & 0.4787 & 5 \\
& $\mathcal{S}_{\mathbf{B}_{3,2}}$ & 26 & 0.5347 & 23 & 0.4893 & 5 \\
\midrule
\multirow{2}*{$M={\displaystyle \begin{pmatrix} 2 & 0 \\ 0 & 5 \end{pmatrix}}$} & $\mathcal{S}_{\mathbf{a}_{M,1}}$ & 38 & 0.6529 & 45 & 0.6969 &  1 \\
& $\mathcal{S}_{\mathbf{a}_{M,2}}$ & 38 & 0.6532 & 40 & 0.6774 & 3 \\
\bottomrule
\end{tabular}
\caption{Laplacian problem.}
\label{table:laplacian}
\end{table}

\subsubsection{Bivariate anisotropic Laplacian problem}
\label{ssubsec:laplacian_anisotropic}

The second example we present arises from the discretization of the bivariate anisotropic Laplacian problem ($q=1$) with Dirichlet boundary conditions, namely
\begin{equation} \label{eq:laplacian_anisotropic}
\begin{cases}
{ \displaystyle - \varepsilon \frac{\partial^2}{\partial x_1^2} u (x_1,x_2) - \frac{\partial^2}{\partial x_2^2}} u (x_1,x_2) = h(x_1,x_2), & \quad (x_1,x_2) \in \Omega = [0,1)^2, \\
u |_{\partial \Omega}= 0. &
\end{cases}
\end{equation}
The parameter $\varepsilon \in (0,1]$ in \eqref{eq:laplacian_anisotropic} is called \emph{anisotropy}. If $\varepsilon=1$, we get the standard isotropic Laplacian problem \eqref{eq:laplacian}. 
If $\varepsilon << 1$, the problem becomes strongly anisotropic. We focus our attention on the latter case \citep{mgm-book}.

\noindent Using finite difference discretization of order 2, for $j=0, \ldots, \ell$, the system matrices $A_{\mathbf{n}_j} = T_{\mathbf{n}_j} (f_j) \in \RR^{\tilde{n}_j \times \tilde{n}_j}$ are multilevel Toeplitz matrices defined by the trigonometric polynomials
\[
f_j^{(\varepsilon)} (x_1,x_2) = \frac{\varepsilon}{(h_j)_1^2} (2 - 2 \cos x_1) + \frac{1}{(h_j)_2^2} (2 - 2 \cos x_2), \qquad (x_1,x_2) \in [0, 2\pi)^2.
\]

Let $\mathbf{m} = (2,2) \in \NN^2$, $\mathbf{k} = (k,k) \in \NN^2$ and $\ell = k-1$. For $j=0, \ldots, \ell$, the $j$-th grid of the V-cycle is symmetric, namely
\[
\mathbf{n}_j = (2^{k-j}-1, 2^{k-j}-1), \qquad (h_j)_1 = (h_j)_2 = 2^{-(k-j)}.
\]
Thus, we can rewrite the trigonometric polynomials $f_j^{(\varepsilon)}$, $j=0, \ldots, \ell$, as
\[
f_j^{(\varepsilon)} (x_1,x_2) = 2^{2(k-j)} \biggl ( \varepsilon \, (2 - 2 \cos x_1) + (2 - 2 \cos x_2) \biggr ), \qquad (x_1,x_2) \in [0, 2\pi)^2.
\]
If $\varepsilon << 1$, the symbol $f_j^{(\varepsilon)}$ is numerically close to 0 on the entire line $x_2 = 0$, for all $j=0, \ldots, \ell$ (see Figure \ref{fig:laplacian_epsilon}).  Due to this pathology, when the anisotropy $\varepsilon$ goes to 0, the number of iterations necessary for the convergence of the V-cycle method rises because the symbol vanishes on a whole curve and hence conditions $(i)$ and $(ii)$ cannot be satisfied together \cite{Fischer2006}. 

\begin{figure}
\centering
\subfloat[Isotropic Laplacian: $\varepsilon = 1$]
{\includegraphics[scale=0.5,trim={1cm 0.5cm 1cm 0.5cm},clip]{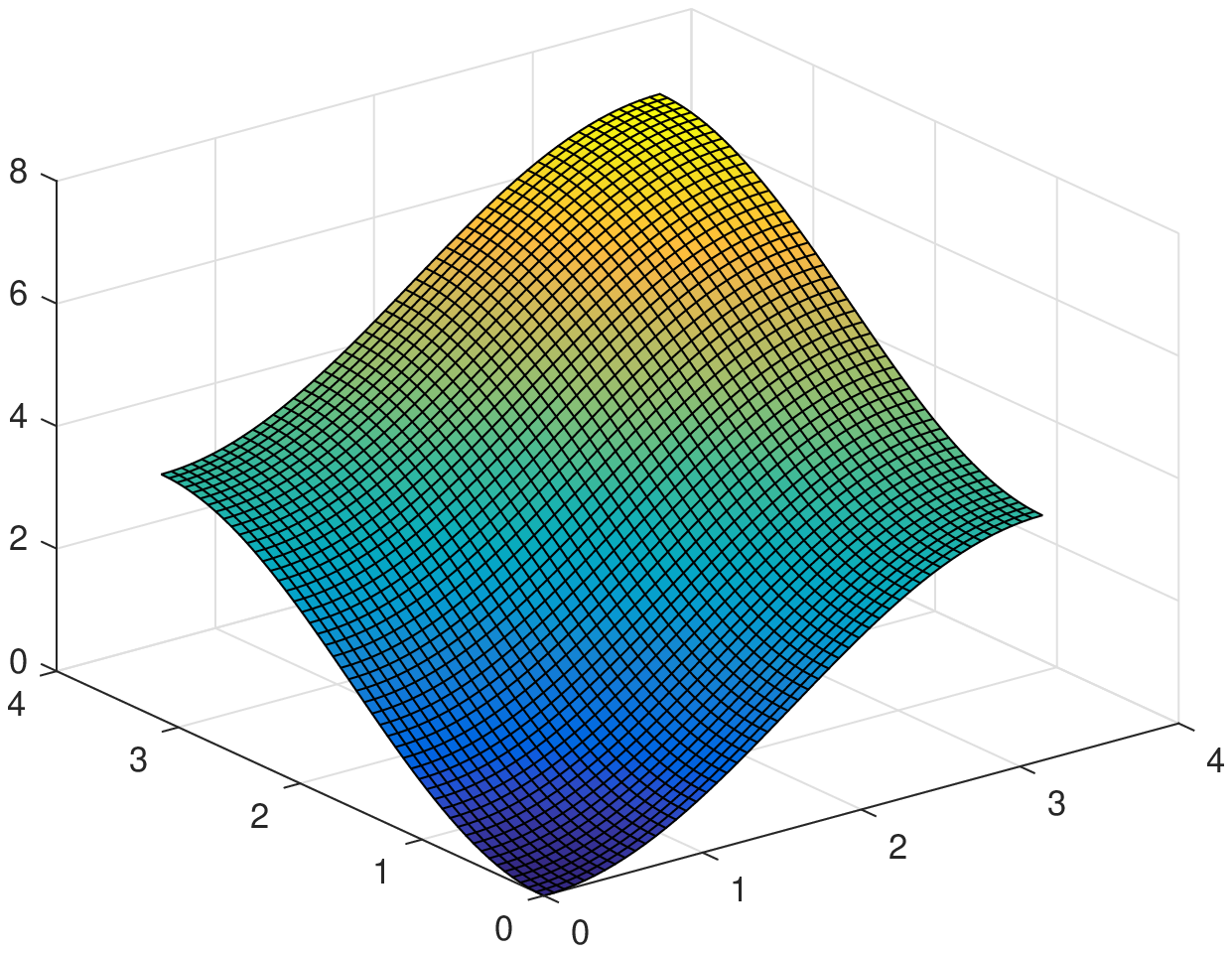}} \qquad
\subfloat[Strongly anisotropic Laplacian: $\varepsilon=10^{-2}$]
{\includegraphics[scale=0.5,trim={1cm 0.5cm 1cm 0.5cm},clip]{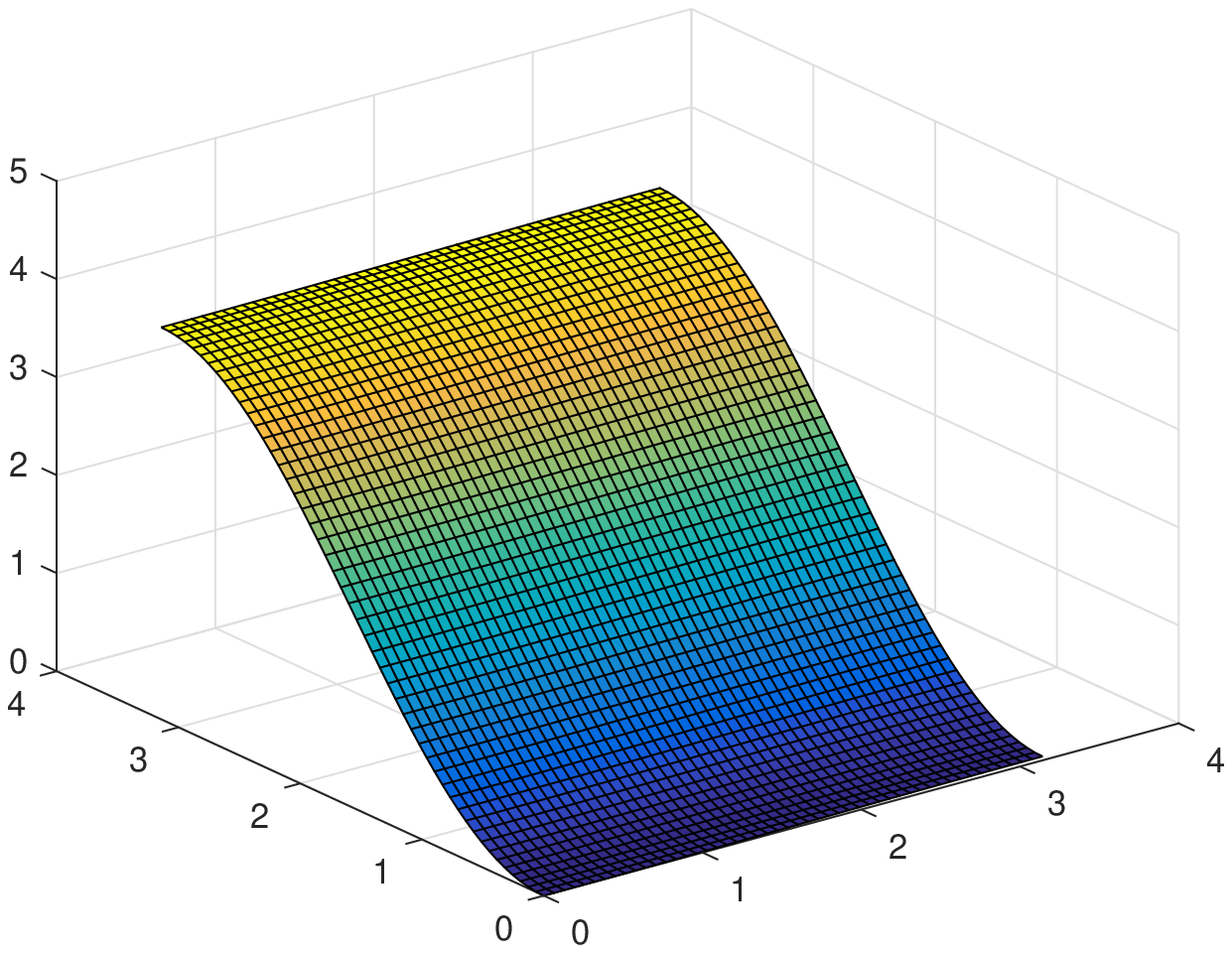}}
\caption{Plot of $f_0^{(\varepsilon)}$ on the reference interval $[0,\pi)^2$ for different values of $\varepsilon \in (0,1]$.}
\label{fig:laplacian_epsilon}
\end{figure}

Let $\mathbf{m} = (2,m) \in \NN^2$, $m > 2$, and define $\mathbf{k} = (k_1,k_2) \in \NN^2$ such that $k_2 = \max \set{k \in \NN \, : \, m^{k}-1 \leq 2^{k_1}-1}$.
We can rewrite the polynomials $f_j^{(\varepsilon)}$, $j=0, \ldots, \ell$,  as
\[
\begin{split}
f_j^{(\varepsilon)} (x_1,x_2) &= \frac{\varepsilon}{(h_j)_1^2} (2 - 2 \cos x_1) + \frac{1}{(h_j)_2^2} (2 - 2 \cos x_2), \\
&= \frac{1}{(h_j)_2^2} \Biggl ( \frac{\varepsilon (h_j)_2^2}{(h_j)_1^2} \, (2 - 2 \cos x_1) + (2 - 2 \cos x_2) \Biggr ), \\
&=  m^{2(k_2-j)} \Biggl ( \varepsilon \,  \frac{2^{2(k_1-j)}}{m^{2(k_2-j)}} \, (2 - 2 \cos x_1) + (2 - 2 \cos x_2) \Biggr ), \\
&=  m^{2(k_2-j)} \Biggl ( \varepsilon_j \, (2 - 2 \cos x_1) + (2 - 2 \cos x_2) \Biggr ), \qquad \varepsilon_j =  \varepsilon \frac{2^{2(k_1-j)}}{m^{2(k_2-j)}}, \qquad (x_1,x_2) \in [0, 2\pi)^2.
\end{split}
\]
The value $\varepsilon_j$ represents the anisotropy of the discretized problem \ref{eq:laplacian_anisotropic} on the $j$-th grid $\mathbf{n}_j$ of the V-cycle, $j=0, \ldots, \ell$.
Especially, we have  
\begin{equation} \label{eq:epsilon_j}
 \varepsilon_j = \varepsilon \, \frac{2^{2(k_1-j)}}{m^{2(k_2-j)}} = \frac{m^2}{4} \Biggl ( \varepsilon \, \frac{2^{2(k_1-(j-1))}}{m^{2(k_2-(j-1))}} \Biggr ) = \frac{m^2}{4} \varepsilon_{j-1} > \varepsilon_{j-1}, \qquad j=1, \ldots, \ell.
\end{equation}
This means that the matrix $A_{\mathbf{n}_j} = T_{\mathbf{n}_j} (f_j)$ at the $j$-th level of the V-cycle is less anisotropic than the matrix $A_{\mathbf{n}_{j-1}} = T_{\mathbf{n}_{j-1}} (f_{j-1})$ at the $(j-1)$-th level of the V-cycle, $j=1, \ldots, \ell$.
Motivated by this property and the observations related to the standard Laplacian problem in subsection \ref{ssubsec:laplacian}, we propose a multigrid strategy which combines both anisotropic and symmetric cutting strategies.
More precisely, we define the starting grid $\mathbf{n}_0$ by
\begin{equation} \label{eq:starting_grid}
\mathbf{n}_0 = (2^{k_1}-1, m^{h} \cdot 2^{k_2}-1), \qquad h \in \NN,
\end{equation}
and we choose $\mathbf{k}=(k_1,k_2) \in \NN^2$ such that $k_2 = \max \set{k \in \NN \, : \,  m^{h} \cdot 2^{k}-1 \leq 2^{k_1}-1}$. 
We fix $\ell = \min \set{k_1, h+k_2}-1$, in order to guarantee a V-cycle method with full length. Then, we define the $j$-th grid of the V-cycle by
\[
\mathbf{n}_j = 
\begin{cases}
(2^{k_1-j}-1, m^{h-j} \cdot 2^{k_2}-1), & j = 0, \ldots, h, \\
(2^{k_1-j}-1, 2^{k_2-(j-h)}-1), & j = h+1, \ldots, \ell.
\end{cases}
\]
Finally, we construct the grid transfer operators $P_{\mathbf{n}_j}$ as subdivision schemes with dilation $M = {\displaystyle \begin{pmatrix} 2 & 0 \\ 0 & m\end{pmatrix}}$ for $j=0, \ldots, h$, and as subdivision schemes with dilation $M = {\displaystyle \begin{pmatrix} 2 & 0 \\ 0 & 2\end{pmatrix}}$ for $j=h+1, \ldots, \ell$. Especially, for our numerical experiments, we use the bi-linear interpolation grid transfer operator for $j=h+1, \ldots, \ell$. 
If we choose $h \in \NN$ properly, due to \eqref{eq:epsilon_j}, we can handle the anisotropy of the problem in $h$ steps of the V-cycle. Thus, for $j=h+1, \ldots, \ell$, a symmetric cutting strategy performs better than an anisotropic cutting strategy.

\noindent For the numerical experiments, we use as pre- and post-smoother one step of Gauss-Seidel method for $j=1, \ldots, \ell$, and 2 steps of Gauss-Seidel method for $j=0$. The zero vector is used as the initial guess and the stopping criterion is $\norm{\mathbf{r}_s}_2 / \norm{\mathbf{r}_0}_2 < 10^{-5}$, where $\mathbf{r}_s$ is the residual vector after $s$ iterations and $10^{-5}$ is the given tolerance. 

\noindent We define the starting grid $\mathbf{n}_0$ by \eqref{eq:starting_grid}, namely
\[
\begin{array}{cc}
\medskip
\text{- for }  M={\displaystyle \begin{pmatrix} 2 & 0 \\ 0 & 2 \end{pmatrix}} : & \quad
\begin{cases}
\mathbf{n}_0 = (2^7-1,2^7-1), & \quad \text{Case } 1, \\
\mathbf{n}_0 = (2^8-1,2^8-1), & \quad \text{Case } 2,
\end{cases}  \\
\medskip
\text{- for }  M={\displaystyle \begin{pmatrix} 2 & 0 \\ 0 & 3 \end{pmatrix}} : & \quad
\begin{cases}
\mathbf{n}_0 = (2^7-1,3^2 \cdot 2^3-1), & \quad \text{Case } 1, \\
\mathbf{n}_0 = (2^8-1,3^2 \cdot 2^4-1), & \quad \text{Case } 2, 
\end{cases} \\
\text{- for }  M={\displaystyle \begin{pmatrix} 2 & 0 \\ 0 & 5 \end{pmatrix}} : & \quad
\begin{cases}
\mathbf{n}_0 = (2^8-1,5 \cdot 2^5-1), & \quad \text{Case } 1, \\
\mathbf{n}_0 = (2^8-1,5^2 \cdot 2^3-1), & \quad \text{Case } 2,
\end{cases}
\end{array}
\]

Tables \ref{table:laplacian_anisotropic_10-2} and \ref{table:laplacian_anisotropic_10-3}  show how the number of iterations and convergence rates for the V-cycle change when the starting grid $\mathbf{n}_0$ becomes finer and the anisotropy $\varepsilon$ in \eqref{eq:laplacian_anisotropic} decreases.
The results support our theoretical analysis. Especially, the grid transfer operators defined from the anisotropic subdivision schemes with dilation $ M={\displaystyle \begin{pmatrix} 2 & 0 \\ 0 & 3 \end{pmatrix}}$ perform better than all the other grid transfer operators. Indeed, after 2 steps of downsampling with the factor $\mathbf{m} = (2,3)$, the anisotropy of the problem increases by a factor $\frac{81}{16} \approx 5$. Moreover, when we downsample the data with the factor $\mathbf{m} = (2,3)$ we lose less information than when we sample the data with the factor $\mathbf{m} = (2,5)$.
Among the grid transfer operators defined from the anisotropic subdivision schemes with dilation $ M={\displaystyle \begin{pmatrix} 2 & 0 \\ 0 & 3 \end{pmatrix}}$, we pay special attention to the interpolatory ones.
The advantage of using the anisotropic interpolatory subdivision schemes is the computational efficiency of the corresponding grid transfer operations. Indeed, the matrices $A_{\mathbf{n}_j}$, $j=0,\dots,\ell$, are independent of the grid transfer operators 
and the computational cost of the restriction and prolongation depends only on the number of nonzero entries of the corresponding operators.
Therefore, since for a fixed $n \in \NN$ the mask $\mathbf{a}_{M,n}$ of the interpolatory subdivision schemes $\mathcal{S}_{\mathbf{a}_{M,n}}$ in Definition \ref{d:interp_anisotropic} has less nonzero entries than the masks $\mathbf{B}_{n, \ell}$, $\ell = 0, \ldots, n-1$, of the approximating subdivision schemes $\mathcal{S}_{\mathbf{B}_{n, \ell}}$ in Definition \ref{d:pseudo_anisotropic}, each iteration of the V-cycle method with the interpolatory grid transfer operator associated to $\mathcal{S}_{\mathbf{a}_{M,n}}$ is cheaper than one V-cycle iteration with the approximating grid transfer operators associated to $\mathcal{S}_{\mathbf{B}_{n, \ell}}$.
Finally, we notice that there is no crucial difference between polynomial generation and reproduction properties for convergence and optimality of the V-cycle method.

\begin{table} 
\centering
\small
\begin{tabular}{cc@{\qquad\;\;}cc@{\qquad\;\;}cc@{\qquad\;\;}c}
\toprule
{Dilation} & {Subdivision} & \multicolumn{2}{c@{\qquad\;\;}}{Case 1} & \multicolumn{2}{c@{\qquad\;\;}}{Case 2} &  {Generation}\\
{matrix} & {scheme} & {iter} & {conv. rate} & {iter} & {conv. rate} & {degree} \\
\toprule
\multirow{3}*{$M={\displaystyle \begin{pmatrix} 2 & 0 \\ 0 & 2 \end{pmatrix}}$} & $\mathcal{S}_{\mathbf{P}_1}$ & 75 & 0.8571 & 80 & 0.8658 & 1 \\
& $\mathcal{S}_{\mathbf{P}_2}$ & 82 & 0.8686 & 86 & 0.8744 & 3 \\
& $\mathcal{S}_{\mathbf{K}}$ & 61 & 0.8273 &76 & 0.8585 & 3 \\
\midrule
\multirow{8}*{$M={\displaystyle \begin{pmatrix} 2 & 0 \\ 0 & 3 \end{pmatrix}}$} & $\mathcal{S}_{\mathbf{a}_{M,1}}$ & 14 & 0.4315 & 16 & 0.4807 &  1 \\
& $\mathcal{S}_{\mathbf{a}_{M,2}}$ & 14 & 0.4307 & 16 & 0.48 & 3 \\
& $\mathcal{S}_{\mathbf{a}_{M,3}}$ & 14 & 0.4312 & 16 & 0.4806 & 5 \\
& $\mathcal{S}_{\mathbf{B}_{2,0}}$ & 13 & 0.5145 & 16 & 0.4780 & 3 \\
& $\mathcal{S}_{\mathbf{B}_{2,1}}$ & 14 & 0.4307 & 16 & 0.48 & 3 \\
& $\mathcal{S}_{\mathbf{B}_{3,0}}$ & 14 & 0.4363 & 17 & 0.5003 & 5 \\
& $\mathcal{S}_{\mathbf{B}_{3,1}}$ & 13 & 0.4112 & 15 & 0.4633 & 5 \\
& $\mathcal{S}_{\mathbf{B}_{3,2}}$ & 14 & 0.4312 & 16 & 0.4806 & 5 \\
\midrule
\multirow{2}*{$M={\displaystyle \begin{pmatrix} 2 & 0 \\ 0 & 5 \end{pmatrix}}$} & $\mathcal{S}_{\mathbf{a}_{M,1}}$ & 20 & 0.5623 & 25 & 0.6307 & 1 \\
& $\mathcal{S}_{\mathbf{a}_{M,2}}$ & 21& 0.5719 & 26 & 0.6385 & 3 \\
\bottomrule
\end{tabular}
\caption{Anisotropic Laplacian problem with $\varepsilon = 10^{-2}$.}
\label{table:laplacian_anisotropic_10-2}
\end{table}

\begin{table} 
\centering
\small
\begin{tabular}{cc@{\qquad\;\;}cc@{\qquad\;\;}cc@{\qquad\;\;}c}
\toprule
{Dilation} & {Subdivision} & \multicolumn{2}{c@{\qquad\;\;}}{Case 1} & \multicolumn{2}{c@{\qquad\;\;}}{Case 2} &  {Generation}\\
{matrix} & {scheme} & {iter} & {conv. rate} & {iter} & {conv. rate} & {degree} \\
\toprule
\multirow{3}*{$M={\displaystyle \begin{pmatrix} 2 & 0 \\ 0 & 2 \end{pmatrix}}$} & $\mathcal{S}_{\mathbf{P}_1}$ & 294 & 0.9616 & 284 & 0.9603 & 1 \\
& $\mathcal{S}_{\mathbf{P}_2}$ & 295 & 0.9617 & 281 & 0.9599 & 3 \\
& $\mathcal{S}_{\mathbf{K}}$ & 253 & 0.9555 & 251 & 0.9551 & 3 \\
\midrule
\multirow{8}*{$M={\displaystyle \begin{pmatrix} 2 & 0 \\ 0 & 3 \end{pmatrix}}$} & $\mathcal{S}_{\mathbf{a}_{M,1}}$ & 33 & 0.7051 & 44 & 0.7694 &  1 \\
& $\mathcal{S}_{\mathbf{a}_{M,2}}$ & 33 & 0.7050 & 44 & 0.7695 & 3 \\
& $\mathcal{S}_{\mathbf{a}_{M,3}}$ & 33 & 0.7050 & 44 & 0.7697 & 5 \\
& $\mathcal{S}_{\mathbf{B}_{2,0}}$ & 30 & 0.6813 & 42 & 0.7592 & 3 \\
& $\mathcal{S}_{\mathbf{B}_{2,1}}$ & 33 & 0.7050 & 44 & 0.7695 & 3 \\
& $\mathcal{S}_{\mathbf{B}_{3,0}}$ & 30 & 0.6807 & 41 & 0.7540 & 5 \\
& $\mathcal{S}_{\mathbf{B}_{3,1}}$ & 31 & 0.6893 & 43 & 0.7641 & 5 \\
& $\mathcal{S}_{\mathbf{B}_{3,2}}$ & 33 & 0.7050 & 44 & 0.7697 & 5 \\
\midrule
\multirow{2}*{$M={\displaystyle \begin{pmatrix} 2 & 0 \\ 0 & 5 \end{pmatrix}}$} & $\mathcal{S}_{\mathbf{a}_{M,1}}$ & 62 & 0.8301 & 69 & 0.8462 &  1 \\
& $\mathcal{S}_{\mathbf{a}_{M,2}}$ & 62 & 0.8304 & 70 & 0.8479 & 3 \\
\bottomrule
\end{tabular}
\caption{Anisotropic Laplacian problem with $\varepsilon = 10^{-3}$.}
\label{table:laplacian_anisotropic_10-3}
\end{table}

Tables \ref{table:laplacian_anisotropic_10-2} and \ref{table:laplacian_anisotropic_10-3} justify the use of the dilation matrix $M = {\displaystyle \begin{pmatrix} 2 & 0 \\ 0 & 3 \end{pmatrix}}$ in our analysis. Note that the schemes with $M = {\displaystyle \begin{pmatrix} 2 & 0 \\ 0 & 5 \end{pmatrix}}$ have a slower convergence rate, which is influenced by the larger support sizes of their masks and by the less efficient approximation caused by inappropriate coarsening of the mesh in the $y$ direction.

\section{Conclusions} \label{sec:conclusion}

In this paper, we have constructed a family of bivariate interpolatory subdivision schemes with dilation $M = \mathrm{diag }(2,m)$, $m \in \NN$ odd. We have investigated their minimality, polynomial reproduction and convergence properties.
In case of $M = \mathrm{diag }(2,3)$, we have also defined two families of bivariate approximating subdivision schemes characterized by specific polynomial generation and reproduction properties.
We have shown that these families of anisotropic subdivision schemes define powerful grid transfer operators in anisotropic geometric multigrid. 
Especially, we have confirmed their strength for the numerical solution of the anisotropic Laplacian problem with anisotropy along coordinate axes.

Our results can be extended in many directions.
A proper analysis of the anisotropic Laplacian problem with anisotropy in other directions (using the approach in \cite{Fischer2006}) is of future interest.
Moreover, we plan to study the higher regularity of the proposed interpolatory and approximating subdivision schemes for an eventual application in surface generation.

\bibliographystyle{plain}
\bibliography{bibliography}

\begin{thebibliography}{10}

\bibitem{V-cycle_opt}
A.~Aric{\`o} and M.~Donatelli.
\newblock {A V-cycle Multigrid for multilevel matrix algebras: proof of
  optimality}.
\newblock {\em Numer. Math.}, 105(4):511--547, 2007.

\bibitem{ETNA2015}
M.~Bolten, M.~Donatelli, and T.~Huckle.
\newblock {Analysis of smoothed aggregation multigrid methods based on Toeplitz
  matrices}.
\newblock {\em Electron. Trans. Numer. Anal.}, 44:25--52, 2015.

\bibitem{BIT2014}
M.~Bolten, M.~Donatelli, T.~Huckle, and C.~Kravvaritis.
\newblock {Generalized grid transfer operators for multigrid methods applied on
  Toeplitz matrices}.
\newblock {\em BIT Numerical Mathematics}, 55(2):341--366, 2015.

\bibitem{Brandt1982guide}
A.~Brandt.
\newblock Guide to multigrid development.
\newblock In {\em Multigrid methods}, pages 220--312. Springer, 1982.

\bibitem{Cabrelli2004self}
C.A. Cabrelli, C.~Heil, and U.M. Molter.
\newblock {\em Self-similarity and multiwavelets in higher dimensions}, volume
  170.
\newblock Amer. Math. Soc., 2004.

\bibitem{Cava_Dam_Micc}
A.S. Cavaretta, W.~Dahmen, and C.A. Micchelli.
\newblock {\em Stationary subdivision}, volume 453.
\newblock J. Amer. Math. Soc., 1991.

\bibitem{Charina2014reproduction}
M.~Charina, C.~Conti, and L.~Romani.
\newblock Reproduction of exponential polynomials by multivariate
  non-stationary subdivision schemes with a general dilation matrix.
\newblock {\em Numer. Math.}, 127(2):223--254, 2014.

\bibitem{Charina2016multigrid}
M.~Charina, M.~Donatelli, L.~Romani, and V.~Turati.
\newblock Multigrid methods: grid transfer operators and subdivision schemes.
\newblock {\em Linear Algebra Appl.}, 520:151--190, 2017.

\bibitem{Charina2017smoothness}
M.~Charina and V.Y. Protasov.
\newblock Smoothness of anisotropic wavelets, frames and subdivision schemes.
\newblock {\em arXiv preprint arXiv:1702.00269}, 2017.

\bibitem{chui1988multivariate}
C.K. Chui.
\newblock {\em {Multivariate splines}}.
\newblock SIAM, 1988.

\bibitem{Conti_Cotronei_Sauer_interp_multivariate}
C.~Conti, M.~Cotronei, and T.~Sauer.
\newblock Full rank interpolatory subdivision: A first encounter with the
  multivariate realm.
\newblock {\em J. Approx. Theory}, 162(3):559--575, 2010.

\bibitem{CH2011}
C.~Conti and K.~Hormann.
\newblock {Polynomial reproduction for univariate subdivision schemes of any
  arity}.
\newblock {\em J. Approx. Theory}, 163(4):413--437, 2011.

\bibitem{Conti2017pseudo}
C.~Conti, k.~Hormann, and C.~Deng.
\newblock Symmetric four-directional bivariate pseudo-splines.
\newblock {\em arXiv preprint arXiv:1706.03056}.

\bibitem{Deslauriers_Dubuc1989symmetric}
G.~Deslauriers and S.~Dubuc.
\newblock Symmetric iterative interpolation processes.
\newblock In {\em Constructive approximation}, pages 49--68. Springer, 1989.

\bibitem{Fuentes2015thesis}
R.~Diaz~Fuentes.
\newblock Perturbacion de los esquemas de dubuc-deslauriers para cualquier
  aridad.
\newblock Master's thesis, University of Havana, Cuba, 2015.

\bibitem{NLL2010}
M.~Donatelli.
\newblock {An algebraic generalization of local Fourier analysis for grid
  transfer operators in multigrid based on Toeplitz matrices}.
\newblock {\em Numer. Linear Algebra Appl.}, 17(2-3):179--197, 2010.

\bibitem{Dyn2008polynomial}
N.~Dyn, K.~Hormann, M.A. Sabin, and Z.~Shen.
\newblock Polynomial reproduction by symmetric subdivision schemes.
\newblock {\em J. Approx. Theory}, 155(1):28--42, 2008.

\bibitem{Dyn_Lev}
N.~Dyn and D.~Levin.
\newblock {Subdivision schemes in geometric modelling}.
\newblock {\em Acta Numer.}, 11:73--144, 2002.

\bibitem{Eirola1992sobolev}
T.~Eirola.
\newblock Sobolev characterization of solutions of dilation equations.
\newblock {\em SIAM J. Math. Anal.}, 23(4):1015--1030, 1992.

\bibitem{Fischer2006}
R.~Fischer and T.~Huckle.
\newblock Multigrid methods for anisotropic bttb systems.
\newblock {\em Linear Algebra Appl.}, 417(2):314 -- 334, 2006.

\bibitem{Guglielmi2013exact}
N.~Guglielmi and V.Y. Protasov.
\newblock Exact computation of joint spectral characteristics of linear
  operators.
\newblock {\em Found. Comput. Math.}, 13(1):37--97, 2013.

\bibitem{Hackbusch1985multi}
W.~Hackbusch.
\newblock {\em Multi-grid methods and applications}, volume~4.
\newblock Springer Science \& Business Media, 2013.

\bibitem{Han_Jia_1998optimal}
B.~Han and R.Q. Jia.
\newblock Optimal interpolatory subdivision schemes in multidimensional spaces.
\newblock {\em SIAM J. Math. Anal.}, 36(1):105--124, 1998.

\bibitem{Hemker1990}
P.W. Hemker.
\newblock On the order of prolongations and restrictions in multigrid
  procedures.
\newblock {\em J. Comput. Appl. Math.}, 32(3):423--429, 1990.

\bibitem{Jetter2001shift}
K.~Jetter and G.~Plonka.
\newblock A survey on $l_2$-approximation orders from shift-invariant spaces.
\newblock In {\em Multivariate approximation and applications}. Cambridge
  University Press, 2001.

\bibitem{Jia1996subdivision}
R.Q. Jia.
\newblock The subdivision and transition operators associated with a refinement
  equation.
\newblock {\em Series in Approximations and Decompositions}, 8:139--154, 1996.

\bibitem{Jia_approx_prop}
R.Q. Jia.
\newblock {Approximation properties of multivariate wavelets}.
\newblock {\em J. Amer. Math. Soc.}, 67(222):647--665, 1998.

\bibitem{kobbelt1996interpolatory}
L.~Kobbelt.
\newblock {Interpolatory subdivision on open quadrilateral nets with arbitrary
  topology}.
\newblock In {\em {Computer Graphics Forum}}, volume~15, pages 409--420. Wiley
  Online Library, 1996.

\bibitem{Levin_gen_non_unif}
A.~Levin.
\newblock {Polynomial generation and quasi-interpolation in stationary
  non-uniform subdivision}.
\newblock {\em Comput. Aided Geom. Design}, 20(1):41--60, 2003.

\bibitem{Rota1960note}
G.C. Rota and W.~Strang.
\newblock A note on the joint spectral radius.
\newblock 1960.

\bibitem{RS1987}
J.W. Ruge and K.~St{\"u}ben.
\newblock {Algebraic multigrid}.
\newblock In {\em Multigrid methods}, pages 73--130. SIAM, Philadelphia, 1987.

\bibitem{Sauer2002ideal}
T.~Sauer.
\newblock Polynomial interpolation, ideals and approximation order of
  multivariate refinable functions.
\newblock {\em Proc. Amer. Math. Soc.}, 130(11):3335--3347, 2002.

\bibitem{sauer2004lagrange}
T.~Sauer.
\newblock Lagrange interpolation on subgrids of tensor product grids.
\newblock {\em Math. Comp.}, 73(245):181--190, 2004.

\bibitem{Serra1999positive}
S.~Serra-Capizzano and C.~Tablino-Possio.
\newblock Positive representation formulas for finite difference
  discretizations of (elliptic) second order pdes.
\newblock {\em Contemp. Math.}, 281:295--317, 1999.

\bibitem{Huckle2002}
J.~Staudacher and T.~Huckle.
\newblock Multigrid preconditioning and toeplitz matrices.
\newblock {\em Electron. Trans. Numer. Anal.}, 13:81--105, 2002.

\bibitem{lemma_smoother}
H.~Sun, R.H. Chan, and Q.S. Chang.
\newblock {A note on the convergence of the two-grid method for Toeplitz
  systems}.
\newblock {\em Comput. Math. Appl.}, 34(1):11--18, 1997.

\bibitem{mgm-book}
U.~Trottenberg, C.W. Oosterlee, and A.~Sch{\"u}ller.
\newblock {\em {Multigrid}}.
\newblock Academic Press, Inc., San Diego, CA, 2001.
\newblock With contributions by A. Brandt, P. Oswald and K. St{\"u}ben.

\bibitem{Tyrthyshnikov1996}
E.E. Tyrtyshnikov.
\newblock A unifying approach to some old and new theorems on distribution and
  clustering.
\newblock {\em Linear Algebra Appl.}, 232:1 -- 43, 1996.

\bibitem{Wesseling1987linear}
P.~Wesseling.
\newblock Linear multigrid methods.
\newblock {\em Multigrid Methods}, pages 57--72, 1987.

\end{thebibliography}

\end{document}